\documentclass{tran-l}

\newtheorem{theorem}{Theorem}[section]
\newtheorem{lemma}[theorem]{Lemma}
\newtheorem{proposition}[theorem]{Proposition}

\theoremstyle{definition}
\newtheorem{definition}[theorem]{Definition}

\theoremstyle{remark}
\newtheorem{remark}[theorem]{Remark}

\numberwithin{equation}{section}

\usepackage{mathtools,amsthm,amscd}
\usepackage{microtype}
\usepackage{hyperref}
\usepackage{xcolor}
\usepackage{subfigure} 
\usepackage{wrapfig}
\usepackage{geometry}
\usepackage{epsf}
\usepackage{galois}
\usepackage{mathrsfs} 
\usepackage{kbordermatrix}
\usepackage{colonequals}

\newcommand{\rZr}{\prescript{\r}{}Z}

\newcommand{\rr}{\prescript{\r}{}}
\newcommand{\bb}{\prescript{\b}{}}

\newcommand{\C}{\mathbb{C}}

\newcommand{\1}{\mathbf{1}}
\newcommand{\D}{\mathbb{D}}
\newcommand{\E}{\mathbb{E}}
\newcommand{\N}{\mathbb{N}}
\newcommand{\Q}{\mathbb{Q}}
\newcommand{\Z}{\mathbb{Z}}
\newcommand{\R}{\mathbb{R}}
\renewcommand{\P}{\mathbb{P}}

\newcommand{\ol}{\overline}
\newcommand{\el}{l} % this is the symbol we're using for inf-based
\newcommand{\wt}{\widetilde}

\newcommand{\eps}{\varepsilon}
\def\P{\mathbb{P}}
\def\E{\mathbb{E}}
\DeclareMathOperator{\Cov}{Cov}
\DeclareMathOperator{\SLE}{SLE}

\DeclareMathOperator{\Var}{Var}

\def\fh{\mathfrak{h}}
\def\cZ{\mathcal{Z}}
\def\scZ{\mathscr{Z}}
\def\scL{\mathscr{L}}
\def\scR{\mathscr{R}}
\def\scX{\mathscr{X}}
\def\scW{\mathscr{W}}

\def\cW{\mathcal{W}}

\def\cT{\mathcal{T}}
\def\cS{\mathcal{S}}
\def\cR{\mathcal{R}}

\def\cP{\mathcal{P}}
\def\cO{\mathcal{O}}

\def\cM{\mathcal{M}}
\def\cL{\mathcal{L}}
\def\cK{\mathcal{K}}
\def\cJ{\mathcal{J}}

\def\cG{\mathcal{G}}
\def\cF{\mathcal{F}}

\def\cC{\mathcal{C}}
\def\cB{\mathcal{B}}
\def\cA{\mathcal{A}}
\def\b{\mathrm{b}}
\def\r{\mathrm{r}}
\def\g{\mathrm{g}}

\def\Zb{Z^{\mathrm{b}}}
\def\Zr{Z^{\mathrm{r}}}
\def\Zg{Z^{\mathrm{g}}}

\newcommand{\Ar}{A_{\mathrm{r}}}
\newcommand{\Ab}{A_{\mathrm{b}}}
\newcommand{\Ag}{A_{\mathrm{g}}}
\newcommand{\Geo}{\mathrm{Geom}(\frac12)}
\newcommand\ta{\mathfrak{t}}
\newcommand\fraketa{\mathfrak{y}}
\definecolor{gold}{rgb}{0.7,0.7,0}
\definecolor{dpurple}{rgb}{0.5,0,0.5}

	\newcommand{\bZb}{\prescript{\mathrm{b}}{}Z}
	\newcommand{\bRb}{\prescript{\mathrm{b}}{}R}
	\newcommand{\bLb}{\prescript{\mathrm{b}}{}L}
		\newcommand{\dsigma}{\overline{\sigma}}
\renewcommand{\arraystretch}{1.3}

\newcommand{\xin}[1]{{#1}}
\newcommand{\nocolor}[1]{{\color{black} #1}}%red stuff in the version submitted to journal in July 2022

\begin{document}

\title[Schnyder woods, SLE$_{16}$ and LQG]{Schnyder woods, SLE$_{16}$, and Liouville quantum gravity}


%    author one information
\author{Yiting Li}
\address{Korea Advanced Institute of Science and Technology}
\curraddr{}
\email{yitingli@kaist.ac.kr}
\thanks{}

%    author two information
\author{Xin Sun}
\address{University of Pennsylvania}
\curraddr{}
\email{xinsun@sas.upenn.edu}
\thanks{Xin Sun was partially supported by Simons Society of Fellows, and by NSF Award DMS-1811092 and the Career award 2046514.}

\author{Samuel S. Watson}
\address{Brown University}
\curraddr{}
\email{sswatson@brown.edu}
\thanks{}

%    \subjclass is required.

\date{}

\dedicatory{}

%    Abstract is required.
\begin{abstract}
		In 1990, Schnyder used a 3-spanning-tree decomposition of a simple
		triangulation, now known as the \textit{Schnyder wood}, to give a
		fundamental grid-embedding algorithm for planar maps. 
		\xin{In the framework of mating of trees, a
			uniformly sampled Schnyder-wood-decorated triangulation can produce a triple of random walks. We show that these three walks converge in the scaling limit to three Brownian motions  produced in the mating-of-trees framework by \nocolor{Liouville quantum gravity (LQG)}
			with parameter $1$, decorated with a triple of SLE$_{16}$'s curves. These three  SLE$_{16}$'s curves are coupled such that the angle
			difference between them is $2\pi/3$ in imaginary geometry.}
		Our convergence result provides a description of the continuum limit of Schnyder's
		embedding algorithm via LQG and \nocolor{SLE}.
\end{abstract}

\maketitle

	\section{Introduction}\label{sec:intro}
	
	A \textit{planar map} is an embedding of a connected  planar graph in the plane, considered modulo orientation-preserving homeomorphism.  A planar map is said to be \textit{simple} if it does not have any edges from a vertex to itself or multiple edges between the same pair of vertices.  We say that a planar map is a \textit{triangulation} if every face is bounded by three edges.  A \textit{spanning tree} of a graph $G$ is a connected, cycle-free subgraph of $G$ including all vertices of $G$. In this paper we will study \textbf{wooded triangulations}, which are simple triangulations equipped with a certain 3-spanning-tree decomposition known as a \textbf{Schnyder wood}.
	
	The Schnyder wood, also referred to as a \textit{graph realizer} in
	the literature, was introduced by Walter Schnyder
	\cite{schnyder1989planar} to prove a characterization of graph
	planarity. He later used Schnyder woods to describe an algorithm for
	embedding an order-$n$ planar graph in such a way that its edges are
	straight lines and its vertices lie on an $(n-2)\times (n-2)$ grid
	\cite{Schnyder}. Schnyder's celebrated construction has continued to
	play an important role in graph theory and enumerative combinatorics
	\cite{bernardi2007catalan,felsner2008schnyder,miracle2016sampling}.
	
  In the present article, we are primarily interested in
		the $n\to\infty$ behavior of wooded triangulations and their
		Schnyder embeddings. We discover an encoding of the wooded
		triangulation via a triple of random walks, where the coordinates of
		vertices under the Schnyder embedding are natural observables of the
		random walks. As we will explain in
		Sections~\ref{subsec:rpm}--\ref{subsec:1tree}, each of these three
		random walks fits into the mating-of-trees framework developed in
		\cite{burger,mating}. By simply applying the invariance principle
		for a single random walk, we may obtain a scaling limit result
		(Theorem~\ref{thm:3pair}) for wooded triangulation in the
		peanosphere sense. However, to understand the large-scale behavior
		of the Schnyder embedding, we need to study the interaction between
		the three spanning trees. This relies on another crucial observation
		we make in this paper: the coupling of the three trees is analogous
		to the coupling of a certain triple of continuum trees formed by
		imaginary geometry flow lines. The technical bulk of this paper
		(Section~\ref{sec:peano}) is to develop this insight and ultimately
		establish a scaling limit result (Theorem~\ref{thm:embedding}) for
		the image of a typical wooded triangulation under the Schnyder
		embedding.
	
	\subsection{ Wooded triangulations}\label{subsec:Schnyder}
	\begin{figure} 
		\centering
		\subfigure[]{\includegraphics[width=27mm]{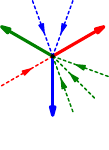}}
		\subfigure[]{\includegraphics[width=6cm]{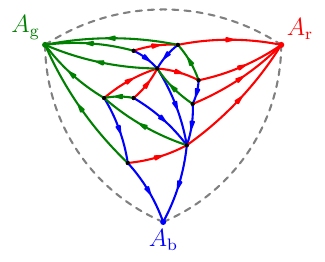}}
		\subfigure[]{\includegraphics[width=5cm]{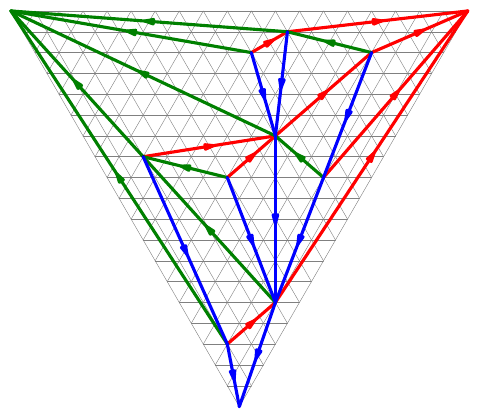}}
		\caption{ \label{fig:swgrid}
			(a) Schnyder's coloring rule: all incoming
			edges of a given color appear between the outgoing edges of
			the other two colors. We draw the incoming arrows as dashed to indicate that the number of incoming edges of a
			given color may be zero. 
			(b) The edges of a planar triangulation can be decomposed into three spanning trees satisfying Schnyder's coloring rule, and (c) Schnyder's algorithm uses this triple of trees to output a grid embedding of the triangulation. 
		}
	\end{figure}

	Consider a simple plane triangulation $M$ with unbounded face $\triangle \Ab \Ag \Ar$ where the \emph{outer vertices} $\Ab$, \nocolor{$\Ag$, and $\Ar$}---which we think of as blue, \nocolor{green, and red}---are arranged in \nocolor{clockwise} order. \nocolor{Here simple means with no multiple edges and self-loops.} We denote by $\overline{\Ab\Ag}$ the edge connecting $\Ab$ and $\Ag$, and similarly for $\overline{\Ag\Ar}$ and $\overline{\Ar\Ab}$.  Vertices, edges, and faces of $M$ other than $A_\b,A_\g,A_\r,\overline{\Ab\Ag},\overline{\Ag\Ar},\overline{\Ar\Ab}$ and $\triangle \Ab\Ag\Ar$ are called \textit{inner} vertices, edges, and faces respectively. We define the \textbf{size} of $M$ to be the number of interior vertices. Euler's formula implies that a simple plane triangulation of size $n$ has $n+3$ vertices, $3n+3$ edges, and $2n+2$ faces.
	
	An \textit{orientation} on $M$ is a choice of direction for every
	inner edge of $M$, and a \textit{3-orientation} on $M$ is an
	orientation for which every inner vertex has out-degree $3$.   \nocolor{Each simple  triangulation admits at least one 3-orientation~\cite{Schnyder}.}	A
	coloring of the  \nocolor{inner} edges of a 3-orientation with the colors blue, green,
	and red is said to satisfy \textit{Schnyder's rule} (see
	Figure~\ref{fig:swgrid}(a)) if (i) the three edges exiting each
	interior vertex are colored in the clockwise cyclic order
	blue-green-red, and (ii) each blue edge $e$ which enters an interior
	vertex $v$ does so between $v$'s red and green outgoing
	edges, and similarly for the other incoming edges. In other words, incoming
	red edges at a given vertex (if there are any) must appear between green and blue outgoing edges, and incoming
	green edges appear between blue and red outgoing edges. As demonstrated in 
\nocolor{As shown in~\cite{Schnyder}, given 3-orientation $\cO$ on a simple triangulation $M$, there is a unique way of coloring the inner edges such that Schnyder's rule is satisfied.}  We describe this coloring as an algorithm which we  call COLOR.
	\begin{enumerate}
		\item Color $\Ab,\Ag,\Ar$ blue, green, and red.
		\item \nocolor{For each} inner edge $e$,  construct a directed path
		$\cP=[e_1,e_2,\cdots e_\ell]$ inductively as follows. Set $e_1=e$.  For $k \geq 1$: if the head of $e_k$ is an outer vertex, set $\ell = k$ and stop. Otherwise, let $e_{k+1}$ be the second outgoing edge encountered when clockwise (or equivalently, counterclockwise) rotating $e_k$ about its head. \nocolor{(Here the head vertex  and tail vertex of an orientated edge are such that the orientation goes from the tail to the head.)} This procedure always yields a finite path \nocolor{without cycle. To see this, note that if $\cP$ has a cycle   $\gamma$ of length $m$, then the planar map $M'$ consisting of the faces of $M$ enclosed by $\gamma$ (together with a single unbounded face) would have $m+v$ vertices  where $v$ is the number of vertices surrounded by $\gamma$. Moreover, $M'$ has  $E=2m+3v$ edges since each vertex on $\gamma$ contributes two outgoing edges, and each vertex not on $\gamma$ contribute three. By Euler's formula $M$ would have $F=2v+m-2$ faces. Since $M'$ is a triangulation except one face,  we have $m+3(F-1)=2E$, a contradiction.}  
		\item Assign to $e$ the color of the outer vertex at which $\cP$ terminates.
	\end{enumerate} 
\nocolor{We now summarize a few properties of COLOR that are essentially from~\cite{Schnyder}; also see the notes~\cite{LS-Schnyder}:}
	\begin{enumerate}
		\item Edges on a path $\cP$ have the same color.
		\item Given an inner vertex $v$ with outgoing edges $e_1,e_2,e_3$, the three paths starting from $e_1,e_2,e_3$ are all simple paths, \nocolor{namely without cycles}. 
		
		\item Given an inner vertex $v$ with outgoing edges $e_1,e_2,e_3$, the
		three paths starting from $e_1,e_2,e_3$ do not intersect except at
		$v$. This may be proved similarly to \#1. 
		\item Since the three paths emanating from $v$ are simple and
		non-intersecting, the three paths must terminate at 
		distinct outer vertices. Therefore, their colors are distinct and
		appear in the same cyclic order as the exterior vertices.
		
		\item As a consequence of \#1 and \#4, 
		Schnyder's coloring rule is satisfied. \nocolor{Indeed, if $v$  is an inner vertex and $e$ is an oriented edge in between the green and red outgoing edges from $v$, then the blue outgoing edge from $v$ must be the next edge on the path  starting from $e$. Therefore $e$ itself is  blue. }
		
		\item The set of all blue edges forms a spanning tree of $M \backslash \{A_\g,A_\r\}$, and
		similarly for red and green. These three  trees form
		a partition of the set of $M$'s interior edges. 
	\end{enumerate}
	
	\begin{definition}
		Given a simple plane triangulation  $M$ \nocolor{equipped with a 3-orientation}, we call a coloring of the interior
		edges satisfying Schnyder's rule a  {\bf Schnyder wood} on $M$.   We denote by $\cS_n$ the set of pairs $(M,\cO)$ where $M$ is a triangulation  of size $n$ \xin{equipped with a 3-orientation} and $\cO$ is a Schnyder wood on $M$. We will refer to elements of $\bigcup_{n\ge1}\cS_n$ as \textit{Schnyder-wood-decorated triangulations}, or  {\bf wooded triangulations} for short. 
	\end{definition}
	The COLOR algorithm exhibits a natural bijection between Schnyder woods on $M$ and 3-orientations on $M$. Accordingly, we will treat Schnyder woods and 3-orientations as interchangeable. For a  wooded triangulation $S$, denote by $T_c(S)$ the tree of color $c$ on $S$, for $c \in \{\mathrm{b,r,g}\}$.  We call the unique blue oriented path from a vertex $v$ the {\bf blue flow
		line} from $v$, and similarly for red and green. 
	
	\subsection{SLE, Liouville quantum gravity, and random planar maps}\label{subsec:rpm}
	The Schramm Loewner evolution with a parameter \nocolor{$\kappa >0$} (abbreviated to $\SLE_\kappa$) is a well-known family of random planar curves discovered by Schramm \cite{Schramm}. These curves enjoy conformal symmetries and a natural Markov property which establish SLE as a canonical family of non-self-crossing planar curves. 
	\nocolor{For a  large class of  2D statistical physics models such as percolation and Ising model,  it is proved or conjectured that their scaling limit at criticality are described by  $\SLE_\kappa$ curve  with various  $\kappa$; see e.g.\  \cite{percolation,lawler2011conformal,SAW,schramm2009contour,FK,chelkak2014convergence}.}
	
	Liouville quantum gravity with parameter $\gamma\in (0,2)$ (abbreviated to $\gamma$-LQG) is a family of random planar geometries formally corresponding to $e^{\gamma h} \, dx\otimes dy$, where $h$ is a 2D random generalized function called \nocolor{the} Gaussian free field (GFF) and $\gamma$ indexes the roughness of the geometry. Rooted in theoretical physics, LQG is closely related to conformal field theory and string theory \cite{Polyakov}. It is also an important tool for studying planar fractals through the Knizhnik-Polyakov-Zamolodchikov relation (see \cite{KPZ} and references therein).  Most importantly for our purposes, LQG describes the scaling limit of decorated random planar maps. Let's discuss this perspective in the context of a classical example called \nocolor{the} uniform-spanning-tree-decorated random planar map.

	\xin{For each $\gamma\in (0,2)$, there exists a canonical random surface of spherical topology whose geometry is given by $\gamma$-LQG. The surface is called the \emph{unit-area $\gamma$-LQG sphere}.  Its area measure is of the form $\mu_\fh= e^{\sqrt{2}\fh}\,dx\,dy$ where $\fh$ is a particular variant of Gaussian free field. Due to the roughness of $\fh$, the rigorous construction of $\mu_\fh$ requires a regularization and normalization procedure called the Gaussian multiplicative chaos (GMC). We refer to the book~\cite{berestycki-powell-notes} for the background on LQG, GFF and GMC.   More details for  $\fh$ and $\mu_\fh$ in will be provided in  Section~\ref{subsec:lqg}. We now explain how the  unit-area $\sqrt2$-LQG sphere is related to the scaling limit of the uniform-spanning-tree-decorated random planar map.}

	Let $\cM_n$ be the set of pairs  $(M,T)$ such that $M$ is an $n$-edge planar map and $T$ is a spanning tree on $M$. Let $(M_n,T_n)$ be a uniform sample from $\cM_n$ and conformally embed $M_n$ in $\C$, for example via circle packing or Riemann uniformization (\nocolor{see e.g.~\cite[Section 2]{ghs-survey}}). Define an atomic measure on $\C$ by associating a unit mass with each vertex in this embedding of $M_n$. It is conjectured that this measure, suitably renormalized, converges as $n\to\infty$ to a unit-mass random area measure $\mu_\fh$ on $\C$, \xin{which is the aforementioned area measure for the  unit-area $\sqrt2$-LQG sphere.}  
	Let $\eta_n$ be the Peano curve which snakes between $T_n$ and its dual. It is conjectured that $\eta_n$ converges to an $\SLE_8$ that is independent of $\mu_{\mathfrak{h}}$. (Convergence to $\SLE_8$ of the uniform spanning tree on the square lattice is proved in \cite{lawler2011conformal}).
	
	One can replace the spanning tree in the above discussion with other structures, such as percolation, Ising model, or random cluster models. Under certain conformal embeddings, \nocolor{random planar maps decorated with models} are believed to converge to unit area $\gamma$-LQG sphere\nocolor{s} decorated with $\SLE_{\gamma^2}$ and $\SLE_{16/\gamma^2}$ curves, for \nocolor{some $\gamma\in(0,2)$}.  As metric spaces, they are believed to converge to a random metric space that formally has $e^{\gamma \fh}dx\otimes dy$ as its metric tensor.  \nocolor{The rigorous construction of the random metric was recently achieved by~\cite{DDDF-tightness,gm-uniqueness} using a regularization and normalization procedure.}
	Both of the two types of convergence remain unproved except for the percolation-decorated random planar map, in which case $M_n$ is uniformly distributed. LeGall \cite{le2013uniqueness} and Miermont \cite{miermont2013brownian} independently proved that the metric scaling limit is a random metric space called the Brownian map. Miller and Sheffield \cite{BM1,BM2,BM3} rigorously established an identification of the Brownian map and $\sqrt{8/3}$-LQG.  \nocolor{Recently in~\cite{hs-cardy-embedding} Holden and the second named author of this paper proved that uniform triangulation under a certain discrete conformal embedding  converges to $\sqrt{8/3}$-LQG. Gwynne, Miller and Sheffield~\cite{gms-invariance,gms-crt,gms-poisson-voronoi} showed that certain random planar map models obtained from coarse-graining the continuum LQG converge to the LQG under another discrete conformal embedding called the Tutte embedding. Here we remark that the Schnyder embedding considered in this paper is \emph{not} a discrete conformal embedding.}

 	\nocolor{Duplantier, Miller, and Sheffield \cite{mating,burger}  developed a powerful approach to LQG  which is particularly suited to build connections between random planar map and LQG.}   The starting point is a bijection due to Mullin \cite{mullin1967enumeration} and Bernardi \cite{Bernardi} between the set $\mathcal{M}_n$  and the set of lattice walks on $\Z^2_{\ge 0}$ with $2n$ steps and starting and ending at the origin. \nocolor{Here a lattice walk means a possible trajectory of a random walk on $\Z^2$.} The walk corresponding to $(M,T)$ is obtained by keeping track of the graph distances in $T$ and its dual $\tilde{T}$ from the tip of the exploration curve $\eta_n$ to two specified roots. This bijection \xin{is an example of a family of bijections that are now known as \textit{mating-of-trees}  bijections}, because the exploration curve can be thought of as stitching the two trees $T$ and $\tilde{T}$ together. 
	
	This mating-of-trees story can be carried out in the continuum as well, with LQG playing the role of the planar map and $\SLE$ the role of the Peano curve.  %More details can be found in Section~\ref{sec:lqg}; for now we just give the idea.
	 Suppose $(\mu_\fh,\eta')$ is the conjectured scaling limit of \nocolor{a certain decorated} random planar map under conformal embedding for some $\gamma\in (0,2)$, as described above. 
Namely, $\mu_\fh$ is area measure of the unit-area $\gamma$-LQG sphere, and \nocolor{$\eta'$} is a variant of SLE$_\kappa$ with $\kappa=16/\gamma^2$ \nocolor{which is independent of $\fh$}.  \nocolor{As explained in~\cite{mating}, the curve $\eta'$ can be parametrized} so that $\eta'$ is a continuous space-filling curve from $[0,1]$ to $\C\cup\{\infty\}$, $\eta'(0)=\eta'(1)=\infty$ and $\mu_\fh(\eta'([s,t]))=t-s$ for $0<s<t<1$. By \cite{Zipper,mating}, one can define lengths $\scL_t$  and $\scR_t$ for the left and right boundaries of $\eta'{[0,t]}$ with respect to $\mu_\fh$. The following \textbf{mating-of-trees theorem} is proved in \cite{mating,mating2,gwynne2015brownian}.
	\begin{theorem}\label{thm:mating}
		In the above setting, the law of $(\scZ_t)_{t\in[0,1]}=(\scL_t,\scR_t)_{t\in[0,1]}$ can be sampled as follows. First sample a two-dimensional Brownian motion $\scZ=(\scL,\scR)$ with
		\begin{equation}\label{eq:covariance}
		\Var[\scL_t]=\Var[\scR_t]=t \quad \textrm{and} \quad \Cov[\scL_t,\scR_t] = -\cos\left(\tfrac{\pi \gamma^2}{4}\right) t.
		\end{equation}
		Then condition $\scZ|_{[0,1]}$ on the event that both $\scL$ and $\scR$ stay positive in $(0,1)$ and $\scL_1=\scR_1=0$. Moreover, $\scZ$ determines $(\mu_\fh,\eta')$ a.s.
	\end{theorem}
	
	The singular conditioning referred to in this theorem statement can be made rigorous by a limiting procedure \cite[Section 3]{mating2}.
		In light of the Mullin-Bernardi bijection and Theorem \ref{thm:mating}, \nocolor{the convergence of the 2D lattice walks to the 2D Brownian motion \xin{(namely the invariance principle)} can be viewed as a convergence of the spanning-tree-decorated random planar map $(M_n,T_n)$  to $\sqrt2$-LQG decorated with an independent $\SLE_8$. We say that this convergence is  in the \emph{peanosphere} sense.} 
		The same type of convergence has been established for several models; see \cite{burger,burger1,burger2,burger3,bipolar,bipolarII,active,Euclidean}.  In many cases, the topology of convergence can be further improved by using special properties of the model of interest \cite{burger1,burger3,bipolarII}. \nocolor{In particular, results in~\cite{BHS19}  are crucial to the aforementioned convergence in~\cite{hs-cardy-embedding}  under conformal embedding. See~\cite{ghs-survey} for a survey on the mating of trees theory and its application to random planar maps.} 
	\nocolor{We will provide more background on SLE, GFF, and LQG in Section~\ref{sec:lqg}. But we emphasize that the mating of trees framework allows us to work mostly  on the Brownian motion side.  Thus most of the paper does not require this  background.}

	\subsection{A Schnyder wood mating-of-trees encoding}\label{subsec:bijection}
	We consider a uniform sample $(M_n,\cO_n)$ from $\cS_n$, which we call a \textbf{uniform wooded triangulation} of size $n$. We view $(M_n,\cO_n)$ as a decorated random planar map. Under the marginal law of $M_n$, the probability of each triangulation is proportional to the number of Schnyder woods it admits. Conditioned on $M_n$, the law of $S_n$ is uniform on the set of Schnyder woods on $M_n$.  We are interested in the scaling limit of $(M_n,\cO_n)$ as $n\to\infty$.
	
	Our starting point is a bijection similar to the Mullin-Bernardi bijection for spanning-tree-decorated planar maps in Section~\ref{subsec:rpm}. 
	Suppose $S\in \cS_n$. We define a path $\cP_S$ (see
	Figure~\ref{fig:schnyderwood}(b)) which starts in the outer face,
	enters an inner face through $\overline{\Ar\Ag}$, crosses all edges
	incident to $\Ag$, enters the outer face through $\overline{\Ag\Ab}$,
	enters an inner face through $\overline{\Ag\Ab}$, explores
	$T_{\mathrm{b}}(S)$ clockwise, enters the outer face through
	$\overline{\Ab\Ar}$, enters an inner face through $\overline{\Ab\Ar}$,
	crosses the edges incident to $\Ar$, enters the outer face through
	$\overline{\Ar\Ag}$, and returns at the starting point. The path
	$\cP_S$ crosses $\overline{\Ar\Ag}$, $\overline{\Ag \Ab}$ and
	$\overline{\Ab \Ar}$ and each red and green edge twice and traverses
	each blue edge twice. \nocolor{\xin{Namely}, we can view the path 	$\cP_S$ as an ordering of the edge set of $S$ where each inner edge is visited twice. }
	
	Define $\varphi(S) = Z^\b$ to be a walk on $\mathbb{Z}^2$ as follows. The walk
	starts at $(0,0)$. When $\cP_S$ traverses a blue edge for the second
	time, $Z^\b$ takes a $(1,-1)$-step. When $\cP_S$ crosses a red edge
	for the second time, $Z^\b$ takes a $(-1,0)$-step. When $\cP_S$
	crosses a green edge for the second time, $Z^\b$ takes a
	$(0,1)$-step. See Figure~\ref{fig:wordwalkmap} for an example of a pair $(S,\varphi(S))$.  
	\begin{figure} 
		\centering
		\subfigure[]{\includegraphics[width=6cm]{figures/woodonly}}
		\subfigure[]{\includegraphics[width=6cm]{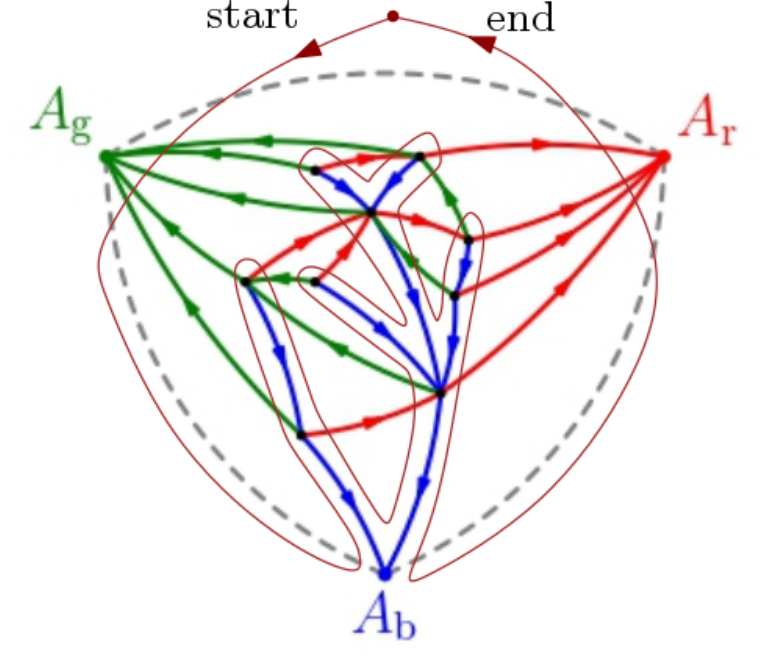}}
		\caption{ \label{fig:schnyderwood}
			(a) a triangulation equipped with a Schnyder wood, and (b) the path $\cP_S$, which traces out the blue
			tree clockwise.
		}
	\end{figure}
	\begin{definition}\label{def:walk}
		Define $\mathcal{W}_n$ to be the set of walks on $\mathbb{Z}^2$ satisfying the
		following conditions.
		\begin{enumerate}
			\item The walk starts and ends at $(0,0)$ and stays in the closed first
			quadrant.
			\item The walk has $3n$ steps, of three types: $(0,1)$, $(-1,0)$ and
			$(1,-1)$. 
			\item No $(1,-1)$-step is immediately preceded by a
			$(-1,0)$-step. 
		\end{enumerate}
	\end{definition}
	\begin{theorem}\label{thm:bijection from wood to walk}
		We have $\varphi(S) \in \cW_n$ for all $S\in \cS_n$, and $\varphi:\cS_n \to \cW_n$ is a bijection.  
	\end{theorem} 
	By symmetry, we may also define $Z^\r$ and $Z^\g$ similarly to $Z^\b$, based on clockwise explorations of $T_\r$ and $T_\g$, respectively. Theorem~\ref{thm:bijection from wood to walk} implies that the $Z^\r$ and $Z^\g$ encodings are also bijective. \nocolor{We will prove Theorem~\ref{thm:bijection from wood to walk} in Section~\ref{subsec:tree-walk}.}
	
	\subsection{SLE$_{16}$-decorated 1-LQG as the scaling limit}\label{subsec:1tree}
	\xin{We now describe a conjectural scaling limit of the uniform-wooded triangulation in terms of LQG surface decorated by SLE curves. Our Theorem~\ref{thm:3pair} is a verification of this conjecture in the mating of trees framework.  Since there are three spanning trees, we need three space-filling SLE curves that are coupled together. Imaginary geometry is a framework that produces such couplings, which is developed in~\cite{IGI,miller2013imaginary}. Fix $\kappa>4$, in this framework, there is a notion of angle between two different space-filling SLE$_\kappa$. We will provide more details of imaginary geometry and the coupling between different SLE curves in  Section~\ref{subsec:lqg}. On the other hand, as mentioned before, we mainly focus on the Brownian motion side of the mating-of-trees story, hence details of this coupling from the SLE side is not necessary to understand our result.  
	
To state our conjecture and theorem, we consider three space-filling $\SLE_{16}$ curves coupled in imaginary geometry, where the angle difference between each other is $2\pi/3$.  We denote the three space-filling  curves by $(\eta^\b,\eta^\r,\eta^g)$ and let $\fh$ be the  field corresponding to a unit-area $1$-LQG sphere which is  independent of  $(\eta^\b,\eta^\r,\eta^\g)$.  The marginal law  $(\fh,\eta^\bullet)$ for $\bullet =\b,\r,\g$ is the same as $(\fh,\eta')$  in Theorem~\ref{thm:mating}. with $\gamma=1$.  The precise description of the joint law of $(\eta^\b,\eta^\r,\eta^g)$  will be given in Section~\ref{subsec:lqg}. } Let $S_n$ be a uniform wooded triangulation of size $n$.   We conjecture that, under a conformal embedding and a suitable normalization, the volume measure of $S_n$ converges to $\mu_{\fh}$, and that the clockwise exploration curves of the three trees in $S_n$ jointly converge to $\eta^\b,\eta^\r,$ and $\eta^\g$.
\xin{Our following theorem verifies the mating-of-trees variant of this conjecture.}

%	Given a whole plane Gaussian free field $h$, we may construct the flow
%	lines of the vector field $e^{i\left(2h/3+\theta\right)}$ using 
%	so-called \emph{imaginary geometry}. These flow lines are called the
%	angle-$\theta$ flow lines and are distributed as $\SLE_1$ curves.
%	Moreover, $h$ determines a unique Peano curve $\eta'^{\theta}$ on $\C$
%	so that the left and right boundaries of $\eta'^\theta$ are flow lines
%	of angle $\theta\pm \tfrac{\pi}{2}$. The space-filling curve
%	$\eta'^\theta$ is distributed as an $\SLE_{16}$. \nocolor{We refer to  Section~\ref{subsec:lqg} for
%	the precise construction.}
%	By analogy to the
%	discrete setting, we abbreviate
%	$(\eta'^0,\eta'^{\tfrac{2\pi}{3}},\eta'^{\tfrac{4\pi}{3}})$ to
%	$(\eta^\b,\eta^\r,\eta^\g)$.
	
%	Let $\fh$ be a unit-area $1$-LQG sphere which is independent of   $(\eta^\b,\eta^\r,\eta^g)$.
	\begin{theorem}\label{thm:3pair}
		Suppose $(\Zb, \Zr, \Zg)$ is the triple of random walks encoding the three trees in a uniformly sampled wooded triangulation of size $n$. For $\mathrm{c}\in\{\b,\r,\g\}$, let $\scZ^c$ be the Brownian excursion associated with  $(\eta^{c},\mu_\fh)$ \xin{as in Theorem~\ref{thm:mating}.  Namely  $(\scZ^c,\mu_\fh,\eta^{c})$ has the same law as $(\scZ,\mu_\fh,\eta')$ in Theorem~\ref{thm:main1}.} Write $Z^{\mathrm{c}}=(L^{\mathrm{c}},R^{\mathrm{c}})$. Then
		\begin{equation}\label{eq:3pair}
		\left(\tfrac{1}{\sqrt{4n}}L^{\mathrm c}_{\lfloor 3nt\rfloor}, 
		\tfrac{1}{\sqrt{2n}}R^{\mathrm c}_{\lfloor 3nt\rfloor}\right)_{t\in [0,1]} 
		\stackrel{\text{in law}}{\longrightarrow} (
		\mathscr{Z}^{\mathrm{c}})_{t\in [0,1]} \quad \textrm{for}\;\; \mathrm c\in \{\b,\r,\g\}.
		\end{equation}
		Furthermore, the three convergence statements in \eqref{eq:3pair} hold jointly. 
	\end{theorem}

	\begin{figure}
		\centering
		\includegraphics[width=0.65\textwidth]{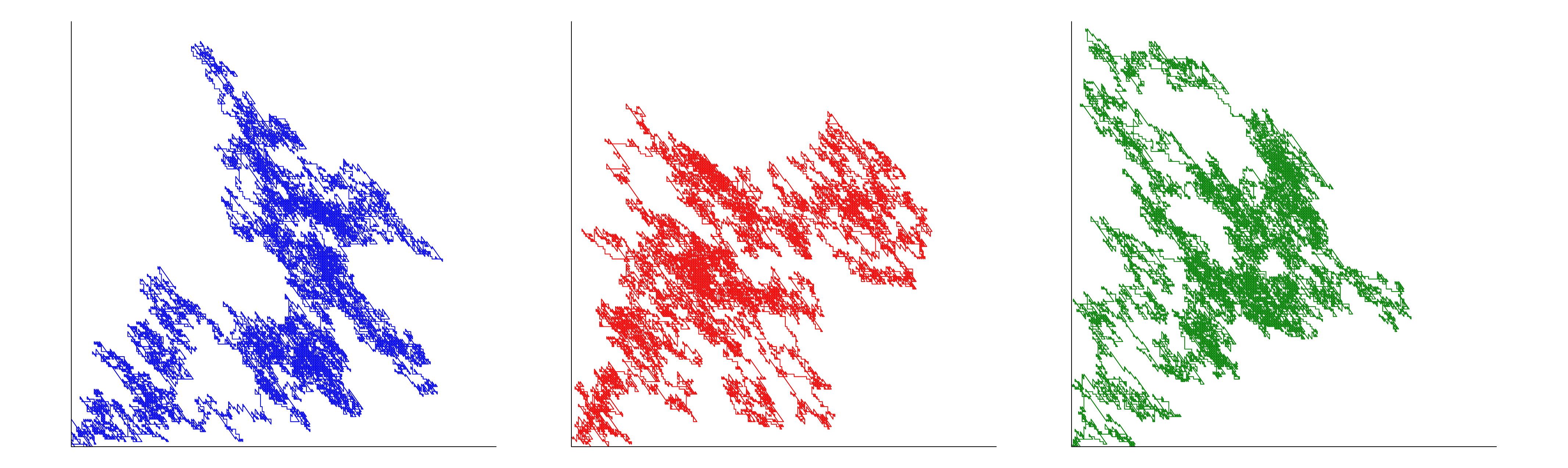}
		\caption{The random walks $Z^\b, Z^\r,$ and $Z^\g$ under the scaling in Theorem~\ref{thm:3pair}. 
			\label{fig:three_excursions}}
	\end{figure}
	
	This theorem is a natural extension of the idea that $M_n$ and \textit{one} of its trees converge in the peanosphere sense. The one-tree version boils down to a classical statement about random walk convergence. The three-tree version requires a more detailed analysis, because the triple $(\mathscr{Z}^{\mathrm{b}}, \mathscr{Z}^{\mathrm{r}}, \mathscr{Z}^{\mathrm{\g}})$ is \xin{more complicated} than a six-dimensional Brownian motion.
	
	\subsection{Schnyder's embedding and its continuum limit}\label{subsec:embed}
	It is well-known that every planar graph admits a straight-line planar embedding \cite{Fary1948,tamassia2013handbook}.
	A longstanding problem in computational geometry was to find a \textit{straight-line drawing algorithm} such that (i) every vertex has integer coordinates, (ii) every edge is drawn as a straight line, and (iii) the embedded graph occupies a region with $O(n)$ height and $O(n)$ width. This was achieved independently by de Fraysseix, Pach, and Pollack \cite{de1990draw} and by Schnyder \cite{Schnyder} via different methods. Schnyder's algorithm is an elegant application of the Schnyder wood: 
	\begin{enumerate}
		\item Given an arbitrary planar graph $G$, there exists a simple maximal planar supergraph $M$ of $G$ (in other words, a simple triangulation $M$ of which $G$ is a subgraph---note that $M$ is not unique). Such a triangulation $M$ can be identified in linear (that is, $O(n)$) time and with $O(n)$ faces, so the problem is reduced to the case of simple triangulations. 
		\item The triangulation $M$ admits at least one Schnyder wood structure, and one can be found in linear time. So we may assume $M$ is equipped with a Schnyder wood.  
		\item Each vertex $v$ in $M$, the blue, red, green flow lines from $v$ partition $M$ into three regions. Let $x(v)$, $y(v)$, and $z(v)$ be the number of faces in these three regions, as shown in Figure~\ref{fig:schnyderembedding}(a), and define $s(v) = (x(v),y(v),z(v))$. It is possible, in linear time, to compute the values of $s(v)$ for \textit{all} vertices $v$ {\cite{Schnyder}}. 
		\item Since the total number of inner faces is some constant $F$, the range of the map $s(v) = (x(v),y(v),z(v))$ is contained in the intersection of $\Z^3$, the plane $x+y+z = F$, and the closed first octant. This intersection is an equilateral-triangle-shaped portion of the triangular lattice $T$ (see Figure~\ref{fig:schnyderembedding}(b)). It is proved combinatorially in \cite{Schnyder} that if we map every edge $(u,v)$ in $M$ to the line segment between $s(u)$ and $s(v)$, then we obtain a proper embedding (that is, no such line segments intersect except at common endpoints). Note that height and width of $T$ are $O(n)$. 
	\end{enumerate}
	
	\begin{figure}[h]
		\centering
		\subfigure[]{\includegraphics[width=5cm]{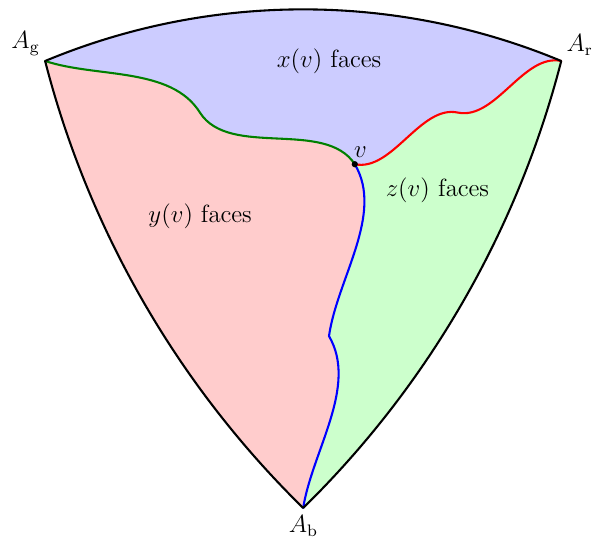}} \quad 
		\subfigure[]{\includegraphics[width=5cm]{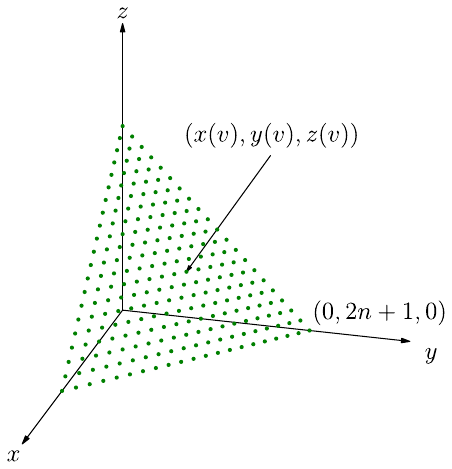}}
		\caption{Schnyder's embedding: we send each vertex $v$ in a simple plane triangulation to the triple of integers describing the number of faces in each region into which the flow lines from $v$ partition the planar map. }
		\label{fig:schnyderembedding}
	\end{figure}
	
	Schnyder's method and de Fraysseix, Pach, and Pollack's method provided the foundational ideas upon which subsequent straight-line grid embedding schemes have been built. For more discussion, we refer the reader to \cite{di2013drawing, tamassia2013handbook}.
	
	In Schnyder's algorithm, the coordinates of a vertex $v$ are determined by how flow lines from $v$ partition the faces of $M_n$. Since these ingredients have continuum analogues, Theorem~\ref{thm:3pair} suggests that Liouville quantum gravity coupled with imaginary geometry can be used to describe the large-scale random behavior of Schnyder's embedding. Consider the coupling of $(\mu_\fh,\eta^b,\eta^r,\eta^g)$ as in Theorem~\ref{thm:3pair}. Given $v\in \C$, run three flow lines from $u$, which are the right boundaries of $\eta^\b,\eta^\r,$ and $\eta^\g$ at the respective times when they first hit $u$. Map $u$ to the point in the plane $x+y+z = 1$ whose coordinates $x(v)$, $y(v)$, and $z(v)$ are given by the $\mu_\fh$-measure of the three regions into which these flow lines partition $\C$. This map is the continuum analogue of the discrete map depicted in Figure~\ref{fig:schnyderembedding}.
	
	\begin{theorem}\label{thm:embedding}
		Consider a uniform wooded triangulation of size $n$, and let $v^n_{1},\cdots, v_k^n$ be an i.i.d.\ list of uniformly chosen elements of the vertex set. Then the list of embedded locations of these vertices, namely $\{(2n)^{-1}s(v_i^n)\}_{1\le i\le k}$, converges in law as $n\to\infty$. 
		
		The limiting law is that of $\{x(v_i),y(v_i),z(v_i)\}_{1\le i\le k}$, where $v_1,\cdots,v_k$  are $k$ points uniformly and independently sampled from an instance of $\mu_\fh$.
	\end{theorem}
	\begin{figure}[h]
		\centering
		\includegraphics[width=0.60\textwidth]{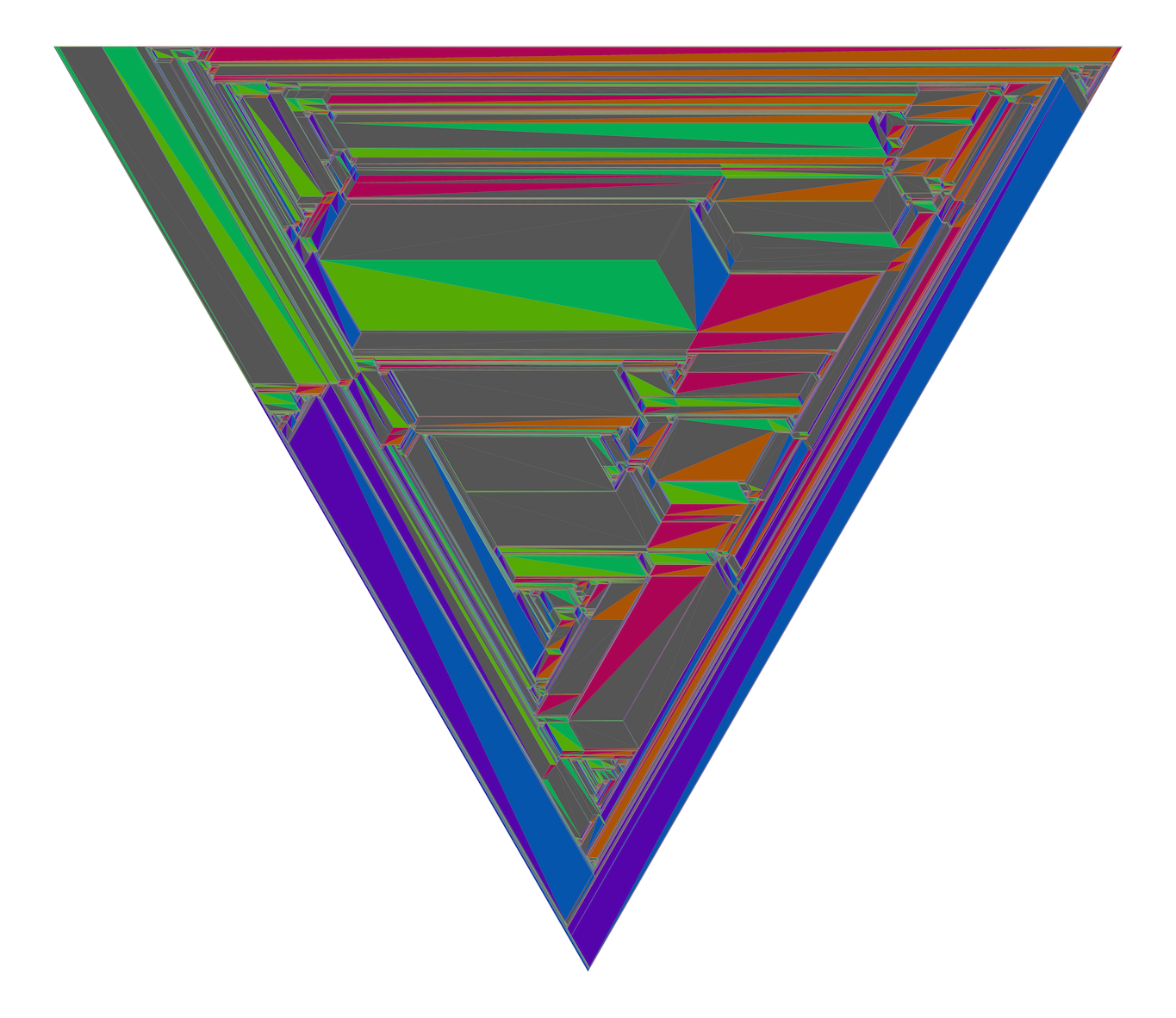}
		\caption{A uniformly sampled 1{,}000{,}000-vertex Schnyder-wood-decorated triangulation, 
			Schnyder-embedded in an equilateral triangle. Each face is
			colored with the average (in RGB color space) of the colors of its three bounding
			edges. \label{fig:sw_example}}
	\end{figure} 
	We will prove Theorem~\ref{thm:embedding} as a corollary to our proof of Theorem~\ref{thm:3pair}.
	\begin{remark} 
		In the discrete setting, there are exactly three flow lines from every vertex. In the continuum, this flow line uniqueness holds almost surely for any fixed point, but  \textit{not} for all points simultaneously. That is, there almost surely exist multiple flow lines of the same angle starting from the same point. This singular behavior is manifested in some noteworthy features of the image of a large  uniform wooded triangulation under the Schnyder embedding (Figure~\ref{fig:sw_example}). For example, there are macroscopic triangles occupying the full area of the overall triangle. These macroscopic triangles come in pairs which form parallelograms, whose sides  are parallel to the sides of the overall triangle. 
		In Theorem~\ref{thm:embedding}, we focus on typical points. A more thorough discussion of the continuum limit of the Schnyder embedding and its relation to imaginary geometry singular points will appear in \cite{sun2017scaling}.
	\end{remark}
	
	\subsection{Relation to other models and  works}\label{subsec:other}
	\subsubsection*{Bipolar orientations} 
	A bipolar orientation on a planar map is an acyclic orientation with a unique source and sink. A mating-of-trees bijection for bipolar-oriented maps was found in \cite{bipolar}, and the authors used it to prove that bipolar-oriented planar maps converge to $\SLE_{12}$ decorated $\sqrt{4/3}$-LQG, in the peanosphere sense. A bipolar orientation also induces a bipolar orientation on the dual map. In \cite{bipolar}, the authors conjectured that the bipolar-decorated map and its dual jointly converge to  $\sqrt{4/3}$-LQG decorated with two $\SLE_{12}$ curves  \nocolor{coupled in the same imaginary geometry}. 	In \cite{bipolarII}, this conjecture was proved in the triangulation case. \nocolor{The convergence is of the same type as in our Theorem~\ref{thm:3pair}. The angle between the  two  $\SLE_{12}$ curves is $\pi/2$.}
	This was the first time that a pair of imaginary-geometry-coupled Peano curves was proved to be the scaling limit of a natural discrete model. The present article provides the first example for a triple of Peano curves. Our work will make use of some results on the coupling of multiple imaginary geometry trees from \cite{bipolarII}. We will review these results in Section~\ref{subsec:singleflow}.
	
	Given $S\in \cS_n$, let $M'=T_\b(S)\cup T_\g(S)\cup \ol{A_gA_b}$, and reverse the orientation of each blue edge. It is proved in \cite[Proposition 7.1]{baxter-counting}
	that (i) this operation gives a bipolar oriented map on $n+2$ vertices with the property that the right side of every bounded face is of length two, and (ii) this procedure gives a bijection between $\cS_n$ and the set of bipolar oriented maps with that property. Applying the mating-of-trees bijection in \cite{bipolar}, it can be proved that the resulting walk coincides with the walk $\cZ^n$ in Lemma~\ref{lem:fin-walk}. Furthermore, both blue and green flow lines in $S$ can be thought of as flow lines in the bipolar orientation on the dual map of $M'$ in the sense of \cite{bipolar}. Therefore, Theorem~\ref{thm:3pair} can be formulated as a result for bipolar orientations. However, the bipolar orientation perspective clouds some combinatorial elements that are useful for the probabilistic analysis. Therefore, rather than making use of this bijection, we will carry out a self-contained development of the requisite combinatorics directly in the Schnyder woods setting.  \nocolor{See~\cite{ghs-exp}  for an application of this bipolar orientation encoding, where the authors showed that the metric exponent for uniform wooded triangulation is the same as the corresponding exponent for LQG with $\gamma=1$. }
	
	\subsubsection*{Non-intersecting Dyck paths}
	In \cite{bernardi2007catalan}, the authors give a bijection between the set of wooded triangulations and the set of pairs of non-crossing Dyck paths. This bijection implies that the number of wooded triangulations of size $n$ is equal to 
	\begin{equation} \label{eq:catalan}
	C_{n+2}C_n-C_{n+1}^2 = \frac{6(2n)!(2n+2)!}{n!(n+1)!(n+2)!(n+3)!}
	\end{equation}
	where $C_n$ is the $n$-th Catalan number
	\cite[Section 3]{bernardi2007catalan}. Our bijection is similar to the one in \cite{bernardi2007catalan} in that they are both based on a contour exploration of the blue tree. In fact, their bijection is related to ours via a shear transformation that maps the first quadrant to $\{(x,y) \in \R^2 \, : \, x \geq y \geq 0\}$. Thus Theorem~\ref{thm:3pair} implies that the three pairs of non-intersecting Dyck paths in \cite{bernardi2007catalan}---coming from clockwise exploring the three trees--- converge jointly to a shear transform of $(\scZ^\b,\scZ^\r,\scZ^\g)$ in Theorem~\ref{thm:3pair}.
	
	The non-crossing Dyck paths are closely related to the lattice structure of the set of Schnyder woods on a triangulation, while our bijection is designed to naturally encode geometric information about the wooded triangulation. 
	
	\subsubsection*{Twenty-vertex model}
	
	Note that a Schnyder wood on a regular triangular lattice has the property that each vertex has in-degree and out-degree 3. In other words, Schnyder woods are equivalent to Eulerian orientations in this case. Similarly, bipolar orientations on the \textit{square} lattice have in-degree and out-degree 2 at each vertex. In the terminology of the statistical physics literature, these are the \textit{twenty-vertex} and \textit{six-vertex} models, respectively (note that $20 = \binom{6}{3}$ and $6 = \binom{4}{2}$). The twenty-vertex model was  first studied by Baxter \cite{baxter1969}, following the suggestion of Lieb. Baxter \cite{baxter1969} showed that the residual entropy of the twenty-vertex model is $\tfrac{3\sqrt3}{2}$, generalizing Lieb's famous result for the six-vertex model.
	
	Since the twenty-vertex model is a special case of the Schnyder wood, we may define flow lines from each vertex using the COLOR algorithm. Furthermore, we may consider the dual orientation on the dual lattice---this orientation sums to zero around each hexagonal face and can therefore be integrated to give a height function associated with the model. It is an easy exercise to check that the winding change of a flow line in the twenty-vertex model can be measured  by the height difference along the flow line. This is analogous to the Temperley bijection for the dimer model, where the dimer height function measures the winding of the branches of the spanning tree generated by the dimer model. A similar flow-line height function relation for the six-vertex model has been found by \cite{bipolarIII}. In light of the correspondence between the dimer model and imaginary geometry with $\kappa=2$, we conjecture that the twenty-vertex model is similarly related to imaginary geometry with $\kappa=1$. To our knowledge, this perspective on the twenty-vertex model is new. We will not elaborate on it in this paper, but we plan to do some numerical study in a future work. See \cite{bipolarIII} for results on the numerical study of the six-vertex model in this direction.
	
\nocolor{	
 \subsubsection*{Bernardi and Fusy decomposition.} Bernardi and Fusy~\cite{Bernardi-Fusy} provides a Schnyder decomposition for   $d$-angulations of girth $d$, generalizing the bipolar orientation ($d=4$) and Schnyder woods $(d=3)$. It is interesting to understand the relation between LQG and their model for $d\ge 5$, in particular, the corresponding LQG parameter $\gamma$.
}
	
	\subsection{Outline}
	
	In Section~\ref{sec:bijection} we show that our relation between wooded triangulations and lattice walks is a bijection, and we demonstrate how this bijection operates locally. We use this construction to define an infinite-volume version of the uniform wooded triangulation, for ease of analysis. In Section~\ref{sec:pair} we prove convergence of the planar map and \textit{one} tree, and we address some technically important relationships between various lattice walk variants associated with the same wooded triangulation. In Section~\ref{sec:lqg} we review the requisite LQG and imaginary geometry material, and we present a general-purpose excursion decomposition of a two-dimensional Brownian motion that plays a key role in connecting $Z^\b$, $Z^\r$ and $Z^\g$. Finally, in Sections~\ref{subsec:1D}--\ref{sec:3tree} we prove the infinite volume version of Theorem~\ref{thm:3pair}, namely Theorem~\ref{thm:main1}. In Section~\ref{sec:finite}, we transfer Theorem~\ref{thm:main1} to the finite volume setting and conclude the proofs of Theorems~\ref{thm:3pair} and~\ref{thm:embedding}.

	\section{Schnyder woods and 2D random walks} \label{sec:bijection}
	We will prove Theorem~\ref{thm:bijection from wood to walk} in Section~\ref{subsec:tree-walk} and record some geometric observations in Section~\ref{subsec:dual}. In Section~\ref{subsec:infinite}, we construct the infinite volume limit of the uniform wooded triangulation. 
	
	\subsection{From a wooded triangulation to a lattice walk}\label{subsec:tree-walk}
	
	We first recall some basics on \textbf{rooted plane trees}.  For more background, see \cite{LeGall-Survey}. A rooted plane tree is a planar map with one face and a specified directed edge called the \textit{root edge}. The head of the root edge is called the \textit{root} of the tree. We will use the term \textit{tree} as an abbreviation for \textit{rooted plane tree} throughout.
	
	\begin{figure}[h] \centering
		\includegraphics[width=12cm]{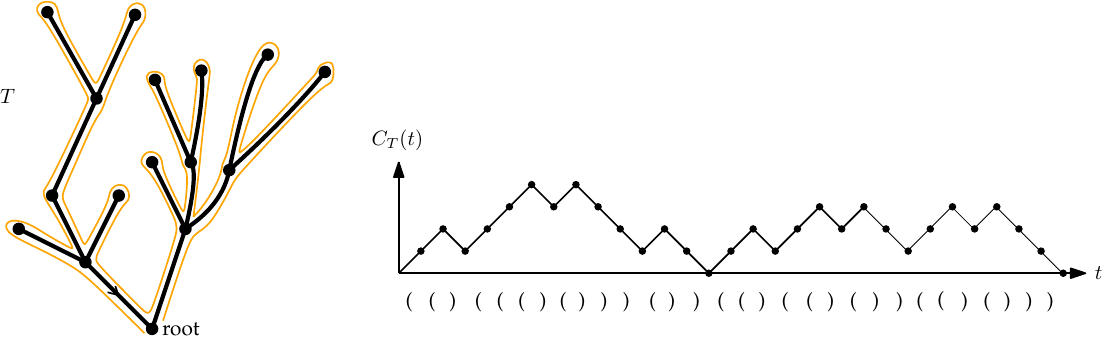}
		\caption{\label{fig:rootedplanetree} A rooted plane tree whose root
			edge is indicated with an arrow. The contour function $C_T$ (shown
			here linearly interpolated) tracks the graph distance to the root
			for a clockwise traversal of the tree. The contour function is a
			Dyck path, and if we associate to each $+1$ step an open
			parenthesis and with each $-1$ step a closed parenthesis, then we
			see that Dyck paths of length $2n$ are in bijection with
			parenthesis matchings of length $2n$. }
	\end{figure}
	
	Let \nocolor{$n$} be a positive integer, and suppose $T$ is a tree with \nocolor{$n+1$} vertices and \nocolor{$n$} edges. \nocolor{Without loss of generality, we draw $T$ on the plane such that (i) its root is below all other vertices and (ii) the root edge is to the left hand side of other edges attached to the root.}  Consider a clockwise exploration of $T$ starting from the \nocolor{ left side of the} root edge (see Figure~\ref{fig:rootedplanetree}), and define a function $C_T$ from $\{0,1,\cdots,2\nocolor{n}\}$ to $\Z$ such that $C_T(0) = 0$ and for $0 \leq t \leq 2\nocolor{n}-1$, $C_{T}(t+1) - C_T(t) = +1$ if the $(t+1)$st step of the exploration traverses its edge in the away-from-the-root direction and $-1$ otherwise. We call $C_T$ the \textbf{contour function} associated with $T$.
	
	The contour function is an example of a \textit{Dyck path} of length $2\nocolor{n}$: a nonnegative walk from $\{0,1,\ldots,2\nocolor{n}\}$ to $\Z$ with steps in $\{-1,+1\}$ which starts and ends at 0. We can linearly interpolate the graph of $C_T$ and glue the result along maximal horizontal segments lying under the graph to recover $T$ from $C_T$. Thus $T\mapsto C_T$ is a bijection from the set of rooted plane trees with \nocolor{$n$} edges to the set of Dyck paths of length $2\nocolor{n}$. 
	
	A \textit{parenthesis matching} is a word in two symbols \texttt{(} and \texttt{)} that reduces to the empty word under the relation $\texttt{()}=\emptyset$.  The gluing action mapping $C_T$ to $T$ is equivalent to parenthesis matching the steps of $C_T$, with upward steps as open parentheses and down steps as close parentheses (see Figure~\ref{fig:rootedplanetree}). Thus the set of Dyck paths of length $2n$ is also in natural bijection with the set of parenthesis matchings of length $2n$. 
	
	Suppose $S=(M,\cO)$ is a wooded triangulation, and denote by $T_{\b}(S)$ its blue tree. Among \nocolor{the} edges of $T_{\b}(S)$, we let the one immediately clockwise from $\overline{\Ag\Ab}$ \nocolor{around $\Ab$} be the root edge of $T_{\mathrm{b}}(S)$, and similarly for the red and green trees. \nocolor{Let the heads of these root edges be the outer vertices. So their orientations are consistent with the corresponding 3-orientation and} $T_{\mathrm{b}}(S)$, $T_{\mathrm{r}}(S)$, and $T_{\mathrm{g}}(S)$ are rooted plane trees.
	
	Given a lattice walk $Z$ whose increments are in $\{(1,-1),(-1,0),(0,1)\}$, we define the associated word $w = w_1w_2\cdots w_{3n}$ in the letters $\{\mathrm{b,r,g}\}$ by mapping the sequence of increments of $Z$ to its corresponding sequence of colors: for $1\leq k\leq 3n$ we define $w_k$ to be $\mathrm{b}$ if $Z_{k}-Z_{k-1} = (1,-1)$, to be r if $Z_{k}-Z_{k-1} = (-1,0)$, and to be g if $Z_{k}-Z_{k-1} = (0,1)$ (see Figure~\ref{fig:wordwalkmap}). We will often elide the distinction between a walk and its associated word, so we can say $w \in \cW_n$ if $w$ is the word associated to $Z$ and $Z \in \cW_n$. Also, we will refer to the two components of an ordered pair as its \textit{abscissa} and \textit{ordinate}, respectively.
	\begin{figure}[h]
		\centering 
		\includegraphics[width=0.9\textwidth]{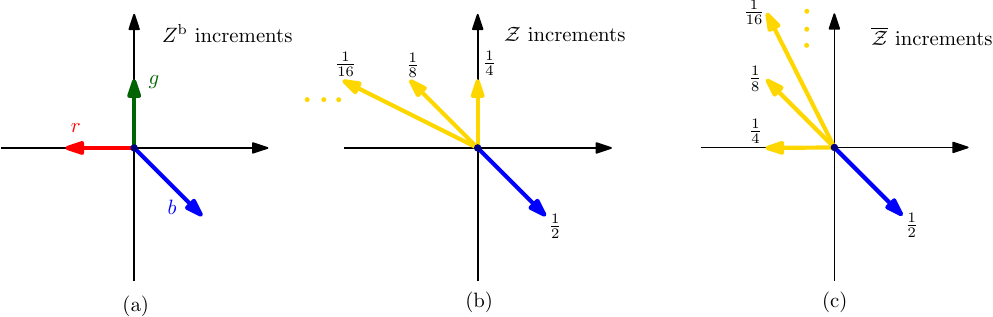}
		\caption{The set of increments of the three types of lattice walk we consider: 
			(a) $Z^\b$:  $(1,-1)$, $(0,1)$, and $(-1,0)$,
			(b) $\cZ$ (see Lemma~\ref{lem:fin-walk} and~\ref{lem:forward}): $\{(1,-1)\} \cup
			(-\Z_{\geq 0} \times \{1\})$, with probability measure indicated by
			the labels, and (c) $\overline{\cZ}$ (see Definition~\ref{def:reverse}): $\{(1,-1)\} \cup
			( \{-1\} \times \Z_{\geq 0})$.
		}
		\label{fig:steps} 
	\end{figure}
	Given a word $w$, let $w_{\g\b}$ be the sub-word obtained by dropping all  r symbols from $w$ and  $w_{\b\r}$ be the sub-word obtained by dropping all g symbols from $w$. Then $w\in \cW_n$ implies that both $w_{\g\b}$  and $w_{\b\r}$ are parenthesis matchings: 
	
	\begin{figure}[h]
		\centering 
		\includegraphics[height=4cm]{figures/woodonly}
		\includegraphics[height=4cm]{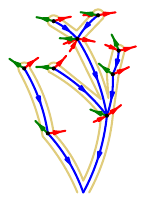}
		\caption{We cut each red and green edge into incoming and outgoing arrows at each vertex and discard the green incoming arrows. Then associating the second traversing of each blue edge with the outgoing red arrow incident to its tail, we see that $w_{\b\r}$ encodes the matching of red incoming and outgoing arrows along the contour of the blue tree. Similarly, associating each outgoing green arrow with the first traversing of the outgoing blue edge from the same vertex, we see that $w_{\b\r}$ describes the contour function of the blue tree.} 
		\label{fig:decorated} 
	\end{figure}
	
	\begin{proposition} \label{prop:paren_matching} 
		For $S\in \cS_n$, we have $w \colonequals \nocolor{\varphi}(S)\in \cW_n$. Moreover, $w_{\g\b}$ is the parenthesis matching corresponding to the contour function of $T_{\b}$, and $w_{\mathrm{br}}$ is the parenthesis matching describing $\cP_S$'s crossings of the red edges (to wit: each first crossing corresponds to an open parenthesis, and each second crossing to a close parenthesis). 
	\end{proposition}
	\begin{proof} 
From Schnyder's rule and the COLOR algorithm, $\cP_S$  crosses each green (resp.\ red) edge $e$  twice, and \nocolor{ the tail (resp. head) of $e$ is on the right} of $\cP_S$  at the second crossing. 
		
		\nocolor{ Obviously $S$ has $n$ edges in each color, so $\varphi(S)$ has $n$ steps in each type and must end at the origin. For each blue edge, the Schnyder's rule for its tail implies that the $(1,-1)$-step corresponding to this blue edge must be preceded by the $(0,1)$-step corresponding to the outgoing green edge at the tail of this blue edge. So the $\varphi(S)$ cannot go below the upper half plane. For each red edge, if its tail is $v$, then $\cP_S$ traverses the outgoing blue edge at $v$ for the second time before $\cP_S$ crosses this red edge for the second time. So the $(-1,0)$-step corresponding to this red edge must be after the $(1,-1)$-step corresponding to the outgoing blue edge at $v$, therefore $\varphi(S)$ never goes to the left half plane. In summary, $\varphi(S)\in\cW_n$.}	
		
		We remove the outer edges and cut all of the green and red edges into outgoing and incoming \textit{arrows} (a.k.a \textit{darts}). Then every g (resp.\ r) symbol in $w$ corresponds to a green outgoing (resp.\ red incoming) arrow. Finally, we remove incoming green arrows (see Figure~\ref{fig:decorated}). 
		
		The outgoing and incoming red arrows in this arrow-decorated tree form a valid parenthesis matching, by planarity of $S$. By Schnyder's rule, each b in $w$ corresponds to the blue edge whose second traversal immediately follows $\cP_S$'s crossing of some outgoing red arrow. Using this identification between outgoing red arrows and b's in $w$, we see that the sequence of b's and r's in $w_{\b\r}$ admits the same parenthesis matching as the sequence of outgoing and incoming red arrows.
		\nocolor{ So $w_{\mathrm{br}}$ is the parenthesis matching describing $\cP_S$'s crossings of the red edges.}
		 Similarly, identifying the outgoing green arrow and outgoing blue edge from each vertex $v$, Schnyder's rule implies that only incoming red arrows may appear between $\cP_S$'s first traversal of $v$'s outgoing blue edge and its crossing of $v$'s outgoing green arrow. Thus $w_{\g\b}$ admits the same parenthesis matching as the contour function of the blue tree. 
	\end{proof} 
	
	\begin{figure} [h]
		\centering 
		\includegraphics[width=\textwidth]{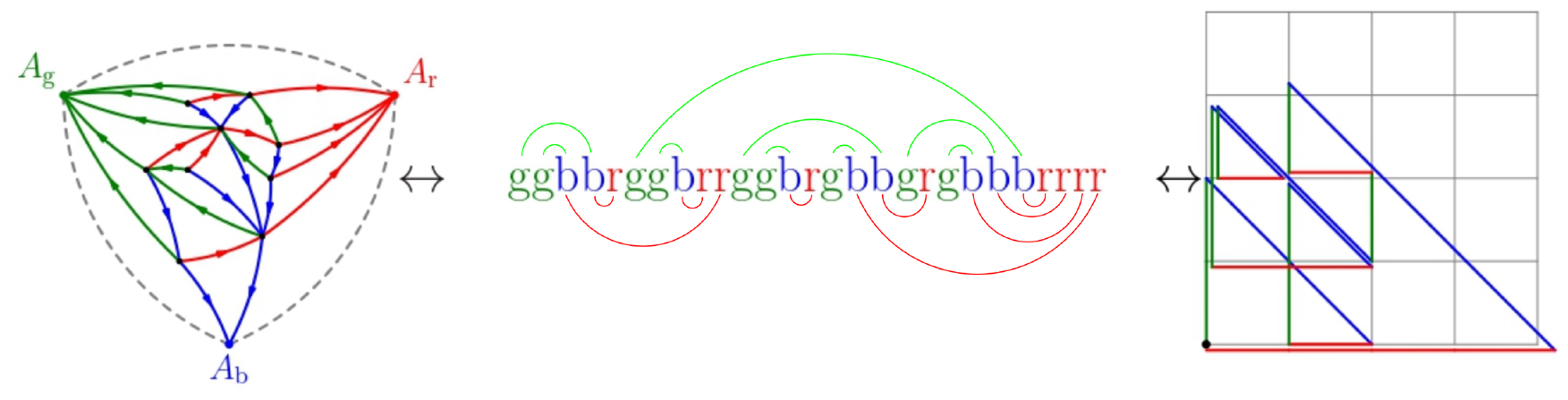}
		\caption{With each wooded triangulation we associate a word and a lattice walk. The values of the lattice walk are perturbed slightly to make it easier to visually follow the path. \nocolor{ The green (resp. red) curves above (resp. below) the word denotes the $\g\b$ (resp. $\b\r$) matches.}
		} 
		\label{fig:wordwalkmap} 
	\end{figure}
	Recall that a \textit{combinatorial map} is a graph together with a (clockwise) cyclic order of the edges incident to each vertex. From this data, we may define combinatorial \textit{faces} and information about how these faces are connected along edges and at vertices. Gluing together polygons according to these rules, we get a surface $X$ together with an embedding of the combinatorial map in $X$. This embedding is unique up to deformation of $X$ \cite{lando2013graphs}. 
	
	For any word $w$ in the symbols b, r, and g, we say that $(w_j,w_k)$ is a {\bf gb match} if $w_j = \g$, $w_k = \b$, and sub-word obtained by dropping the r's from $w_{j+1}\cdots w_{k-1}$ reduces to the empty word under the relation $\g\b = \emptyset$. We also say $w_k$ (resp. $w_j$) is the gb match of $w_j$ (resp. $w_k$). We define the term  {\bf br match} similarly. \nocolor{By the constructions of the gb match and the br match, if a letter in $w$ has a gb (resp. br) match, then the gb (resp. br) match must be unique. }
	
	The following definition provides a recipe to recover a wooded triangulation from its word. 
	
	\begin{definition} \label{def:graph_from_word} 
		Given $w \in\cW_n$, we define a graph $\psi(w)$ with colored oriented edges as follows. The vertex set  is the union of  the set of $\b$'s in $w$ and the symbols $\{A_\b,A_\r,A_\g\}$. A vertex $v$ is called an outer vertex  if $v\in \{A_\b,A_\r,A_\g\}$ and an inner vertex otherwise.
		We define $\ol{A_\b A_\r}, \ol{A_\r A_\g},$ and $\ol{A_\g A_\b}$ to be the three outer edges of $\psi(w)$. For each $1\le i\le 3n$, we define an inner edge associated with  $w_i$ as follows (see Figure~\ref{fig:edgepaths}): 
		\begin{enumerate}
			\item For $w_i=\b$, if there exists $(j,k)$ such that $j<i<k$  and  $(w_j,w_k)$ is a $\g\b$ match, find the least such $k$ and construct a blue edge from $w_i$ to $w_{k}$.  Otherwise construct a blue edge  from $w_i$ to  $A_\b$. 
			
			\item For  $w_i=\r$, find $w_i$'s $\b\r$ match  $(w_{i'},w_i)$.  If there exists $j>i$ with $w_{j}=\g$, find the smallest such $j$ and identify $w_j$'s  $\g\b$ match $(w_j,w_k)$. Construct a red edge from $w_{i'}$ to $w_k$. Otherwise construct a red edge from $w_{i'}$ to $A_\r$.
			
			\item For $w_i=\g$, find $w_i$'s $\g\b$ match $(w_{i},w_{i'})$. If there exists $(j,k)$ such that $j<i<k$ and $(w_j,w_k)$ is a $\b\r$ match, find the greatest such $j$ and construct a green edge from $w_{i'}$ to $w_j$. Otherwise, construct a green edge from $w_{i'}$  to $A_\g$.
		\end{enumerate}
\nocolor{	Here we identify symbols in $w$ with inner edges of $\psi(w)$  and identify b symbols with inner vertices of $\psi(w)$ (the tail of a blue edge identified with a b symbol is the vertex corresponding to that b symbol).
}

		By the edge assigning rule, each inner vertex has exactly one outgoing edge of each color. We now upgrade $\psi(w)$ to a combinatorial map by defining a clockwise cyclic order around each vertex. For an inner vertex, \nocolor{ the clockwise cycling order is the following and obeys  Schnyder's rule}: the unique blue outgoing edge, the incoming red edges, the unique outgoing green edge, the incoming blue edges, the unique outgoing red edge, incoming green edges. We also have to specify the order for incoming edges of each color: the incoming blue edges are in order of the appearance of their corresponding $\b$-symbol in $w$; same rules apply for incoming  red edges; the  incoming green edges are in the reverse order of appearance of their corresponding $\g$-symbol in $w$.
		
		For the edges attached to $A_\b$, we define the clockwise cyclic order by $\ol{A_\b A_\g}$, followed by incoming blue edges  in order of the appearance  in $w$, followed by $\ol{A_\b A_\r}$. For the edges attached to $A_\r$, we define the clockwise cyclic order by $\ol{A_\r A_\b}$, followed by incoming red edges  in order of the appearance  in $w$, followed by $\ol{A_\r A_\g}$. For the edges attached to $A_\g$, we define the clockwise cyclic order by $\ol{A_\g A_\r}$, followed by incoming green edges  in the reverse order of the appearance  in $w$, followed by $\ol{A_\g A_\b}$.
	\end{definition}
	
	\begin{figure} 
		\centering
		\subfigure[]{\includegraphics[width=0.3\textwidth]{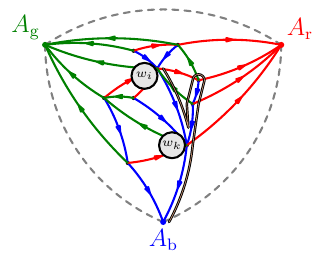}}
		\subfigure[]{\includegraphics[width=0.3\textwidth]{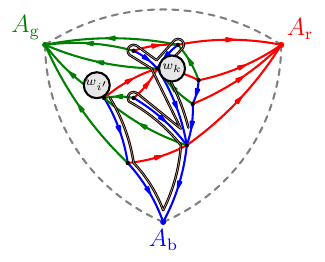}}
		\subfigure[]{\includegraphics[width=0.3\textwidth]{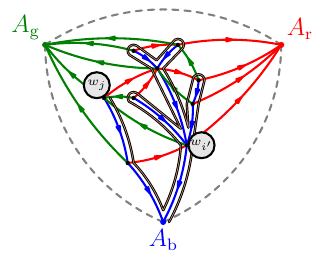}}
		{\includegraphics[width=\textwidth]{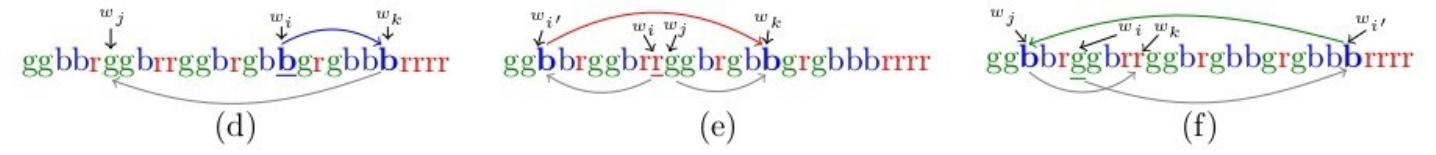}}
		\caption{The relevant portions of the path $\cP(S)$ for the construction of (a) blue edges, (b) red edges, and (c) green edges. And an illustration of  how to construct inner edges from its word in Definition~\ref{def:graph_from_word}: (d) blue edges, (e) red edges, and (f) green edges}
		\label{fig:edgepaths}
	\end{figure}
	
	We will use the following two lemmas to conclude the proof of Theorem~\ref{thm:bijection from wood to walk}. Given a word $w\in \cW_n$, define $\mathcal{M}_{\b\r}(w)$ to be the submap of $\psi(w)$ whose edge set consists of all of the blue edges, all of the red edges, and $\ol{A_\b A_\r}$, and whose vertex set consists of all the vertices of $\psi(w)$ except $A_\g$. 
	
	\begin{lemma} \label{lem:Mbr_planar} 
		$\mathcal{M}_{\b\r}(w)$ is a planar map, for all $w\in \cW_n$. 
	\end{lemma} 
	
	\begin{proof} 
		Let $\cM_\b$ be the subgraph of $\psi(w)$ consisting of all its blue edges.  Recall the definition of blue edges in Definition~\ref{def:graph_from_word}. We  can embed $\cM_\b$ on the plane so that $w_{\g\b}$ is its Dyck word and $A_b$ is \nocolor{the head of} its root.  Moreover $\ol T_\b:=T_\b\cup\{\ol{A_\b A_\r}\}$ is a spanning tree of $\cM_{\b\r}$. We further embed $\ol{A_\b A_\r}$ so that it is the last edge in the clockwise exploration of  $\ol T_\b$.  Now we cut each red edge of $\cM_{\b\r}$ into an incoming and outgoing arrow so that $\cM_{\b\r}$ is transformed to $\ol T_\b$ with $2n$ total red arrows attached at its vertices. Now we can embed the red arrows in the plane uniquely so that the edge ordering around each vertex is consistent with the ordering rule in Definition~\ref{def:graph_from_word}.
		
		For each inner vertex $v$ we find the symbol $w_i=\b$ corresponding to $v$. Let $w_j=\g$ and $w_k=\r$ be the gb match and br match of $w_i$  respectively. Then by Definition~\ref{def:graph_from_word}, the edges corresponding to $w_i,w_j,w_k$ are the unique blue, green, red outgoing edges from $v$. We identify the red outgoing arrow from $w_k$ with the b symbol $w_i$.  Let $\ell=\max\{m<j: w_m\neq\r  \}$. Then there exists an incoming red arrow at $v$ if and only if $w_{j-1}=\r$ and the $\r$ symbols corresponding to these incoming arrows, appearing in clockwise order, are $w_{\ell+1},\cdots,w_{j-1}$. Let $\ell_0=\max\{m: w_m\neq \r  \}$. Then the $\r$ symbols corresponding to the incoming red arrows at $A_\r$,  appearing in clockwise order, are $w_{\ell_0+1},\ldots,w_{3n}$. Therefore if we clockwise-explore $\ol T_\b$, the $2n$ red arrows encountered appear in the same order as in $w_{\b\r}$, where each b (resp., r) symbol corresponds to  an outgoing (resp., incoming) arrow. Since $w_{\b\r}$ is a parenthesis matching, we may link red arrows in a planar way to recover $\cM_{\b\r}$. 
	\end{proof}
	Since $\cM_{\b\r}(w)$ is planar, we may embed it in the plane so that the face right of $\ol{A_\b A_\r}$ is the unbounded face.  
	We now describe the face structure of $\cM_{\b\r}$. Let $Z^\b=(L^\b,R^\b)$ be the walk corresponding to $w$. For each $w_{i}=\b$, let $v$ be its corresponding inner vertex and $w_j=\r$ be the br match of $w_i$. Then the blue and red outgoing edges from $v$ are $w_i$ and $w_j$.  Let
	\begin{equation}\label{eq:green}
	\cG_i=\left\{k>i: w_{k}=\g \;\textrm{ and }\; L^\b_{k}=L^\b_i=\min_{i  \le j \le k} L^\b_j \right\}.
	\end{equation}
	By Definition~\ref{def:graph_from_word}, there exist incoming green edges attached to $v$ if and only if $\cG_i\neq \emptyset$. In this case, write elements in $\cG_i$ as $k_1<\cdots <k_m$. Then $i<k_1<\cdots<k_m<j$ and the green edges attached to $v$  in counterclockwise order between $w_i$ and $w_j$ are $w_{k_1}\cdots,w_{k_m}$.
	Let $v_0$ (resp., $v_{m+1}$) be the head of $w_i$ (resp., $w_j$).  For $1\le \ell\le m$, let $v_\ell$ be the  tail of $w_{k_\ell}$. Let $\cF(e)$ be the face of $\cM_{\b\r}$ on the left of $e$ where $e$ is the blue edge corresponding to $w_i$. The following lemma describes the structure of $\cF(e)$. 
	\begin{lemma} \label{lem:face_structure}
		The vertices on $\cF(e)$ are $v,v_0,\cdots,v_{m+1}$ in counterclockwise order.  
		
		Furthermore, $\cF$ is a bijection between blue edges and inner faces of $\cM_{\b\r}$.
	\end{lemma}
	\begin{figure} [h]
		\centering
		\subfigure[]{\includegraphics[width=0.3\textwidth]{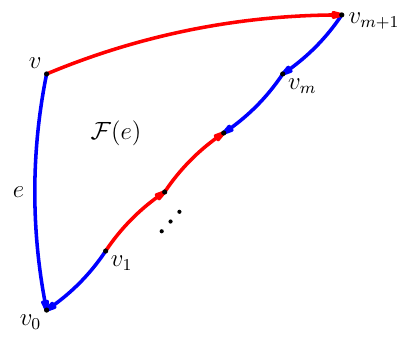}}
		\subfigure[]{\includegraphics[width=0.3\textwidth]{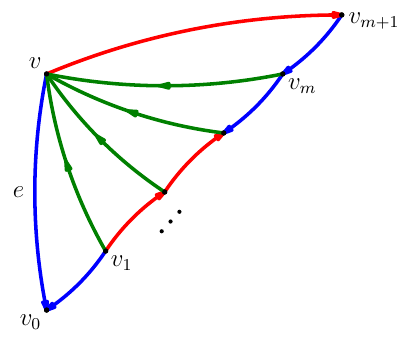}}
		\caption{(a) Each face of $\mathcal{M}_{\b\r}$, traversed counterclockwise, consists of a blue forward-oriented edge $e$, followed by a sequence of reverse blue or forward red edges, followed by a reverse red edge back to $e$'s tail. (b) The green edges triangulate each face of $\mathcal{M}_{\b\r}$. 
			\label{fig:face}} 
	\end{figure}
	\begin{proof} 
		First note that $v$ and $v_{m+1}$ are on $\cF(e)$. We split the proof into cases. 
		
		If $w_{i+1}=\r$, then $m=0$ and $v,v_0$, and $v_1$ form a triangle where $\ol{v_1v_0}$ is a blue edge. If $w_{i+1}=\b$, then  the match of $w_{i+1}$ in $w_{\b\r}$ is $w_{j-1}$ if $m=0$ and $w_{k_1-1}$ if $m>0$. In both cases  $\ol{v_0v_1}$ is a red edge.  If  $w_{i+1}=\g$, then $\ol{v_1v_0}$ is a blue edge. Moreover, regardless of the value of $w_{i+1}$, there are no blue and red edges incident to $v_0$ between $\ol{v_0v}$ and $\ol{v_0v_1}$. Therefore $v_0$ is on $\cF(e)$ and is counterclockwise after $v$. 
		
		If $m>0$, the same argument with $w_{k_1+1}$ in place of $w_{i+1}$ implies that  $v_1$ is on $\cF(e)$ and counterclockwise after $v_0$. Moreover, if $w_{k_1=1}=\b$, then $\ol{v_1v_2}$ is a red edge. Otherwise $\ol{v_2v_1}$ is a blue edge. By induction, the  general statement holds for $v_i$ and $v_{i+1}$ for all $0\le i\le m$. This proves the first statement in Lemma~\ref{lem:face_structure}. Moreover, it also yields that $\cF$ is an injection  \nocolor{from the blue edges to the} inner faces of $\cM_{\b\r}$.
		
		Note that $\mathcal{M}_{\b\r}$ has $V = n+2$ vertices and $E = 2n+1$ edges. Therefore, Euler's formula implies that $\mathcal{M}_{\b\r}$ has $2 -V + E = n+1$ faces and $n$ inner faces. Since $\cM_{\b\r}$ also has $n$ blue edges and $\cF$ is an injection,  it follows that  $\cF$ is a bijection.
	\end{proof}
	\begin{figure}[h]
		\centering 
		\subfigure[]{\includegraphics[width=0.3\textwidth]{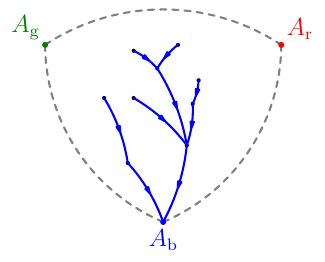}}
		\subfigure[]{\includegraphics[width=0.3\textwidth]{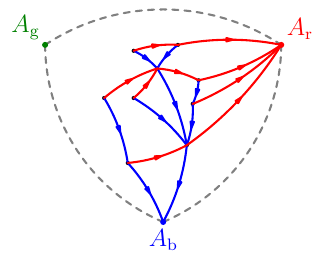}}
		\subfigure[]{\includegraphics[width=0.3\textwidth]{figures/woodonly}}
		\caption{Construction from $w$ to the wooded triangulation $\psi(w)$: (a) draw the planar tree $\mathcal{M}_{\b}$; (b) draw  $\mathcal{M}_{\b\r}$ in a planar way as in Proposition~\ref{lem:Mbr_planar}; (c) finally, the green edges triangulate each face in Figure~\ref{fig:blueredgreen}(b)  in the manner of Figure~\ref{fig:face}(b).} 
		\label{fig:blueredgreen} 
	\end{figure}
	By Lemma~\ref{lem:face_structure}, each inner face  in $\cM_{\b\r}$ is of the form in Figure~\ref{fig:face}(a). Namely, there is a unique blue (resp., red) edge so that the face is left (resp., right) of that edge.   Let  $\ol\cM_{\b\r}:=\cM_{\b\r}\cup\{\ol{A_\g A_\b}, \ol{A_\g A_\r}  \}$. We can embed 
	$\ol \cM_{\b\r}$ so that $\ol{A_\b A_\r}, \ol{A_\g A_\b}, \ol {A_\g A_\r}$ form the unbounded face of $\ol \cM_{\b\r}$. 
	Let $f_{\b\r}$ be the inner face of $\ol\cM_{\b\r}$ containing $A_\g$.  To understand the structure of $f_{\b\r}$, let $w'\in \cW_{n+1}$ be the concatenation of g,b, $w$, and $\r$. If we identify  $\ol{A_\g A_\b}$ (resp., $\ol{A_\g A_\r}$) as  a blue (resp., red) directed edge, then $\ol\cM_{\b\r}$ is isomorphic to $\cM_{\b\r}(w')$ and $f_{\b\r}$ becomes an inner face of $\cM_{\b\r}(w')$, thus by Lemma~\ref{lem:face_structure} also has the form in Figure~\ref{fig:face}(a).
	
	\begin{proof}[Proof of Theorem~\ref{thm:bijection from wood to walk}]
		By Proposition~\ref{prop:paren_matching}, we have  $\varphi(\cS_n)\subset\cW_n$. For all $w\in \cW_n$, Definition~\ref{def:graph_from_word} constructs a combinatorial map $\psi(w)$.  Now we show that $\psi(w)\subset \cS_n$ and $\varphi(\psi(w))=w$, which will yield the surjectivity of $\varphi$.
		
		According to the structure of inner faces in $\ol\cM_{\b\r}$ above and Lemma~\ref{lem:face_structure}, we can recover $\psi(w)$ by triangulating each inner face $f$ of $\ol\cM_{\b\r}$  by green edges. When $f$ is an inner face of $\cM_{\b\r}$, then vertices of $f$ are of the form $v,v_0,v_1,\cdots,v_{m+1}$ as in Lemma~\ref{lem:face_structure}. In this case we add  green edges from $\{v_i\}_{1\le i\le m}$ to $v$  provided $m>0$. If $f=f_{\b\r}$, we add green edges to $A_g$ in the same way. This shows that $\psi(w)$ is a triangulation.
		Moreover, by the cyclic order in Definition~\ref{def:graph_from_word}, $\psi(w)$ is a size-$n$ \nocolor{wooded triangulation}. Knowing that $\psi(w)\in \cS_n$, it is clear that the edge-symbol correspondence in Definition~\ref{def:graph_from_word} is identical to the one defined by clockwise exploring the blue tree of $\psi(w)$ as in Section~\ref{subsec:bijection}.
		
We are left to show that $\psi(\varphi(S))=S$ for all $S\in \cS_n$.  \nocolor{Let $w=\varphi(S)$. First of all, as mentioned in the proof of Lemma \ref{lem:Mbr_planar} , the Dyck word of $\cM_{\b}(w)$ is $w_{\g\b}$. So we can identify $\cM_{\b}(w)$ with $T_{\b}(S)$ since the Dyck word of $T_{\b}(S)$ is also $w_{\g\b}$. Under this identification, if $w_k$ is a b symbol, then the inner vertex of $\psi(w)$ corresponding to $w_k$ identifies with the tail of a blue edge in $S$ such that $w_k$ is given by $\cP_S$'s second traversal of this blue edge. Second, for each inner vertex of $\psi(w)$ and $S$, consider the incoming (resp., outgoing) red edges as incoming (resp., outgoing) arrows attched at this vertex, as in the proof of Lemma \ref{lem:Mbr_planar}. By Schnyder’s rule, the number of outgoing arrows at each inner vertex is one. Suppose $w_i$ is a b symbol and $w_j$ is the gb match of $w_i$. By Definition \ref{def:graph_from_word}, $w_i$ corresponds to a blue edge in $\psi(w)$ as long as its tail, say, $v$. As mentioned in the proof of Lemma \ref{lem:Mbr_planar}, the number of incoming red arrows attached to $v$ is the number of consecutive r symbols right before $w_j$ in $w$. On the other hand, by the definition of $\varphi$, $w_i$ comes from   $\cP_S$'s second traversal of a blue edge, say, $e$. Moreover, $\cP_S$'s crossing of the incoming red edge at the tail of $e$ corresponds to the consecutive r symbols right before $w_j$ in $w$. Therefore the number of incoming red arrows at the tail of $e$ equals the number of consecutive r symbols right before $w_j$ in $w$. So the number of red (incoming and outgoing) arrows at each inner vertex of $S$ is the same as that for the corresponding inner vertex of $\psi(w)$. By Lemma \ref{lem:Mbr_planar},  $T_\b(S)\cup T_\b(S)$ coincides with $\cM_{\b\r}(w)$. Finally,} the way to obtain $\psi(\varphi(S))$ from $\ol\cM_{\b\r}$ by adding green edges as above is also the way to obtain $S$ from \nocolor{$T_\b(S)\cup T_\r(S)\cup\{\ol{\Ab\Ar}, \ol{\Ar\Ag},\ol{\Ag\Ab}\}$}. This yields $\psi(\varphi(S))=S$.
	\end{proof}
	\subsection{Dual map, dual tree and counterclockwise exploration}\label{subsec:dual}
	
	\begin{figure}
		\centering
		\subfigure[]{\includegraphics[height=5cm]{figures/woodonly}}
		\subfigure[]{\includegraphics[height=5cm]{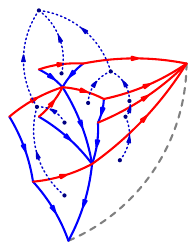}}
		\subfigure[]{\includegraphics[width=5cm]{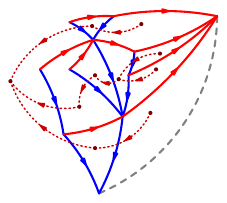}}
		\caption{A Schnyder wood along with its dual blue tree and dual red tree.\label{fig:dualstuff}}
	\end{figure}
	
	In light of Theorem~\ref{thm:bijection from wood to walk}, we may apply the above constructions of $T_\b,\ol T_\b$ and $\cM_{\b\r}$ to obtain these planar maps for any $S \in \cS_n$. Fix such an $S$, and note that $\ol T_\b$ is a spanning tree of $\cM_{\b\r}$. We define a spanning  tree on the \textit{dual map} of $\cM_{\b\r}$ (that is, the map of faces of $\cM_{\b\r}$) rooted at the outer face by \emph{counterclockwise} rotating each red edge. In other words, we form a directed edge from $F_1$ to $F_2$ if they are the faces  on the right and left sides (respectively)  of some directed red edge in $\cM_{\b\r}$. We call this tree the \textbf{dual blue tree} of $S$.
	
	We can define $T_\r$ and $\ol T_\r$ similarly to $T_\b$ and $\ol T_\b$ with $\r$ in place of $\b$. Then $\ol T_\r$ is also a spanning tree of $\cM_{\b\r}$.  By \emph{clockwise} rotating each blue edge, we obtain the \textbf{dual red tree} of $S$.  For any inner face $f$ of $\cM_{\b\r}$,  the branch on the dual blue tree from $f$ towards the dual root is called the \textbf{dual blue flow line}. We define  the \textbf{dual red flow line} similarly with red in place of blue.
	
	Recall the map $\cF$ in Lemma~\ref{lem:face_structure}.  We extend $\cF$ to the set of all inner edges as follows:  let $\cF(e)$ be the face of $\cM_{\b\r}$ on the \emph{left} of $e$ if $e$ is a blue or red edge;  let $\cF(e)$ be the face \emph{containing} $e$ if $e$ is a green edge.   We call $\cF$ the \textbf{face identification map}.  We also define $\wt\cF(e)$ by replacing left with right in the definition of $\cF(e)$.
	
	Recall the exploration path $\cP_S$ in the definition of
	$Z^{\mathrm{b}}$.    By
	reversing the direction of $\cP_S$ and swapping the roles of red and green
	edges, we can define a lattice walk $\prescript{\mathrm{b}}{}{Z}$ corresponding to the counterclockwise exploration of $T_\b$. More precisely, let  $\wt{\cP}_S$ be the time reversal of $\cP_S$. Then  $\prescript{\mathrm{b}}{}{Z}$ takes a
	$(1,-1)$-step if $\wt \cP_S$ traverses a blue edge for the second time, a $(0,1)$-step (resp., $(-1,0)$-step) if $\wt\cP_S$ crosses a red (resp., green) edge for the second time. We can similarly define $\prescript{\r}{}{Z}$ and $\prescript{\g}{}{Z}$ for the counterclockwise explorations of the red and green tree,  respectively. Note that $\prescript{\mathrm{b}}{}{Z}$ is \emph{not} the time reversal of $Z^\b$ in general.  We introduce these counterclockwise walks because if we switch from clockwise exploration of $T_\b$ to the counterclockwise exploration and  swap the roles of red and green, then $(Z^\b,\cM_{\b\r}, \ol T_\b, \textrm{ dual blue  tree})$ becomes $(\rr Z,\cM_{\b\r}, \ol T_\r, \textrm{ dual red tree})$. This symmetry will be important in Section~\ref{sec:peano}. (Also see Proposition~\ref{prop:dual} below.)
	
	Write  $Z^\b=(L^\b,R^\b)$ and $\prescript{\r}{}{Z}=(\prescript{\r}{}{R},\prescript{\r}{}{L})$. The letters $R$ and $L$ are arranged in this way because $T_\b$ (resp., $T_\r$) is right (resp., left) of $\cP_S$ (resp., $\wt \cP_S$), while by Proposition~\ref{prop:paren_matching},   $R^\b$  (resp., $\prescript{\r}{}{L}$) is the contour function (modulo flat steps) of $T_\b$ (resp., $T_\r$).

	To complete the mating-of-trees picture of our bijection $\varphi$, we now show that
	$L^\b$ and  $\prescript{\r}{}{R}$ are contour functions (modulo  flat steps) of the dual blue and dual red trees respectively.
	\begin{proposition}\label{prop:dual}
		In the above setting, for all $1\le i\le 3n$, $L^\b_i$
		equals the number of edges on the dual blue flow line from
		$\cF(w_i)$ to the dual root.
		
		The similar result holds for $\prescript{\r}{}{R}$, 
		$\wt \cF$, and dual red flow lines.
	\end{proposition}
	\begin{proof}
		It suffices to show that for all $1\le k\le 3n$,
		$L^\b_{k}-L^\b_{k-1}$ equals the difference between the
		number of edges on the dual blue flow lines starting from
		$\cF(w_{k+1})$ and $\cF(w_k)$. When $w_k=\r$ (resp.,
		$w_k=\b$), the face $\cF(w_{k})$ is one step forward (resp.,
		backward) along the dual blue flow line from
		$\cF(w_{k-1})$. Moreover, $\cF(w_k)=\cF(w_{k-1})$ when
		$w_k=\g$. Therefore the claim holds for each possible value
		of $w_i$. The result for $\prescript{\r}{}{R}$ follows by
		the symmetry of blue and red in $\cM_{\b\r}$.
	\end{proof}
	
	\subsection{Uniform infinite wooded triangulation}\label{subsec:infinite}
	
	Thus far we have considered uniform samples from the set of all wooded
	triangulations of size $n$. We refer to this as the \textit{finite
		volume} setting. In this section we introduce an \textit{infinite
		volume} setting by defining an object which serves as the
	$n\to\infty$ limit of the uniform wooded triangulation of size $n$,
	rooted at a uniformly selected edge.
	
	Let $\{w_k \, : \, 1 \leq k \leq 3n\}$ be a Markov chain on the state space
	$\{\mathrm{b},\mathrm{r},\mathrm{g}\}$ with $w_1=\mathrm{g}$ and
	transition matrix given by
	\begin{equation}\label{eq:matrix} 
	P= 
	\kbordermatrix{
		& \mathrm{b} & \mathrm{r} & \mathrm{g} \\
		\mathrm{b} & \frac{1}{2} &\frac{1}{4} & \frac{1}{4} \\
		\mathrm{r} & 0 & \frac{1}{2} & \frac{1}{2} \\
		\mathrm{g} &\frac{1}{2} &\frac{1}{4} & \frac{1}{4} \\
	}
	\end{equation}
	
	Define $f:\{\b,\r,\g\} \to \Z^2$ by
	\begin{equation} \label{eq:f}
	f(\mathrm{b})=(1,-1), \quad f(\mathrm{r})=(-1,0), \quad 
	f(\mathrm{g})=(0,1).
	\end{equation}
	Define the lattice walk $Z^{\mathrm{b},n}$
	inductively by $Z^{\mathrm{b},n}_{0} = (0,0)$ and
	$Z^{\mathrm{b},n}_{k} -Z^{\mathrm{b},n}_{k-1}=f(w_k)$ for
	$1 \leq k \leq 3n$. The following proposition tells us that the law of
	$Z^{\mathrm{\b,n}}$ may be mildly conditioned to obtain the uniform
	measure on $\cW_n$.
	\begin{proposition} \label{prop:sampler} The conditional law of $Z^{\mathrm{b},n}$
		given $\{Z^{\mathrm{b},n} \in \mathcal{W}_n\}$ is the uniform measure on
		$\mathcal{W}_n$. Moreover,
		\begin{equation}\label{eq:exponent}
		\P[Z^{\mathrm{b},n}\in \mathcal{W}_n]=\left(\tfrac{48}{\pi}+o_n(1)\right) n^{-5}, 
		\end{equation}
		where $o_n(1)$ denotes a quantity tending to zero as $n\to\infty$. 
	\end{proposition}
	\begin{proof}[Proof of Proposition~\ref{prop:sampler}]
		We claim that for all walks $w \in \mathcal{W}_n$, we have
		$\P[Z^{\mathrm{b},n} = w] = 2 \cdot 16^{-n}$. To see this, write
		the transition matrix as
		\[
		P=\frac{1}{4}\left[ \begin{array}{ccc}
		2 & 1 & 1 \\
		0 & 2 & 2 \\
		2 & 1 & 1
		\end{array} \right].
		\]
		For a word $w=w_1\cdots w_{3n}\in \cW_n$, 
		\[
		\P[\mathrm{word}(Z^{\mathrm{b},n}) = w] = \left( \frac{1}{4} \right)^{3n-1} 2^{\#(\mathrm{g} \to \mathrm{b})}
		2^{\#(\mathrm{b} \to \mathrm{b})}2^{\#(\mathrm{r} \to \mathrm{g})}2^{\#(\mathrm{r} \to \mathrm{r})}, 
		\]
		where $\#(\mathrm{r} \to \mathrm{b})$ denotes the number of integers
		$1 \leq i < 3n$ for which
		$(w_i,w_{i+1}) = (\mathrm{r},\mathrm{b})$, and similarly for the
		other expressions. \nocolor{Here we used the bijection between $\cS_n$ and $\cW_n$.} The first two factors multiply to give $2^{n}$,
		since there are $n$ occurrences of $\mathrm{b}$ in $w$. Similarly, the
		last two factors multiply to give $2^{n-1}$, since there are $n$
		occurrences of $\mathrm{r}$ in $w$, with one at the end. Multiplying
		these probabilities gives the desired result.
		
		For the second part of the proposition statement, we use
		the formula \eqref{eq:catalan} for the number of wooded triangulations
		of size $n$. Stirling's formula implies that
		the right-hand side of \eqref{eq:catalan}
		is    asymptotic to $\tfrac{24}{\pi} n^{-5} 16^n$. Applying the first part
		of the present proposition and Theorem \ref{thm:bijection from
			wood to walk} concludes the proof.
	\end{proof}
	\begin{remark} 
		Consider a 2D  \nocolor{Brownian motion} whose covariance is $-\cos(4\pi/\kappa)$
		as in Theorem~\ref{thm:mating}. The probability that it stays in the
		first quadrant over the interval $[0,2n]$ and returns to the origin is of order
		$n^{-\frac{\kappa}4 -1}$; \nocolor{see e.g.~\cite{duraj2015invariance}.}  \nocolor{The Brownian motion is the limit of the lattice walk, therefore,}  the exponent $\alpha = 5$ in
		Proposition~\ref{prop:sampler} and \eqref{eq:catalan} is related to
		$\kappa = 16$ via the equation $\alpha = \frac{\kappa}4+1$.  This
		also explains why for uniform spanning tree decorated map we have 
		$\kappa=8$ and $\alpha=3$ \cite{burger}; and for a bipolar-oriented
		map we have 
		$\kappa=12$ and $\alpha=4$ \cite{bipolar}.  
	\end{remark}
	
	From now on we \nocolor{slightly abuse the notation and redefine  $Z^{\b,n}$ to be} the conditioned \nocolor{walk} in
	Proposition~\ref{prop:sampler}.  Note that the stationary distribution
	for $P$ is the uniform measure on $\{\b,\r,\g \}$. Now we define a
	bi-infinite random word $w$ so that $\{w_i\}_{i\ge 0}$ has the law of
	the Markov chain $P$ starting from the stationary distribution. And we
	extend $w$ to $\Z_{<0}$ so that $\{w_i\}_{i\in \Z}$ is a stationary
	sequence.  We set $Z^{\mathrm{b},\infty}_{0}=(0,0)$ and
	$Z^{\mathrm{b},\infty}_{k} - Z^{\mathrm{b},\infty}_{k-1}=f(w_k)$ for
	all $k\in \Z$, where $f$ is defined as in \eqref{eq:f}.
	
	\begin{proposition} \label{prop:markov_local_limit} In the above setting, let $U$ be a
		uniform element of $\{0,1,\ldots,3n-1\}$, independent of
		$Z^{\mathrm{b},n}$. Define  the  re-centered walk
		$\widetilde{Z}^{\mathrm{b},n}_k :=
		Z^{\mathrm{b},n}_{(k+U)}-Z^{\mathrm{b},n}_{U}$ for
		$-U\leq k\leq 3n-U$. Then $\widetilde{Z}^{\mathrm{b},n}$ converges locally to
		$Z^{\mathrm{b},\infty}$, in the sense that for any fixed $T$, the
		law of $\left.\widetilde{Z}^{\mathrm{b},n}\right|_{[-T,T]}$
		converges as $n\to\infty$ to the law of 
		$Z^{\mathrm{b},\infty}|_{[-T,T]}$.
	\end{proposition}
	
	\begin{proof} 
		Since $\P[Z^{\b,n}\in \mathcal{W}_n]$ has subexponential decay by
		Proposition~\ref{prop:sampler}, the result follows by Cram\'er's
		theorem for Markov chains \cite{donsker1975asymptotic,donsker1975asymptotic2}. We refer the
		reader to Section 4.2 of \cite{burger} for more
		details.
	\end{proof}
	It is straightforward to verify that the bi-infinite word $w$ satisfies the following:
	\begin{enumerate}
		\item the sub-word $w_{\g\b}$ obtained by dropping all $\r$'s in  $w$ is a parenthesis matching a.s.
		\item the sub-word $w_{\b\r}$ obtained by dropping all $\g$'s in $w$ is a parenthesis matching a.s.
		\item every $r$ in $w$ is followed by an r or g but not a $\b$.
	\end{enumerate}
	
	Knowing the three properties, we can construct a random infinite graph
	$S^\infty$ with colored directed edges as in
	Definition~\ref{def:graph_from_word}. Namely, we identify the b
	symbols in $w$ as vertices in $S^\infty$. Then using the edge
	identification rule (1)-(3) in Definition~\ref{def:graph_from_word},
	we identify b, g, and r symbols in $w$ with blue, red and green
	directed edges on $S^\infty$.  We call the edge $e_0$ corresponding to
	$w_0$ the root of $S^\infty$.  Note that we don't introduce
	$A_\b,A_\r, A_\g$ here because in the edge identification rule in
	Definition~\ref{def:graph_from_word}, the possibility when
	$A_\b,A_\r, A_\g$ are attached to a colored edge a.s. never
	occur. Moreover, we don't define the ordering of edges around vertices
	at this moment since we don't want to involve the technicality of
	infinite combinatorial maps. Therefore $S^\infty$ is not a map
	yet. However, Proposition~\ref{prop:infinite} below will imply that
	$S^\infty$ naturally carries an infinite planar map structure and a
	Schnyder wood structure (that is, a 3-orientation).

	We adapt the notion of Benjamini-Schramm convergence
	\cite{benjamini2011recurrence} to the setting of rooted graphs with
	colored directed edges: let $(S^n)_{n\ge 0}$ be a sequence of random
	rooted graph with colored and directed edges.  We say that
	$S^n \to S^0$ as $n\to\infty$ in the Benjamini-Schramm sense if for
	every colored and oriented rooted graph $T$ and every $r>0$, the
	probability that the radius-$r$ neighborhood of the root in $S^n$ is
	isomorphic to $T$ converges as $n\to\infty$ to the probability that
	the radius-$r$ neighborhood of the root in $S^0$ is isomorphic to $T$.
	
	\begin{proposition} 
		\label{prop:infinite} Let $S^n$ be a uniform wooded triangulation of
		size $n$ and $e_n$ be a uniformly and independently sampled inner
		edge. Then $(S^n,e_n)$ converges in the Benjamini-Schramm sense to
		$(S^\infty,e_0)$.
	\end{proposition}
	
	\begin{figure} 
		\centering 
		\includegraphics[width=0.35\textwidth]{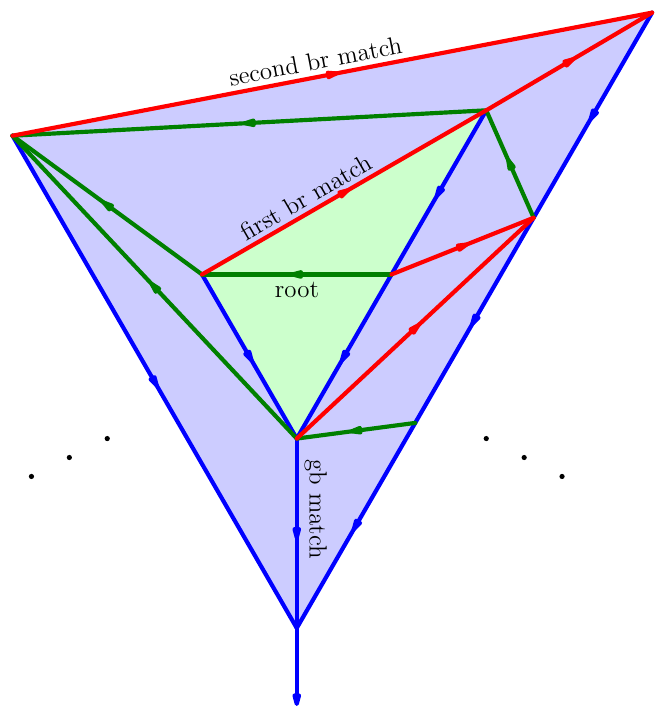}
		\caption{To prove Proposition~\ref{prop:infinite},  we use an
			alternating sequence of enclosing br and gb matches to identify
			a sequence of subgraphs $G_1,G_2\cdots$. The  green region is $G_1$ and  the blue region is $G_{m'}$. \label{fig:pockets}}
	\end{figure}
	
	\begin{proof}
		Let $w^n$ be the word corresponding to $S^n$ and $\wt w^n$ be $w_n$
		recentered at $U_n\in \{1,\cdots,3n\}$ where $U_n$ is the index
		corresponding to $e_n$ (that is, $w^n_{U_n}=e_n$).  Let
		$(\wt w^n_{j_1(n)}, \wt w^n_{k_1(n)} )$ be the br match with minimal
		$k_1(n)$ such that $j_1(n)<0<k_1(n)$. Let $(w_{j_1}, w_{k_1} )$
		denote the br match in $w$ with minimal $k_1$ so that $j_1<0<k_1$.
		Then by Proposition~\ref{prop:markov_local_limit}, we can couple
		$w,\wt w^n$ so that $\wt w^n([j_1(n),k_1(n)])=w([j_1,k_1])$ with
		probability $1-o_n(1)$. Let $G^n_1$ (resp., $G_1$) be the graph
		consisting of edges in $\wt w^n([j_1(n),k_1(n)])$ (resp.,
		$w([j_1,k_1])$). Then $G^n_1$ and $G_1$ can be coupled so as to be
		equal with probability $1-o_n(1)$.  This also means that $G_1$
		is a.s. planar.
		
		We let $(w_{j_m}, w_{k_m} )$ be the br match in $w$ so that
		$j_m<0<k_m$ and $k_m$ is the $m$-th smallest such index. By
		examining the word $w_{\b\r}$, it is straightforward to verify that
		$k_m<\infty$ a.s. for all $m$. We define $G^n_m$ and $G_m$
		similarly. Then $G^n_m$ and $G_m$ can be coupled so as to be equal
		with high probability. Therefore, $G_m$ is a.s. planar.
		
		Now we claim that the graph distance $r_m$ from  $e_0$ to $S^\infty \setminus G_m$   tends to $\infty$ a.s.
		Note that $S^\infty=\bigcup_{m=1}^{\infty} G_m$. This will conclude the proof.
		Since $r_m$ is a non-decreasing integer sequence, it suffices to show that for all $m$ we can find an $m'>m$ such that $r_{m'}>r_m$. We only explain this for $m=1$ since the general case is the same. 
		By examining $w_{\g\r}$, it is easy to see that there are infinitely many gb matches of the form $(w_{i},w_{l})$ so that $i<0<l$. In particular, we can find one such that  $i<j_1<k_1<l$. Let $m'$ be such 
		that $j_{m'}< i<l<k_{m'}$. We abuse notation and also define $i,l,m'$
		for $\wt w^n$ as for $w$. Then by the coupling result above, $i,l,m'$  are well defined with probability $1-o_n(1)$. Moreover, 
		by Schnyder's rule and the COLOR algorithm,  on the event that $i,l,m'$ are well defined for $\wt w^n$,  the graph $G^n_1$ is contained in the interior of $G^n_{m'}$ in the sense that the graph distance from $G^n_1$ to $S^n\setminus G^n_{m'}$ is positive.  (See Figure~\ref{fig:pockets} for an illustration.)  In particular, $r_{m'}>r_1$. \qedhere
	\end{proof}
	\begin{definition} \label{defn:UIWT} 
		We call $(S^\infty,e_0)$ the rooted  {\bf uniform infinite  wooded triangulation},
		abbreviated to {\bf UIWT}. 
	\end{definition}
 		\nocolor{Let $w=\{w_i\}_{i\in \Z}$ be the stationary sequence associated to the Markov chain 
			$P$.  From proof of Proposition~\ref{prop:infinite}, a sample of  $(S^\infty,e_0)$ can be obtained from $w$ as follows. We let $(w_{j_m}, w_{k_m} )$ be the br match in $w$ so that $j_m<0<k_m$ and $k_m$ is the $m$-th smallest such index. Let  $G_m$ be the graph consisting of edges in $w([j_m,k_m])$ as in Figure~\ref{fig:pockets}. Then $S^\infty=\bigcup_{m=1}^{\infty} G_m$ is a sample of UIWT. Let  $Z^{\mathrm{b},\infty}_{0}=(0,0)$ and
			$Z^{\mathrm{b},\infty}_{k} - Z^{\mathrm{b},\infty}_{k-1}=f(w_k)$ with $f$ in \eqref{eq:f}. Then $w$, $Z^{\b,\infty}$, and $(S^\infty,e_0)$ determine each other
			a.s.}
	We call $Z^{\b,\infty}$ is the random walk encoding of
	$S^\infty$ associated with the clockwise exploration of the blue
	tree. Results in Section~\ref{subsec:tree-walk} and \ref{subsec:dual}
	have natural extensions to $(S^\infty,e_0)$.  We can similarly define
	$Z^{\mathrm{r},\infty}, Z^{\mathrm{g},\infty},
	\prescript{\b}{}{Z}^{\infty}, \prescript{\r}{}{Z}^{\infty},
	\prescript{\g}{}{Z}^{\infty}$, and we obtain joint convergence for all
	of them:  
	\begin{proposition} \label{prop:triple_joint_converge_in_law}
		Let $(S^n,e_n)$ be defined in Proposition~\ref{prop:infinite}. Recall
		Section~\ref{subsec:dual}. We  have six walks $Z^{\mathrm{b},n},
		Z^{\mathrm{r},n}, Z^{\mathrm{g},n}$, $\prescript{\b}{}{Z}^{n},
		\prescript{\r}{}{Z}^{n},
		\prescript{^\g}{}{Z}^{n}$ associated with different explorations of
		$S^n$.  The tuple
		$(S^n,e_n,Z^{\mathrm{b},n}, Z^{\mathrm{r},n}, Z^{\mathrm{g},n}, \prescript{\b}{}{Z}^{n}$,$\prescript{\r}{}{Z}^{n}$,
		$\prescript{\g}{}{Z}^{n})$
		jointly converges   to
		$(S^\infty,e_0,$ $Z^{\mathrm{b},\infty}, Z^{\mathrm{r},\infty}, $ $
		Z^{\mathrm{g},\infty}, \prescript{\b}{}{Z}^{\infty}, \prescript{\r}{}{Z}^{\infty}, \prescript{\g}{}{Z}^{\infty})$ in law. 
		Here the convergence of walks is in the re-centered sense as in Proposition~\ref{prop:markov_local_limit}.
	\end{proposition}
\nocolor{	\begin{proof}
  As $(S^n,e_n,Z^{\mathrm{b},n})$ converges to $(S^\infty,e_0,Z^{\mathrm{b},\infty})$, the path $\cP_{S^n}$ also converges to a bi-infinite path $\cP_{S^\infty}$. 
  Since $\prescript{\b}{}{Z}^{n}$ is defined in terms of $\cP_{S^n}$ in a local manner, we see that  $(S^n,e_n,Z^{\mathrm{b},n},\prescript{\b}{}{Z}^{n})$ jointly converge to $(S^\infty,e_0,$ $Z^{\mathrm{b},\infty},  \prescript{\b}{}{Z}^{\infty} )$ in law. By the clockwise/counterclockwise symmetry,   $(S^\infty,e_0)$ and $ \prescript{\b}{}{Z}^{\infty} $ determine each other a.s. By the symmetry of the three colors, the same joint convergence holds for the red and green colors. Since  $(S^\infty,e_0)$ determines  $ Z^{\mathrm{r},\infty}$, $Z^{\mathrm{b},\infty}$, $\prescript{\r}{}{Z}^{\infty}$, and $ \prescript{\g}{}{Z}^{\infty}$, we have the  joint convergence of everything as desired.
	\end{proof}
}

	Now for a UIWT, we can define the map $\cM_{\b\r}$, which is the union
	of blue and red edges. Flow lines, dual flow lines and the face
	identification mappings $\cF$ and $\wt \cF$ can also be defined for a UIWT. And the results in Section~\ref{subsec:tree-walk}
	and~\ref{subsec:dual} have straightforward extensions to UIWT. In particular, we will
	use the following corollary of Proposition~\ref{prop:dual} in Section~\ref{sec:3tree}.
	\begin{proposition}\label{prop:dual2}
		In the above setting, write
		${Z}^{\b,\infty}=(L^{\b,\infty},R^{\b,\infty})$.  For any
		$k_1<k_2$, let $F$ be the face in $\cM_{\b\r}$ where the dual blue
		flow lines from $\cF(w_{k_1})$ and $\cF(w_{k_2})$ merge, which exists almost surely. Then the difference
		between the number of dual edges from $\cF(w_{k_2})$ and $\cF(w_{k_1})$ to
		$F$ equals $L^{\b,\infty}_{k_2}-L^{\b,\infty}_{k_1}$.  	The similar result holds for  $\prescript{\r}{}{Z}^{\infty}$, $\wt \cF$, and dual red flow lines.
	\end{proposition} 
	We conclude this section by studying the number of incoming green
	edges at a vertex. Write $\Geo$ as a geometric random
	variable with success probability $\frac12$ supported on $\Z_{>0}$.
	\begin{lemma} \label{lem:distribution_green_edges_in_corner} Let $T$
		be a stopping time for $w$ with the property that $w_T = \b$ almost
		surely. Then the number $G$ of incoming green edges incident to the
		vertex corresponding to $w_T$ is distributed as $\Geo-1$.
		
		Furthermore, $G$ is measurable with respect to the
		sequence of symbols between $w_T$ and its $\b\r$ match.  
	\end{lemma} 
	\begin{proof} 
		Note that Lemma~\ref{lem:face_structure} holds for the UIWT, by
		Proposition~\ref{prop:infinite}.  In the language of
		\eqref{eq:green}, $G=|\cG_T|$. Now the lemma follows from the strong
		Markov property of $Z^{\b,\infty}$ and the fact that the transition
		probability from b to r and g to r are both $\frac12$.
	\end{proof}

	\section{Convergence of one tree}\label{sec:pair}
	
	In this section we prove marginal convergence of the triple of random
	walks featured in Theorem~\ref{thm:3pair}.  Throughout this section, we let $S^n$ be a uniform wooded triangulation of size $n$ \xin{and} $S^\infty$ be a UIWT. Let $Z^{\b,n},\prescript{\b}{}{Z}^{n}, Z^{\b,\infty},\prescript{\b}{}{Z}^\infty$ be defined as in Section~\ref{subsec:infinite}. Let $w^n$ and $w$ be the words corresponding to $S^n$ and $S$. 
	
	The following observation is an immediate consequence of
	Proposition~\ref{prop:sampler}
	\begin{lemma}\label{lem:fin-walk}
		Let \nocolor{$\ta^n(0)=1$.  For $1\le k< 2n$,  let
		$\ta^n(k) =\min\{ i>\ta^n(k-1) : w^n_i \neq \mathrm{r}\}$.  Define
		$\cZ^{\mathrm b,n}_k\colonequals Z^{\b,n}_{\ta(k)}-Z^{\b,n}_{\ta(0)}$ for $0\le k<2n$. Let   $\cZ^{\mathrm b,n}_{2n}=(0,0)$.}
		%Let $\P^\infty$ be the law
		%of a random walk with i.i.d.  increments distributed as
		%$\tfrac{1}{2} \delta_{(1,-1)} + \sum_{i=0}^\infty 2^{-i-2}
		%\delta_{(-i,1)}$ . 
		\nocolor{Then $(\cZ^n_k)_{0\le k\le 2n}$ is distributed as a $(2n)$-step random walk with   i.i.d.  increments of step distribution 
		\[
		\tfrac{1}{2} \delta_{(1,-1)} + \sum_{i=0}^\infty 2^{-i-2}
		\delta_{(-i,1)}\quad \quad \textrm{(see Figure~\ref{fig:steps} (b) for an
		illustration of the steps)}
		\]
	  conditioned on starting  and ending at the origin  and
		staying in $\{(i,j)\in \Z^2:i\ge 0; j\ge -1   \}$.  We call $\cZ^{\mathrm b,n}$} the {\bf grouped-step
			walk of $Z^{\b,n}$}. 
	\end{lemma}
	\begin{proof}
	\nocolor{Recall that $Z^{b,n}$  starts and ends at $(0,0)$ and stays in the closed first
		quadrant, so it must start with $(0,1)$ and end with $(-1,0)$. Then by the construction we see that $\cZ^{\mathrm b,n}$ must  start and end at $(0,0)$ and stay in  $\{(i,j)\in \Z^2:i\ge 0; j\ge -1   \}$. According to \eqref{eq:matrix} and Lemma \ref{prop:sampler}, the increments of $\cZ^{\mathrm b,n}$ are i.i.d. with the distribution given in the lemma.	
	}
	\end{proof}
	We also would like to define the grouped-step walk for
	$Z^{\b,\infty}$. However, care must be taken about where we start
	grouping red steps. The following lemma is straightforward to check. 
	\begin{lemma}\label{lem:forward}
		Let $T$ be a forward stopping time (that is, a stopping time for
		the filtration $\sigma(\{w_i\}_{i\le k})$) of $w$ so that
		$w_T=\b$ or $w_T = \g$ a.s. Let $\ta(0)=T$ and for all $k\ge 1$, let
		$\ta(k)=\min\{i>\ta(k-1): w_i\neq r \}$. Let
		$\cZ_k=Z^{\b,\infty}_{\ta(k)}-Z^{\b,\infty}_T$ for all $k\ge
		0$. Then $\cZ$ is distributed as \nocolor{a random walk with i.i.d.  increments distributed as
			$\tfrac{1}{2} \delta_{(1,-1)} + \sum_{i=0}^\infty 2^{-i-2}
			\delta_{(-i,1)}$}. We call $\cZ$ the {\bf forward grouped-step walk of $Z^{\b,\infty}$ viewed
			from $T$}, \nocolor{and denote its law by $\P^\infty$.}
	\end{lemma}
		\begin{remark}\label{rmk:step}
			The increments $\{(\eta^x_k,\eta^y_k)\}_{k\ge 0}$ of $\P^{\infty}$
		\nocolor	{satisfy}
			$\E [(\eta^x_k,\eta^y_k)] = (0,0)$, $\E(|\eta^x_k|^2)=2$, $\E(|\eta^y_k|^2)=1$, and
			$\mathrm{cov}(\eta^x_k,\eta^y_k)=-1$.
		\end{remark}
	\begin{remark}\label{rmk:argument}
		The reason we use a random time $T$ to re-center is because otherwise
		we would not get \nocolor{the distribution}  $\P^\infty$. However,  if we set
		$T=\min\{i\ge 3nt : w_i=b\}$, then \xin{the} law of $T-3nt$ does not
		depend on $t$ and has an exponential tail.  Therefore, when we
		consider scaling limit questions where all times are rescaled by
		$(3n)^{-1}$, we can effectively think of $(3n)^{-1}T$ as the
		deterministic constant $t$, even in the finite volume setting where
		we're conditioning on the polynomially unlikely event $w^n\in \cW_n$.
	\end{remark}

%	By Proposition~\ref{prop:sampler},  $\{\ta^n(k)-\ta^n(k-1) \}_{1\le k\le
%		2n}$ are i.i.d.\ random variables distributed as $\ta(1)-\ta(0)$ in
%	Lemma~\ref{lem:forward}  conditioned
%	on  event $\{w^n\in \cW_n\}$.  
	\begin{lemma} \label{lem:mc_concentration} 
		Let $\{X_k\}_{k\ge 1}$ be a sequence of i.i.d. random variables where  $X_1$ distributed as $\Geo$ or $\Geo-1$ or $\ta(1)-\ta(0)$ in Lemma~\ref{lem:forward}. Let $N_t=\inf\{n: \sum_1^n X_i \le t \}$.
		Then
		\begin{equation}\label{eq:renewal}
		\lim_{n\to\infty}\frac {N_{tn}}{tn}=\frac{1}{\E[X_1]} \quad \textrm{a.s. for all $t$ simultaneously}.
		\end{equation}
		Moreover, there exist  absolute constants $c_1,c_2>0$ such that
		\begin{equation}\label{eq:concen}
		\P\left[ \nocolor{\max_{1\le k\le n}  \left|\sum_{i=1}^{k}X_i-k\E[X_1]\right|}>\sqrt{n}\log n\right] \le  c_1n^{-c_2\log n}.
		\end{equation}
	\end{lemma}
	\begin{proof}
		Formula \eqref{eq:renewal} follows from the renewal theorem for
		random variables with finite means, which works for all the three
		distributions.  For \eqref{eq:concen}, when $X_1=\Geo$ or $\Geo-1$, by known concentration results for geometric random variables (see e.g.  		\cite[Theorem 2.1 and 3.1]{janson-tail}), we have  \begin{equation}\label{eq:concen1}
		\P\left[    \left|\xin{\sum_{i=1}^{k}X_i-k\E[X_1]}\right|>\sqrt{n}\log n\right] \le  c_1n^{-c_2\log n}\quad\xin{\textrm{for }1\le k\le n}.
		\end{equation}
		When $X_1\overset{d}{=}\ta(1)-\ta(0)$ in Lemma~\ref{lem:forward}, we have $\P[X_1=1]=\frac34$ and $\P[X_1=i]=2^{-i-1}$ for $i\ge 2$. Therefore  we can couple $Y=\Geo-1$ with $X_1$ such that $Y=X_1$ if $X_1\ge 2$,
		$$\P[Y=0|X_1=1]=\frac12,\quad \textrm{and}\quad \P[Y=1|X_1=1]=\frac14.$$
		Write $X_1=Y+(X_1-Y)$. Then $|X_1-Y|\le 2$, thus by Azuma-Hoeffding inequality $X_1-Y$  satisfies the concentration in \eqref{eq:concen}. \nocolor{Therefore~\eqref{eq:concen} holds for $X$ in all three cases. 
		By \xin{the union bound}   we get~\eqref{eq:concen} from \eqref{eq:concen1} after possibly increasing $c_1$ and decreasing $c_2$.} 
	\end{proof}
	\begin{remark}\label{rmk:concen}
		\nocolor{Sums of the form of $\sum_{i=1}^{n}X_i$  in Lemma~\ref{lem:mc_concentration}  frequently appear when dealing the scaling limits of  UIWT and its associated walks. Lemma~\ref{lem:mc_concentration}  allows us to replace $\sum_{i=1}^{n}X_i$ by $n\E[X_1]$.  Since the error $n^{-c\log n}$ in \eqref{eq:concen} decays much faster than $\P[w^n\in \cW_n]$, this replacement is still valid under  the
		finite volume conditioning. We will apply this observation several	times in Section~\ref{sec:peano}.}
	\end{remark}
	
	Write $Z^{\b,\infty}$ as $(L^{\b,\infty},R^{\b,\infty})$.  By taking $T=\inf\{i\ge0:w_i=b\}$ and \nocolor{applying}  \eqref{eq:error} to $Z^{\b,\infty}$ and $\cZ$, we have that   $\left(\tfrac{1}{\sqrt{4n}}L^{\mathrm{b},\infty}_{\lfloor 3nt
		\rfloor},  \tfrac{1}{\sqrt{2n}}R^{\b,\infty}_{\lfloor 3nt
		\rfloor}\right)_{t\ge 0}$ weakly converges to a Brownian motion $\scZ$ satisfying \eqref{eq:covariance}. By the stationarity of $w$ we have that $\left(\tfrac{1}{\sqrt{4n}}L^{\mathrm{b},\infty}_{\lfloor 3nt
		\rfloor},  \tfrac{1}{\sqrt{2n}}R^{\b,\infty}_{\lfloor 3nt
		\rfloor}\right)_{t\in \R}$ weakly converges  to a two-sided Brownian motion (which we still denote by $\scZ$) with the same variance and covariance. 
	We now explain that the scaling limit of the walk $^\b Z^\infty$ is
	the time reversal of the scaling limit of $Z^{\b,\infty}$, although
	this relation does not hold at the discrete level.

	\begin{proposition}\label{prop:reverse}
		Write $
		\bZb^\infty=(\bRb^\infty,\bLb^\infty)$. Then
		$$\left( \tfrac{1}{\sqrt{4n}}L^{\mathrm{b},\infty}_{\lfloor 3nt
			\rfloor},  \tfrac{1}{\sqrt{2n}}R^{\b,\infty}_{\lfloor 3nt
			\rfloor} \right)_{t\in\R} \quad \textrm{and} \quad
		\left(\tfrac{1}{\sqrt{4n}}\bRb^\infty_{\lfloor 3nt
			\rfloor},  \tfrac{1}{\sqrt{2n}}\bLb^\infty_{\lfloor 3nt
			\rfloor}\right)_{t\in \R}
		$$
		jointly converges in law to the process $\scZ$, defined above,
		and its time reversal.
	\end{proposition}
	
	\begin{proof}
		To show that the
		scaling limits of $Z^{\mathrm{b},\infty}$ and
		$\bZb^\infty$ are related by time
		reversal, it suffices by tightness to show that any subsequential limit
		$(\mathscr{Z},\widehat{\mathscr{Z}} \:\:)$ of
		the two  processes has the property that $\mathscr{Z}$ and
		$\widehat{\mathscr{Z}}$ are time reversals of each other. In other
		words, we want to show that $\mathscr{Z}_{t} = \widehat{\mathscr{Z}}_{-t}$ a.s.
		for all $t$. Without loss of generality we take $t=1$. 
		
		The process $R^{\mathrm{b},\infty}$ over the interval
		$[0,3n]$, with flat steps excised, is equal to the contour
		function of the portion of the blue tree traced by $\cP_S$
		during that interval. Furthermore, by
		Lemma~\ref{lem:mc_concentration}, the asymptotic effect of
		including the flat steps is to time-scale the contour function
		by a factor of $\tfrac{3}{2}$. Similarly, $\bLb^\infty$ is
		asymptotically within $o_n(1)$ of the contour function of the
		same portion of the same tree (also time-scaled by a factor of
		$\tfrac{3}{2}$), but traced in reverse.  Therefore, the
		ordinates of $\mathscr{Z}_{1}$ and
		$\widehat{\mathscr{Z}}_{-1}$ are equal.
		
		It remains to show that the abscissas of $\mathscr{Z}_{1}$ and
		$\widehat{\mathscr{Z}}_{-1}$ are equal. Roughly speaking, the
		idea is to (i) observe that the former counts the discrepancy
		between unmatched b's and r's in the segment $w_0\cdots w_n$
		of the word $w$, and (ii) show that the latter approximately
		counts the same.  By definition, $L^{\mathrm{b},\infty}_n$ is
		equal to $|L \cap [1,3n]_{\Z}| - |R \cap [1,3n]_{\Z}|$, where
		$L$ is the set of integers $k \geq 1$ such that $(w_j,w_k)$ is
		a br match and $j < 0$, and $R$ is the set of integers
		$j \leq 3n$ such that $(w_j,w_k)$ is a br match and $k >
		3n$. (The absolute value bars denote cardinality.)
		
		For $j \geq 1$, suppose that $w_q$ is the $j$th least element
		of $L$ and define $G^L_j$ to be the set of all integers $k$
		such that $w_k = \g$ and there exists $p$ so that $(w_p, w_q)$
		is the innermost br match enclosing $w_k$ (in other words, it is
		the match with maximal $p$ among those satisfying 
		$p < k < q$).  Similarly, for $j \geq 1$, let $w_p$ be the
		$j$th largest element of $R$ and define $G^R_j$
		to be the set of all integers $k$ such that $w_k = \g$ and the
		br  match $(w_p,w_q)$ is the innermost one enclosing $w_k$. 
		Let $G^L$ and $G^R$ be disjoint unions of these sets, as
		follows: 
		\begin{equation} \label{eq:decomposition}
		G^L = G^L_1 \cup \cdots \cup G^L_{|L \, \cap [1,3n]_\Z|} \quad \text{
			and } \quad   G^R = G^R_1 \cup \cdots \cup G^R_{|R \, \cap [1,3n]_\Z|}
		\end{equation}
		By Definition
		\ref{def:graph_from_word}, 
		$\bRb^\infty_{-n}$ is equal to
		$|G^L| - |G^R|$ .
		By the measurability part of
		Lemma~\ref{lem:distribution_green_edges_in_corner},
		$(|G^L_j|)_{j=1}^\infty$ and $(|G^R_j|)_{j=1}^\infty$ are
		i.i.d. random variables with distribution $\Geo-1$.
		
		Since elements of $L$ and $R$ correspond to running infima of
		$L^{\b,\infty}$ and its time reversal, respectively, we have
		$|L \cap [1,3n]_\Z| \le \sqrt n \log n$ and
		$|R \cap [1,3n]_\Z| \le \sqrt n \log n$ with probability
		$1-o_n(1)$. Combined with \eqref{eq:decomposition} and
		Lemma~\ref{lem:mc_concentration}, we have
		$L^{\mathrm{b},\infty}_{3n}=\bRb^\infty_{3n}+O(n^{1/4})$ with
		probability $1-o_n(1)$. This concludes the proof.
	\end{proof}
	We conclude this section with the finite-volume version of
	Proposition~\ref{prop:reverse}. 
	
	\begin{proposition} \label{prop:dual_finite} 
		Write $Z^{\mathrm{b},n}=(L^{\mathrm{b},n},R^{\mathrm{b},n})$ and $\bZb^n=(\bRb^n,\bLb^n)$. Then
		$$\left( \tfrac{1}{\sqrt{4n}}L^{\mathrm{b},n}_{\lfloor 3nt
			\rfloor},  \tfrac{1}{\sqrt{2n}}R^{\b,n}_{\lfloor 3nt
			\rfloor} \right)_{t\in[0,1]} \quad \textrm{and} \quad
		\left(\tfrac{1}{\sqrt{4n}}\bRb^n_{\lfloor 3nt
			\rfloor},  \tfrac{1}{\sqrt{2n}}\bLb^n_{\lfloor 3nt
			\rfloor}\right)_{t\in [0,1]}
		$$
		jointly converges in law to $\scZ^\b$ defined in Theorem~\ref{thm:3pair} and its time reversal.
	\end{proposition}
	
	\begin{proof} 
		Let \nocolor{$\cZ^{\mathrm b,n}=(\cL^{\mathrm b, n}, \cR^{\mathrm b,n})$} be the grouped-step walk of $Z^{\mathrm b,n}$. By the
		invariance principle for random walks in cones (see, e.g.,
		\cite{duraj2015invariance}) and Remark~\ref{rmk:step},
		\nocolor{$\left(\tfrac{1}{\sqrt{4n}}\cL_{\lfloor2nt\rfloor}^{\mathrm b,n},
			\tfrac{1}{\sqrt{2n}}\cR_{\lfloor2nt\rfloor}^{\mathrm b,n}\right)$} converges to $\scZ^b$. 
		By the step distribution of \nocolor{$\cZ^{\mathrm b,n}$}, we have
		\begin{equation}\label{eq:error}
		|Z^\b_m-\cZ^{\mathrm b,n}_{k-1}|\le |\cZ^{\mathrm b,n}_k-\cZ^{\mathrm b,n}_{k-1}|\quad \textrm{for all }  1\le k\le 3n \textrm{ and } \ta^n(k-1) \le m< \ta^n(k).
		\end{equation}
		\nocolor{Now the convergence of $\left( \tfrac{1}{\sqrt{4n}}L^{\mathrm{b},n}_{\lfloor 3nt
			\rfloor},  \tfrac{1}{\sqrt{2n}}R^{\b,n}_{\lfloor 3nt
			\rfloor} \right)_{t\in[0,1]} $} follows from
		Lemma~\ref{lem:mc_concentration} by setting $X_1=\ta(1)-\ta(0)$.
	
	\nocolor{To show the joint convergence in Proposition~\ref{prop:dual_finite},}  
	we note that the proof of Proposition~\ref{prop:reverse} works here as well, except that we have to condition on the polynomially unlikely event $\{w^n\in\cW_n\}$. As explained in Remark~\ref{rmk:concen}, the
		estimate \eqref{eq:concen} allows us to ignore this conditioning and
		the argument in Proposition~\ref{prop:reverse} goes through.
	\end{proof}

	\section{Multiple Peano curves  and an excursion theory}\label{sec:lqg}
In this section we review the necessary background in the continuum.	We focus  on the case $\kappa\in (0,2)$ since eventually we only need $\kappa=1$. Throughout the section, we set 
	\begin{equation}\label{eq:parameter}
	\gamma=\sqrt{\kappa}\in (0,\sqrt2)\quad \textrm{and}\quad \kappa'=16/\kappa>8.
	\end{equation}

	\subsection{SLE, LQG and the peanosphere theory}\label{subsec:lqg}
\nocolor{We review  the peanosphere/mating-of-trees theory~\cite{mating}. \xin{We closely follow Sections 3 and 4 of  the recent survey \cite{ghs-survey}.} 
	We first recall the LQG area measure.
	Let  $\fh$ be a variant of Gaussian free field (GFF) on the plane. The $\gamma$-LQG area measure  $\mu_{\fh}$ is the random measure on the plane formally defined by $e^{\gamma\fh}\,dx\,dy$, which is made rigorous by a regularization and limiting procedure as   in \cite{KPZ}. In this section we let $\fh$ be a particular variant of GFF on the complex plane $\C$,  which is the field describing a LQG surface  called the $\gamma$-quantum cone. In this case, the measure $\mu_{\fh}$ is conjectured to be the scaling limit of the volume measure of the natural infinite volume random planar map model  in the $\gamma$-LQG universality class. For example, we conjecture that when $\gamma=1$, the measure  $\mu_{\fh}$ is the scaling limit of the vertex counting measure of the UIWT under a conformal embedding. We will not recall the precise definition of $\gamma$-quantum cone  and the construction of $\mu_\fh$, which can be found in~\cite[Section 3.4]{ghs-survey}. For the rest of the paper it suffices to know that $\mu_\fh$ is an almost surely infinite  random measure on $\C$ such that Theorem~\ref{thm:mating2} below holds. 
 	
 We now recall the space-filling SLE following \cite[Section 3.6.3]{ghs-survey}. Let $R\D=\{z\in \C: |z|<R\}$ with $R>0$. A chordal   $\SLE_{\kappa'}$ on $R\D$ from $-iR$ to $iR$ is a random non-self-crossing space-filling curve from $-i R$ to $iR$. We denote it  by  $\eta'_R$ and parametrize it such that $\eta'_R(0)=0$ and in each unit of time it traverses a region of Euclidean area 1. Then $\eta'_R$ converges  in law to a space-filling curve $\eta'$ on  $\C$ in the local uniform topology. 
The limiting curve $\eta'$ modulo monotone reparametrization is called a \emph{whole plane space-filling  $\SLE_{\kappa'}$}.   
 Consider an independent coupling of the $\gamma$-quantum cone field $\fh$ and  a whole-plane space-filling $\SLE_\kappa'$ curve $\eta'$, with $\kappa'=16/\kappa$ and $\gamma=\sqrt{\kappa}$ as in~\eqref{eq:parameter}.  
 %The curve  $\eta'$ is viewed as a continuous subjection from a time interval to $\C$ modulo monotone reparametrization.
In Theorem~\ref{thm:mating2}, we will parametrize $\eta'$ using $\mu_\fh$ such that
  \begin{equation}\label{eq:eta-para}
\eta'(0)=0\quad \textrm{and}\quad  \mu_\fh(\eta'([s,t]))=t-s\quad \textrm{for } s<t.
  \end{equation}

 The curve $\eta'$ has the following properties which does not depend on its parametrization, although we use the one in~\eqref{eq:eta-para} for concreteness. 
  First, $\lim_{t\to \pm \infty} |\eta'(t)|=\infty$. Moreover, 
 for $t\in \R$, the intersection of $\eta'(-\infty, t]$ and $\eta'[t,\infty)$ can be viewed as the union of two curves from $\eta(t)$ to $\infty$, which we denote by $\eta_t^{\mathrm{L}}$ and $\eta_t^{\mathrm{R}}$, respectively.  
  We let  $\eta_t^{\mathrm{L}}$  be the curve such that  $\eta'(-\infty, t] $ is on the left side of  $\eta_t^{\mathrm{L}}$ and call $\eta_t^{\mathrm{L}}$  the left boundary of $\eta'(-\infty,t]$.   We call $\eta_t^{\mathrm{R}}$  the right boundary of $\eta'(-\infty,t]$.
  Under our assumption that $\kappa\in (0,\sqrt2)$,  the following holds almost surely: for each $t\in \R$ the curves $\eta_t^{\mathrm{L}}$  and $\eta_t^{\mathrm{R}}$ are simple and they  do not intersect except at their endpoints. 
  According to \cite{Zipper,mating}, under the independent coupling of $\fh$ and $\eta'$,   the field  $\fh$ induces a length measure called the \emph{$\gamma$-quantum length} on the left and right boundaries of $\eta'(-\infty, t]$ for all $t\in \R$.
Let $\scL_t$ (resp. $\scR_t$) be the net change of the quantum lengths of the left (resp. right) boundary of $\eta'(-\infty,t]$ relative to $\eta'(-\infty,0]$.} The following mating-of-trees theorem proved in \cite{mating,gwynne2015brownian} is the infinite volume version of Theorem~\ref{thm:mating}:
\begin{theorem}\label{thm:mating2}
	$\scZ=(\scL,\scR)_{t\in \R}$ is a two-sided Brownian motion such that
	\begin{equation}\label{eq:cov}
	\Var[\scL_1]=\Var[\scR_1]=1 \quad \textrm{and}\quad   \Cov [\scL_1,\scR_1]=-\cos\left(\frac{4\pi}{\kappa'}\right),
	\end{equation} 
	Moreover,  $\scZ$ determines $(\fh,\eta')$ up to rotations.\footnote{Our \eqref{eq:cov} differs from \cite[Theorem 9.1]{mating} by a constant. However, we can redefine the quantum length measure of $\SLE_\kappa$ via a rescaling to ensure~\eqref{eq:cov}.
		This convention is also implicitly made in \cite{mating2}.}
\end{theorem}
Intuitively, the unit area $1$-LQG sphere in Theorem~\ref{thm:mating}  is the $1$-quantum cone conditioned on having mass 1.  Various rigorous conditioning procedures are performed in \cite{mating,mating2} to construct this object.  Another construction appears in \cite{DKRV}.  See \cite{AHS} for their equivalence. The proof of Theorem~\ref{thm:mating} in \cite{mating2}
 is also via conditioning based on~Theorem~\ref{thm:mating2}. 
	
	\subsection{Imaginary geometry on LQG surfaces}\label{subsec:singleflow} 
\nocolor{	Let $\eta'$ be a whole-plane space-filling $\SLE_{\kappa'}$. We  first recall some distributional properties of  $\eta'$ as summarized in~\cite[Section 3.6.3]{ghs-survey}.
	For a fixed deterministic point $z\in \C$,  almost surely  there exists a unique time $t_z$
 such that $\eta'(t_z)=z$. The law of the left boundary  $\eta^{\mathrm L}_{t_z}$  and the right boundary  $\eta^{\mathrm R}_{t_z}$  of $\eta'(-\infty, t_z]$ have the same law, which  is a   variant of $\SLE_\kappa$ curve  called the \textit{whole plane $\SLE_\kappa(2-\kappa)$} from $z$ to $\infty$.  Conditioning on $\eta^{\mathrm L}_{t_z}$ and $\eta^{\mathrm R}_{t_z}$, the curves $\{\eta'(t_z+s)\}_{s\ge 0}$ and $\{\eta'(t_z-s)\}_{s\ge 0}$ evolve as two independent chordal $\SLE_{\kappa'}$ from $z$ to $\infty$ on the two connected components of $\C\setminus (\eta^{\mathrm L}_z \cup \eta^{\mathrm R}_z)$. The law of $\eta'$ has time-reversal symmetry.

	We now recall the basics of imaginary geometry \cite{IGI,miller2013imaginary}. For $\kappa\in(0,2)$, let $\chi=2/\sqrt{\kappa}-\sqrt{\kappa}/2$.
	Fix $\theta\in (-\pi,\pi]$ and $z\in \C$. Consider the ordinary differential equation 
		\begin{equation} \label{eq:ODE}
	\dot{\eta}=e^{\frac{i h(\eta)}{\chi}+\theta}, \quad \eta(0)=z.
	\end{equation}If  $h$ is  a smooth function, then the initial-value problem 
	has a unique solution. When $h$ is a variant of GFF, since $h$ is not pointwise defined, the equation~\eqref{eq:ODE} does not make literal sense. However, according to~\cite{IGI,miller2013imaginary} \xin{(in particular~\cite[Section 1.2.1]{miller2013imaginary})}, there exists a coupling of a whole-plane $\SLE_\kappa(2-\kappa)$ curve $\eta_z^\theta$ from $z$ to $\infty$ and a particular variant of GFF called the  whole-plane Gaussian free field modulo $2\pi\chi$, such that $(h,\eta_z^\theta)$ can be interpreted as satisfying~\eqref{eq:ODE}.  \xin{We do not recall the definition of this particular variant of GFF and the full construction of this coupling, but only summarize the properties that are relevant to mating of trees.} In this coupling $\eta_z^\theta$ is   determined by $h$ a.s. We call $\eta_z^\theta$ the flow line of angle $\theta$ emanating from $z$.  For each $\theta$, there exists a  whole-plane space-filling $\SLE_{\kappa'}$ curve $\eta'^\theta$ such that the following holds. For each fixed $z$, almost surely $\eta^{\mathrm L}_{t_z}$ (resp.,  $\eta^{\mathrm R}_{t_z}$) is the flow line of angle $\theta+\frac{\pi}{2}$ (resp., $\theta-\frac{\pi}{2}$) emanating from $z$. 
	This uniquely characterizes  a coupling of $(h,\eta'^\theta)$ where $\eta'^\theta$ is a.s. determined by $h$.
	If $(h,\wt\eta'^\theta)$ satisfies the same property, then we must have $\eta'^\theta=\wt \eta'^\theta$ almost surely. We call $\eta'^\theta$ the \emph{Peano curve for $h$ of angle $\theta$}.  
	
	The three curves in the Theorem~\ref{thm:3pair} are three Peano curves for the same $h$ with angle $0, \frac{2\pi}{3},\frac{4\pi}{3}$, respectively, where $\kappa=1$ and $\chi=\frac{3}2$. To prove Theorem~\ref{thm:3pair}, we need to understand the coupling between Peano curves of different angles. We first review the results from~\cite{bipolarII} on the coupling between one Peano curve and one  flow line of an arbitrary angle. Suppose  $\fh$ as in Theorem~\ref{thm:mating2} and $h$ is a whole-plane GFF modulo $2\pi\chi$, which is independent of $\fh$.  Suppose $\eta'$ is the Peano curve for $h$ of angle $0$ so that $(\eta',\fh)$ are coupled independently as in Theorem~\ref{thm:mating2}.  Parametrize $\eta'$ by $\mu_\fh$ according to~\eqref{eq:eta-para}. Fix $\theta\in (-\frac\pi 2, \frac\pi 2)$. Let $\eta^\theta$ be the flow line of $h$ of angle $\theta$.  Given $t\ge 0$, we say that $\eta'$ \emph{crosses} $\eta^\theta$ at time $t$ if and only if: (1) $\eta'(t)$ is on the trace of $\eta^\theta$; (2) for each $\eps>0$ there exists $s_1,s_2\in (t-\eps, t+\eps) $ such that $\eta'(s_1)$ and $\eta'(s_2)$ lie on opposite sides   of $\eta^\theta$. 
	As explained in  \cite[Remark 3.1]{bipolarII}, when $\theta$  is such that $\eta^\theta$ does not intersect $\eta^{-\frac\pi2}$ and $\eta^{\frac \pi2}$, then $\eta'$ crosses $\eta^\theta$ at time $t$ if and only if $\eta'(t)$ is on the trace of $\eta^\theta$. When $\kappa=1$,  this corresponds to $|\theta|\le \frac\pi 6$.	
	  The following proposition is extracted from \cite[Propositions 3.2---3.4]{bipolarII}. See Figure~\ref{fig:twosides} for an illustration.}

	\begin{figure}[h]
		\centering
		\subfigure[]{\includegraphics[width=2.5cm]{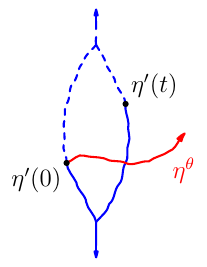}} \hspace{1cm} 
		\subfigure[]{\includegraphics[width=2.5cm]{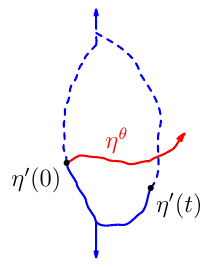}}
		\caption{In Proposition~\ref{prop:W}, the point $\eta'(t)$ may fall (a) left of the red flow line $\eta^\theta$ from $\eta'(0)$, or (b) right of it. \label{fig:twosides}}
	\end{figure}

	\begin{proposition}\label{prop:W}
	For $\theta\in (-\frac{\pi}2,\frac{\pi }2)$, 	let $\cA^\theta$ be \nocolor{the} set of times when $\eta'$ crosses $\eta^\theta$. 	Let $\scZ$ be the two-sided Brownian motion determined by $(\eta',\fh)$ as in  Theorem~\ref{thm:mating2}.
		For  $t>0$, let $\wt\tau_t=\sup\{s\le t: s\in \cA^\theta\}$ and $\scW_t=\scZ_t-\scZ_{\wt \tau_t}$.  We have
		\begin{enumerate}
			\item $\cA^\theta$ has the same distribution as the zero set of the standard Brownian motion. 
			\item For each $t>0$, $\scW|_{[0,t]}$ and $\scZ|_{[0,t]}$ almost surely determine each other. \nocolor{More precisely, they generate the same  $\sigma$-algebra modulo null events.} 
			\item Let  $\ell^\theta$ be the Brownian local time of $\cA^\theta$. There exists a \nocolor{deterministic} constant $c>0$ such that for all $t>0$, the $\ell^\theta$-local time accumulated on   $\{s\in [0,t]: \eta'(s)\in\eta^\theta\}$ a.s.\  equals $c$ times  the quantum length of $\eta^\theta\cap\eta'[0,t]$. \nocolor{(By~\cite[Lemma 3]{bipolarII},  $\eta^\theta\cap\eta'[0,t)$ is a segment of $\eta^\theta$.)}
		\end{enumerate} 
	\end{proposition} 
We will describe the  law of $\scW$ in Section~\ref{subsec:excursion}, for which we need the following lemma.
	\begin{lemma}\label{lem:p}
		Let $p(\theta)$ be the probability that $\eta'(1)$ lies on the right side of $\eta^\theta$, \nocolor{namely the region bounded by $\eta^\theta$ and the right frontier of $\eta'(-\infty,0]$}. Then $p$ is a homeomorphism between $(-\frac \pi 2,\frac\pi2)$ and $(0,1)$ and therefore has an inverse function $\theta(p)$. 		
		Moreover, $p(\theta)+p(-\theta)=1$.
	\end{lemma}
	\begin{proof}
		By scaling, $p(\theta)$ also \nocolor{equals} the probability that $\eta'(t)$ lies on the right side of $\eta^\theta$ for  \nocolor{any} $t>0$.  By Fubini's theorem, $p(\theta)$ \nocolor{equals}  the expected $\mu_\fh$-area of the sub-region of \nocolor{$\eta'[0,1]$} on the right side of $\eta^\theta$. Now the first statement  follows from the monotonicity of flow lines with respect to angle \cite{IGI} and the Monotone Convergence Theorem.  
		The second statement follows from symmetry \xin{since the law of the image of $(\eta',\eta^\theta,\mu_\fh)$ under the reflection $z\mapsto \bar z$ is  that of $(\eta',\eta^{-\theta},\mu_\fh)$.} 
	\end{proof}
	
	\subsection{A Poisson point process on half-plane Brownian excursions}\label{subsec:excursion}
	For \nocolor{$\alpha \in (-1,1)$, let $\scZ=(\scL,\scR)$ be a 
	two-dimensional Brownian motion with covariance matrix
	\begin{equation}\label{eq:covariance2}
\left(
\begin{array}{cc}
1 & \alpha \\ 
\alpha & 1
\end{array}
\right)
	\end{equation} In} 
	 this section we describe an excursion decomposition \xin{of $\scZ$}  which is the peanosphere counterpart of the  theory  in Section~\ref{subsec:singleflow}.  It is implicitly developed in \cite[Section 3]{bipolarII}, but we formulate it here via excursion theory. We will use this machinery to decompose $\cZ^\b$ into excursions away from the red flow line from the root. 
	This excursion description of the coupling of the blue Peano curve with red flow lines is a key tool in the proof of Theorem~\ref{thm:3pair}. 
	
% We will consider  and $p\in (0,1)$ is fixed. 
	
	For topological spaces $A$ and $B$, we denote by $C(A,B)$ the set of continuous functions from $A$ to $B$. Given $\omega\in C(\R_+,\R)$, we define the infimum process $I(\omega)$ by  $I(\omega)_t= \inf\{\omega_s:s\in[0,t]\}$. For $\ell\ge 0$, we define the hitting and crossing times $\tau^-(\omega)$ and $\tau(\omega)$ of $-\ell$:
 \begin{align}
	\tau^-(\omega)_{\ell}=\inf\{t: I(\omega)_t=-\ell \}\label{eq:inf} \quad \textrm{and}\quad 
	\tau(\omega)_{\ell}=\inf\{t: I(\omega)_t<-\ell \}.
	\end{align}
	 Suppose \nocolor{$(B_t)_{t\ge 0}$} is a standard \xin{one dimensional} Brownian motion. By Levy's theorem,then  $B-I(B)$ has the same law as the reflected Brownian motion $|B|$. \nocolor{Let $e_\ell(t) $ be $
	 	B_{\tau^-(B)_{\ell}+t}-B_{\tau^-(B)_{\ell}}$ if $0\le  t\le  \tau(B)_{\ell}-\tau^-(B)_{\ell}$ and 0  otherwise. Then $(e_\ell(t) \, : \ell \geq 0)$ is a Poisson point process. We denote its intensity by $ds\otimes \mathbf{n}$ where $ds$ is the Lebesgue measure on $[0,\infty)$ and $\mathbf n$ is an infinite measure on  $C(\R_+,\R)$. 
	 	In this construction $(B_t)_{t\ge 0}$, $B-I(B)$, and  $(e_\ell: \ell \ge 0)$ determine each other.   We call  $I(B)$ the local time process of  $B-I(B)$.
	 	We call $\mathbf n$ Ito's excursion measure for $B-I(B)$.  \xin{See e.g. \cite[Chapter 13]{Revuz-Yor} for more background on the excursion theory.} }
	 
%	\begin{enumerate}
%		\item $B-I(B)$ has the same law as a reflected Brownian motion, 
%		\item $I(B)$ is the local time process of $B-I(B)$, and 
%		\item $\tau(B)$ is the  inverse local time of $B-I(B)$, 
%	\end{enumerate}
%	where we use the convention that the local time at zero of $|B|$ is same as that of $B$. 
	
	Now \nocolor{let $\scZ^1=(\scL^1,\scR^1)$ and $\scZ^0=(\scL^0,\scR^0)$ be two independent copies of the 2D Brownian motion $\scZ$ with covariance matrix~\eqref{eq:covariance2}. Let}
	\begin{equation}\label{eq:e1}
	e^{1}_\ell(t)=\left\{
	\begin{array}{ll}
	\scZ^1_{\tau^-(\scL^1)_{\ell}+t}-\scZ^1_{\tau^-(\scL^1)_{\ell}},&\textrm{if}\; 0\le t\le  \tau(\scL^1)_{\ell}-\tau^-(\scL^1)_{\ell};\\
	0  & \textrm{if}\; t>\tau(\scL^1)_{\ell}-\tau^-(\scL^1)_{\ell}.
	\end{array}
	\right.
	\end{equation}
and
	\begin{equation}\label{eq:e0}
	e^{0}_\ell(t)=\left\{
	\begin{array}{ll}
	\scZ^0_{\tau^-(\scR^0)_{\ell}+t}-\scZ^0_{\tau^-(\scR^0)_{\ell}},&\textrm{if}\; 0\le t\le  \tau(\scR^0)_{\ell}-\tau^-(\scR^0)_{\ell};\\
	0  & \textrm{if}\; t>\tau(\scR^0)_{\ell}-\tau^-(\scR^0)_{\ell}.
	\end{array}
	\right.
	\end{equation}
	Then $(e^1_\ell(t) \, : \ell \geq 0)$ and $(e^0_\ell(t)\, : \, \ell \geq 0)$ are Poisson point processes with excursion spaces
	\begin{align*}
	E^1 &= \bigcup_{\zeta > 0}\big\{ e = (x,y) \in C([0,\zeta],\R^2) \, : \, e(0) = (0,0) \text{ and } x(\zeta) = 0 \text{ and } \left. x \right|_{(0,\zeta)} > 0\big\}, \text{ and} \\
	E^0 &= \bigcup_{\zeta > 0}\big\{ e = (x,y) \in C([0,\zeta],\R^2) \, : \, e(0) = (0,0) \text{ and }  y(\zeta) = 0 \text{ and } \left. y \right|_{(0,\zeta)} > 0\big\},
	\end{align*}
	respectively. \xin{In  words, $e^1$ (resp., $e^0$) is obtained from  $\scZ$ by keeping track of excursions in the right (resp., upper) half plane.}
As in the 1D case above,  $\scZ$ and  $e^1$ determine each other.  If we only keep the $x$-coordinate of excursions in  $(e^1_\ell(t) \, : \ell \geq 0)$, we get a Poisson point process with the same law as $(e_\ell: \ell \ge 0)$ above. The same holds for $e^0$.
\xin{In particular, the time set when the excursion occurs is distributed as the zero set of a 1D Brownian motion. Let $\mathbf n^1$	 and $\mathbf n^0$ be the excursion measures for $e^1$ and $e^0$, respectively. 	Then for all $t>0$,   $\mathbf
	n^1$ conditioned on $\zeta(e)>t$ equals the law of $\scZ$	conditioned on staying in the right half plane during $[0,t]$.  The same holds for   $\mathbf{n}^0$  with the upper half plane in place of the right half plane.
	This specifies the measures $\mathbf n^1$ and $\mathbf{n}^0$  up to a multiplicative constant.  We fix the constant by letting $\mathbf n^1(\zeta(e)>1)$ and $\mathbf n^0(\zeta(e)>1)$  equal to	the corresponding quantity for the 1D excursion measure. }

\nocolor{Now fix $p\in (0,1)$. For $\ell>0$, let $e^p{[0,\ell]}=e^1_{[0,p\ell]}\cup e^0_{[0,(1-p)\ell]}$.  \xin{In words, $e^p$ is obtained by taking the union of $e^1$ and $e^0$ with their intensity diluted by a factor of $p$ and $1-p$, respectively.}
Then $e^p$ is the Poisson point process on $[0,\infty) \times (E^1 \cup E^0)$ with intensity $ds\otimes \mathbf{n}^p$ where $\mathbf{n}^p=p \mathbf n^1+(1-p) \mathbf n^0$. Thus we obtain a coupling $(e^1,e^0,e^p,\scZ^1,\scZ^0)$, for which $(e^1,\scZ^1)$ and $(e^0,\scZ^0)$ are independent. Moreover, $(e^1,e^0)$ and $e^p$ determine each other.}

%By restricting to $E^1$ or $E^0$ we obtain two independent Poisson point processes which we denote by $pe^1$ and $(1-p)e^0$. Then $pe^1$ has intensity $p\,ds\otimes \mathbf{n}^1$  and $(1-p)e^0$ has intensity $(1-p)ds\otimes \mathbf{n}^0$. Also, $pe^1$ can be constructed from $e^1$ via reparametrizing time by $p\ell\mapsto\ell$, and similarly for $(1-p)e^0$ with the reparameterization $(1-p)\ell \mapsto \ell$. 
	
\nocolor{We now recall the construction of a \xin{two dimensional} Brownian motion $\scZ^p$ out of $e^p$ from~\cite{bipolarII}. In contrast to the construction of  $\scZ^1$ out of $e^1$, we rely on imaginary geometry to stitch the excursions of $e^p$ together to form a Brownian motion $\scZ^p$.}
	Given $e\in E^1\cup E^0$, let $f(e)$ be the $x$-coordinate process of $e$ if $e\in E^1$ and  the $y$-coordinate process of $e$ if $e\in E^0$. \nocolor{Then  $f(e^p):=(f(e^p_\ell)  \, : \ell \geq 0)$ matches in law with the excursion process for $B-I(B)$ above.} Therefore, we may define a reflected Brownian motion $\scX^p$ such that the excursion process of $\scX^p$ is $f(e^p)$ (here a function $f$ on an excursion space is understood to act pointwise on the corresponding Poisson point process).  For all $t>0$, let  
 \[
 \tilde{\tau}^p_t=\sup\{s:\nocolor{\scX^p_s}=0,s\le t\}, \quad \textrm{and}\quad  \tau^p_t=\inf\{s:\scX^p_s=0, s>t\}. 
 \]
	Let $(\ell^p_t\,: t \geq 0)$ be the local time process for $\scX^p$. \xin{Namely, $(\ell^p_t\,: t \geq 0)$ is the local time process of a 1D Brownian motion  $B$ such that $|B|=\scX^p$.} Then we can define a process $\scW^p$ by requiring that $\scW^p|_{[\tilde{\tau}^p_t,\tau^p_t]}$ evolves as $e^p_{\ell^p_t}$ for all $t \geq 0$; \nocolor{namely $\scW^p_{\tilde{\tau}^p_t+s}=e^p_{\ell^p_t}(s)$ for $s\in [0,\tau^p_t-\tilde{\tau}^p_t]$}. Clearly, $\scW^p$ is continuous on  $\bigcup_t\, (\tilde\tau^p_t,\tau^p_t)$, and $\tilde{\tau}_t^p$ is the last discontinuity time of $\scW^p$ before $t$, almost surely.  We now recall \cite[Proposition 3.2]{bipolarII}.
	\nocolor{\begin{proposition}[\cite{bipolarII}]
The law of  $\scW^p$ is the same as that of $\scW$ in Proposition~\ref{prop:W}  with $\theta=\theta(p)$, where $\theta(p)$ is as in  Lemma~\ref{lem:p}. 
	\end{proposition}}
\nocolor{By Proposition~\ref{prop:W} (2), $\scW^p$ determines a Brownian motion $\scZ^{p}$. Then $\scW^p$,  $\scZ^{p}$, and $e^p$  
  determine each other. We call $e^p$ the \emph{$p$-excursion} of $\scZ^p$.}  To summarize, we constructed a coupling of the following objects: $e^1$, $e^0$, $e^p$, $\scZ^1$, $\scZ^0$, $\scZ^p$, $ \scW^p$, $\scX^p$, and $\ell^p$. In this coupling,  $\scZ^1$ and $\scZ^0$ are independent, and the dependence relations, expressed in terms of $\sigma$-algebras, are as follows: 
\begin{align}\label{eq:relation}
& \sigma(e^1)=\sigma(\scZ^1),\sigma(e^0)=\sigma(\scZ^0);\\\nonumber
& \sigma(e^p)=\sigma(e^0,e^1)=\sigma(\scZ^p)=\sigma(\scW^p);\\\nonumber
&\sigma(\ell^p)\subset \sigma(|\scX^p|) \subset \sigma(\scZ^p).\nonumber
\end{align}

%
% 	Now we have a coupling $(\scZ^p,e^p)$ in which $\scZ^p$ and $e^p$ determine each other. Since $e^p$ is a functional of $\scZ^p$, 
%	We also define the two-dimensional Brownian motion $\scZ^1$ obtained by concatenating the excursions in $e^1$ and subtracting from the first coordinate its local time at zero. Define $\scZ^0$ similarly, concatenating $e^0$ and subtracting from the second coordinate its local time at zero.

	The following lemma gives a more explicit connection between $\scZ^1,\scZ^0,$ and $\scW^p$. 
	%and define $\wt \tau^1_a = \tau^-(\scL^1)_a$ and $\wt\tau^0_a=\tau^-(\scR^0)_a$ as in \eqref{eq:inf}.
	
	\begin{lemma}\label{lem:constructW}
	We write $\scZ^1=(\scL^1,\scR^1)$ and  $\scZ^0=(\scL^0,\scR^0)$. For $t\ge 0$, let 
	$\el_t=\sup\{a: \nocolor{\tau^-(\scL^1)_{pa}+ \tau^-(\scR^0)_{(1-p)a}}<t \}$. For each fixed $t\ge 0$, almost surely  $\ell^p_t= \el_t$
	hence	%In particular,  for $t>0$,
		\[
		\tilde\tau^p_t=\nocolor{\tau^-(\scL^1)_{p\el_t}+ \tau^-(\scR^0)_{(1-p)\el_t}} \quad \textrm{and}\quad \nocolor{\scW^p_{\tilde{\tau}^p_t+s}=e^p_{\el_t}(s) \quad \textrm{for } s\in [0,\tau^p_t-\tilde{\tau}^p_t]}.
		\]
		Moreover, \nocolor{for each  $t > 0$,  $e^p_{\ell_t}=e^1_{p\el_t}$  with probability $p$ and $e^p_{\ell_t} = e^0_{(1-p)\el_t}$ with probability $1-p$.}  
	\end{lemma}
	\begin{proof} 
		\nocolor{By definition, $\tilde\tau^p_t = \sum_{\ell < \ell^p_t} \zeta (e^p_\ell)=\sum_{\ell < p\ell^p_t} \zeta (e^1_\ell)+\sum_{\ell < (1-p)\ell^p_t} \zeta (e^0_\ell)$. Therefore $\tilde\tau^p_t =\tau^-(\scL^1)_{p\ell^p_t}+ \tau^-(\scR^0)_{(1-p)\ell^p_t}$. Since $\tilde\tau^p_t <t$, we have $\el_t\ge \ell^p_t $ a.s.}
		
		For a rational $\ell>\ell^p_t$, find another rational $\ell'\in (\ell^p_t,\ell)$. Then a.s. there is an excursion of  $e^p$ in $(\ell', \ell)$. Therefore $\nocolor{\tau^-(\scL^1)_{p\ell}+ \tau^-(\scR^0)_{(1-p)\ell}}>t$. 
		This means $\el_t\le \ell$ hence $\el_t\le \ell^p_t$ a.s.   
		
%		For a rational $\ell<\ell^t_p$, find another rational $\ell'\in (\ell,\ell^p_t)$.  Then a.s. there is an excursion of  $e^p$ in $(\ell', \ell)$.   Therefore $\wt\tau^1_{p\ell}+ \wt \tau^0_{(1-p)\ell}<t$. This means $\el_t\ge \ell$ hence $\el_t\ge \ell^p_t$ a.s. 
		
		Note that $e^p_{\ell^p_t}\in e^1_{p\ell^p_t}\cup e^0_{(1-p)\ell^p_t}$ a.s. By the definition of $e^p$, we have $e^p_{\ell^p_t}\in E^1$ with probability $p$. Therefore $e^p_{\ell_t}=e^1_{p\el_t}$  with probability $p$ and $e^p_{\ell_t} = e^0_{(1-p)\el_t}$ with probability $1-p$.  
	\end{proof}
	\begin{definition} \label{defn:delta} 
		For later reference, we define a random variable indicating the event described at the end of the statement of Lemma~\ref{lem:constructW}: $\delta(t) \colonequals 1$ if $e^p_{\ell_t} = e_{pl_t}^1$ and $\delta(t)\colonequals -1$ otherwise.  
	\end{definition}
	
	In Section~\ref{sec:3tree}, we will have a discrete analogue of the triple $(\scZ^1,\scZ^0,\scZ^p)$ for $\alpha=-\tfrac{\sqrt{2}}{2}$ and $p=\tfrac{\sqrt{2}}{1+\sqrt{2}}$ or $\tfrac{1}{1+\sqrt{2}}$  and establish weak convergence using a tightness plus uniqueness argument.
	\nocolor{The values of $p$ will be derived from the symmetry of the three trees in the Schnyder wood and the normalizing constants in the random walk scaling limit.} Lemma~\ref{lem:constructW} combined with the following characterization of the coupling $(\scZ^1,\scZ^0,\scZ^p)$ will help with the uniqueness part of the proof.  
	
	\begin{lemma}\label{lem:atomic}
\nocolor{Consider $(\scZ^1,\scZ^0,\scZ^p)$ as coupled above~\eqref{eq:relation}, where $\scZ^p$ is measurable with respect to $(\scZ^1,\scZ^0)$. Suppose in addition $\wt{\scZ}^p$ is coupled with $(\scZ^1,\scZ^0)$ such that} 
		\begin{enumerate}
			\item $\wt{\scZ}^p$ has the same law as $\scZ^0$.
			\item $\forall \, t>0$, $\wt \scZ^p_t-\wt \scZ^p_{\tilde{\tau}_t}=\scW^p_t$ a.s. where $\tilde{\tau}$ and $\scW^p$ are defined from $\scZ^0,\scZ^1$  as in Lemma~\ref{lem:constructW}.
			\item  \nocolor{Both $\scZ^p$ and $\wt \scZ^p$ are $\cF$-Brownian motions for   $\cF_t=\sigma((\scZ^p,\wt \scZ^p)|_{[0,t]})$.}
		\end{enumerate}
		Then $ \wt \scZ^p=\scZ^p$.
	\end{lemma}
	\begin{proof}
		Note that $\tilde{\tau}_t$ is c\'adl\'ag and $\scZ^p$ is continuous.  By considering rational $t$, we have that with probability 1,  $\scZ^p_t-\scZ^p_{\tilde{\tau}_t}=\scW^p_t$  for all $t$. Now Lemma~\ref{lem:atomic} is an immediate consequence of \cite[Lemma 3.17]{bipolarII}.	
	\end{proof}
	
	\section{Joint convergence to three Peano curves}\label{sec:peano}
	In Sections~\ref{subsec:1D}--\ref{sec:3tree}, we work in the infinite-volume setting. In the discrete, let $S^\infty$ be a UIWT \nocolor{as in Definition~\ref{defn:UIWT}.}
	Let $(Z^\b,Z^\r,Z^\g)$ be the triple of lattice walks encoding of $S$ with respect to the clockwise explorations of its three spanning trees.  	
	In the continuum, recall the notations of Section~\ref{subsec:lqg}. We
	assume $\kappa=\gamma=1$ and $\kappa'=16$. 
	\nocolor{Let $\fh$ be a field describing the $1$-LQG quantum cone. Let $h$ be a whole plane GFF modulo $2\pi\chi$ with $\chi=\frac{2}{\sqrt\kappa} - \frac{\sqrt\kappa}{2}$. Let $(\eta^\b,\eta^\r,\eta^\g)$ be the Peano curves for $h$ of angle $(0,\tfrac{2\pi}{3},\tfrac{4\pi}{3})$ as defined below~\eqref{eq:ODE}.}  For $\mathrm{c}\in\{\b,\r,\g\}$, let \nocolor{$\scZ^{\mathrm{c},\infty}$} be the Brownian motion corresponding to $(\mu_\fh,\eta^{\mathrm{c}})$ as descibed in Theorem~\ref{thm:mating2}. In Sections~\ref{subsec:1D}--\ref{sec:3tree} we will  prove the following Theorem~\ref{thm:main1}, which infinite-volume version of Theorem~\ref{thm:3pair}. In the statement of Theorem~\ref{thm:main1} and throughout Sections~\ref{subsec:1D}--\ref{sec:3tree}, we we drop the supersript $\infty$ to simplify the notion. \nocolor{Namely, we write  the UIWT $S^\infty$ as $S$, and write $(\scZ^{\mathrm{c},\infty})_{t\in \R}$ as $(\scZ^{\mathrm{c}})_{t\in\R}$. } 
	\begin{theorem}\label{thm:main1}
		In the above setting, write $\Zb=(L^\b,R^\b),\Zr=(L^\r,R^\r),\Zg=(L^\g,R^\g)$. Then 
		\[
		\left(\tfrac{1}{\sqrt{4n}}L^{\mathrm{c}}_{\lfloor 3nt\rfloor}, 
		\tfrac{1}{\sqrt{2n}}R^{\mathrm c}_{\lfloor 3nt\rfloor}\right)_{t\in \R} 
		\stackrel{\text{in law}}{\longrightarrow} (
		\mathscr{Z}^{\mathrm{c}})_{t\in \R} \quad \textrm{for}\;\; \mathrm c\in \{\b,\r,\g\}.
		\] 
		and the convergence holds jointly for $(\Zb, \Zr, \Zg)$.
	\end{theorem}
	\nocolor{
	\begin{remark}[Abuse of notation in the infinite volume setting]\label{rmk:notation}
		Although $\scZ^{\mathrm{c}}$ was also used  in Theorem~\ref{thm:3pair} to denote the \xin{limiting} Brownian excursions from the finite volume setting, we believe this abuse of notation does not cause confusion since we will exclusively focus on the infinite volume setting in Sections~\ref{subsec:1D}--\ref{sec:3tree}. We will use  $Z^{\mathrm c,n}$ to denote the rescaled walks in Theorem~\ref{thm:main1}. \xin{Although}  this notation is \xin{used} in Section~\ref{sec:bijection} and~\ref{sec:pair} to denote the unscaled walk in the finite volume case, for the same reason we believe this abuse of notation adds simplicity without causing confusion. 
	\end{remark}} 
	
	In Section~\ref{sec:finite}, we use a conditioning argument to deduce Theorems~\ref{thm:3pair} and~\ref{thm:embedding} from Theorem~\ref{thm:main1}. \nocolor{We now give an overview for the proof of  Theorem~\ref{thm:main1}, which occupies Sections~\ref{subsec:1D}--\ref{sec:3tree}. From 
	Proposition~\ref{prop:reverse} we already have the marginal convergence in Theorem~\ref{thm:main1}, which in particular implies the tightness of the three walks. It suffices to show that in the subsequential limit the three Brownian motions are coupled as desired. From Section~\ref{subsec:excursion}, we know that  the second Brownian motion should  be encoded by the first one through the local time for the left/right excursions  of the first Peano curve, away  flow lines coming from the frontier of the second Peano curve. In Section~\ref{subsec:1D} we show that this local time encoding of quantum lengths for flow lines has an exact discrete analog for UIWT. In Section~\ref{sec:triple}, 
we focus on one flow line coupled with the first Peano curve, and show that the random walks and local time processes from UIWT \xin{converge} to their continuous counterpart as in Proposition~\ref{lem:atomic}. In Section~\ref{sec:3tree}  we put flow lines together to get the second Peano curve to show the joint convergence of the first two random walks. By symmetry we get the joint convergence of the three walks.
}

	\subsection{A decomposition of a one-dimensional  random
		walk}\label{subsec:1D}
	%Suppose we are in the setting of Theorem~\ref{thm:mating2}. 
	Let $w$ be the stationary bi-infinite word associated with the \nocolor{UIWT} $S$. For
	$i\in \Z$, let $\fraketa^\b(i)$ be the edge in $S$ corresponding
	to $w_i$. \nocolor{Then $(\fraketa^b(i))_{i\in \Z}$}  can be thought of as the Peano curve exploring
	the blue tree of $S$. For any forward stopping time $T$ such
	that $w_T=\b$ a.s., recall the grouped-step walk $\cZ$ defined
	in Lemma~\ref{lem:forward}.  We drop the dependence of $\cZ$ on
	$T$ since the law of $\cZ$ (namely $\P^\infty$) is independent
	of $T$ (Lemma~\ref{lem:forward}).
	\begin{figure}[h]
		\centering
		\subfigure[]{\includegraphics[width=0.4\textwidth]{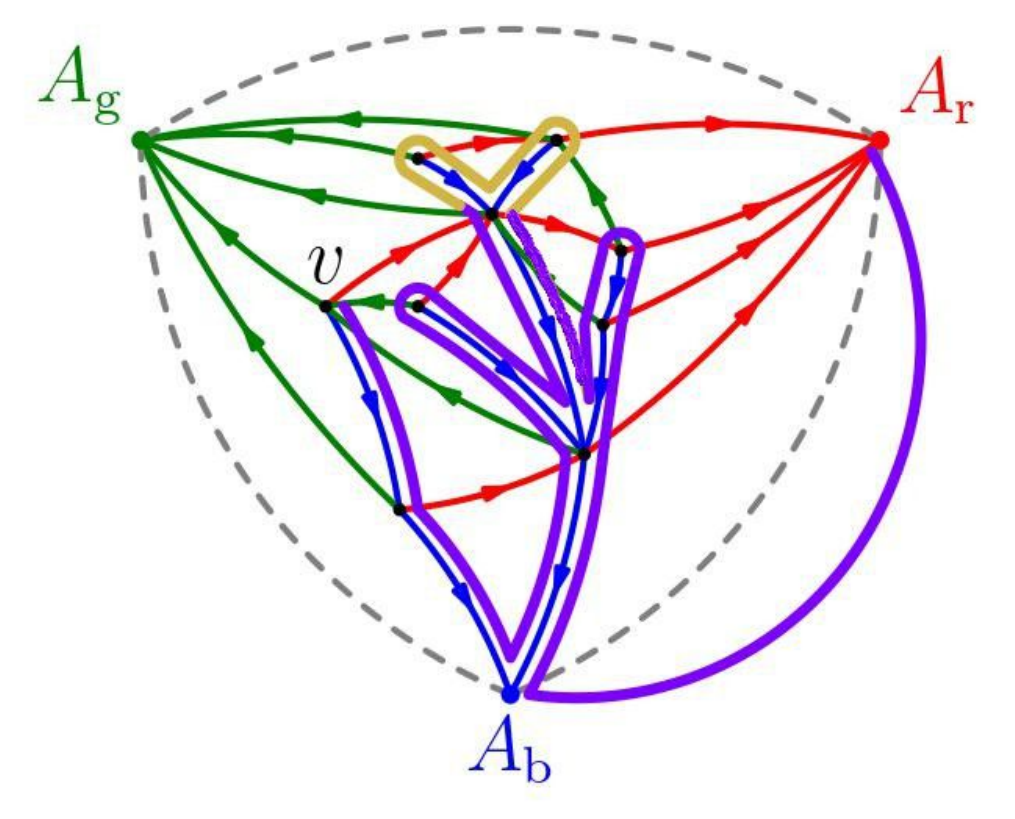}} \quad
		\subfigure[]{\includegraphics[width=0.3\textwidth]{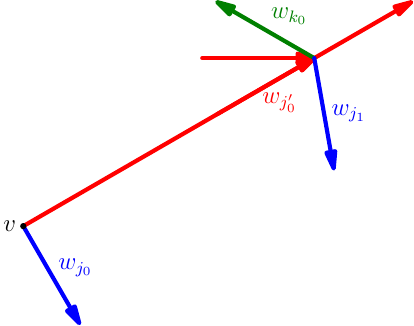}}
		\caption{(a) A decomposition of the path $\cP(S)$ into excursions
			right and left of the red flow line from a vertex $v$, shown in
			purple and gold, respectively. The corresponding
			decomposition of the word is as follows: $
			\color{gray}{\g\g\b}
			{\color{dpurple}\underline{\b\r\g\g\b\r\r}} \:
			{\color{gold}\underline{\g\g\b\r\g\b}} \:
			{\color{dpurple}\underline{\b\g\r}}\:{\color{gold}\underline{\g}}\:{\color{dpurple}\underline{\b\b\b\r\r\r}}\r
			$. (b) An illustration of the symbols involved in the excursion
			decomposition of the word: $(w_{j_0},w_{j_0'})$ is a br match, $w_{k_0}$ is the first $g$ after $w_{j_0'}$ and
			$(w_{k_0},w_{j_1})$ is a gb match.  
			\label{fig:flow_line_excursions}}
	\end{figure}
	
	We now develop wooded-triangulation analogues of the objects featured
	in Sections~\ref{subsec:singleflow}
	and~\ref{subsec:excursion}. Similar discrete analogues have been
	considered in the bipolar orientation setting \cite{bipolarII}. The
	forward and reverse explorations of a bipolar orientation are
	symmetric. However, they behave differently in the wooded
	triangulation setting considered below.  We start from the forward
	exploration. Given a blue edge $\fraketa^\b(i)$ and its tail $v$, let
	$P_i$ be the red flow line from $v$.  As a ray in a topological half
	plane bounded by the blue and green flow lines from $v$, the path
	$P_i$ admits a natural notion of left and right sides. 
	
	\begin{lemma} \label{lem:leftrightexcursions} In the above setting,
		let $T$ be a forward stopping time such that $w_T=\b$ a.s. and
		let $\cZ=(\cL,\cR)$ be the forward grouped-step walk viewed from $T$.
		Define a sequence $\sigma$ inductively by $\sigma_0 = 0$ and, for
		$i \geq 1$,
		\begin{align*}
		\sigma_{i} &= \inf \{j > \sigma_{i-1} \, : \, \cL_j <
		\cL_{\sigma_{i-1}}\} , \text{ if } i \text{ is
			odd,} \\
		\sigma_{i} &= \inf \{j > \sigma_{i-1} \, : \, \cR_j <
		\cR_{\sigma_{i-1}}\}, \text{ if } i \text{ is
			even,} 
		\end{align*}
		Then for all integers $k\in [\ta(\sigma_{i}),
		\ta(\sigma_{i+1}-1)]$, the edge $\fraketa^{\b}(k)$ is right of
		the red flow line $P_T$ if $i$ is even and left of $P_T$ if
		$i$ is odd.
	\end{lemma}
	
	See Figure~\ref{fig:timecorrespondence} for an illustration of the
	relationship between the excursion index $i$, the  time
	index $j$ in $\cZ$, and  index $k$ \nocolor{considered in Lemma~\ref{lem:leftrightexcursions}}. 
	\begin{figure}[h]
		\centering
		\fbox{\includegraphics[width=0.55\textwidth]{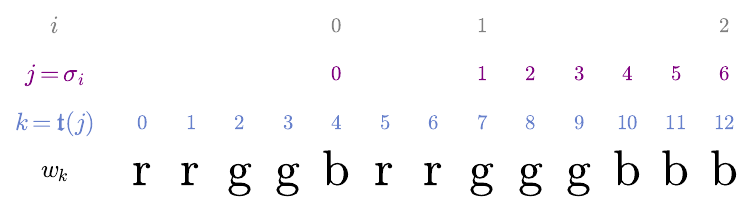}}
		\caption{Three indices involved in Lemma~\ref{lem:leftrightexcursions} \nocolor{illustrated for a segment of a word}.
			\label{fig:timecorrespondence}}
	\end{figure}
	\begin{proof} 
		Beginning with $j_0 \colonequals T$ (see Figure~\ref{fig:flow_line_excursions}(b)), we perform the
		following algorithm :
		\begin{enumerate} 
			\item Identify $w_{j_0}$'s br match $(w_{j_0},w_{j_0'})$;
			\item Let $k_0=\min\{k>j_0':w_k=g\}$;
			\item Find {$w_{k_0}$}'s gb match $({w_{k_0}},w_{j_1})$;
			\item Repeat, beginning  with $j_1$ in place of $j_0$ to obtain $\{(j_i,j'_i, k_i)\}_{i\ge 0}$ iteratively.
		\end{enumerate} 
		According to Definition \ref{def:graph_from_word}, the red edges
		$\fraketa^{\b}(j'_0),\fraketa^{\b}(j'_1)\cdots$ form the red flow line $P_T$. Moreover, $\fraketa^{\b}$
		is right (resp. left) of  $P_T$ during $ [j_i, j'_i-1]$
		(resp. $[j'_{i}+1, j_{i+1}-1]$) for all $i\geq 0$.  Recall the sequence $\{\ta(k)\}_{k\ge 0}$ in Lemma~\ref{lem:forward}. By the parenthesis
		matching rule,  we have $j_i=\ta(\sigma_{2i})$ and
		$k_i=\ta(\sigma_{2i+1})$ for all $i\geq 0$. Therefore $w_{\ta(\sigma_{2i})}=\b$ and thus 
		$\ta(\sigma_{2i}-1)=\ta(\sigma_{2i})-1$.
		Since
		$\ta(\sigma_{2i+1}-1) < j'_i < k_i$, we conclude the proof.
	\end{proof}
	
	\begin{figure} 
		\centering
		\subfigure[]{\includegraphics[width=0.37\textwidth]{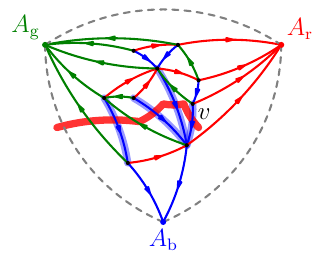}}
		\subfigure[]{\includegraphics[width=0.37\textwidth]{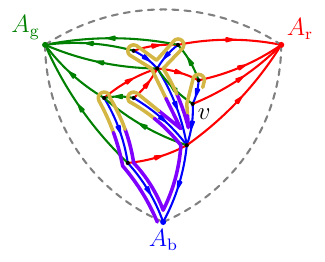}}
		\caption{ (a) The dual red flow line from $v$, which is obtained by
			starting from the face left of the outgoing blue edge from $v$. The corresponding sequence of blue edges is
			illustrated with the thick blue marker, and the path passing
			through all these edges is shown in red.
			(b) A  decomposition according to the left and right excursions from the dual red flow line.		\label{fig:dual_flow_line_excursions}}
	\end{figure}
	
	\begin{definition}  \label{defn:leftrightexcursions}
		For each $l \geq 0$, we call the walk 
		$(\cZ_{k} - \cZ_{\sigma_{2l}})_{\sigma_{2l} \leq k \leq
			\sigma_{2l+1}}$ a \textbf{right
			excursion} and the walk 
		$(\cZ_{k} - \cZ_{\sigma_{2l+1}})_{\sigma_{2l+1} \leq k \leq
			\sigma_{2l+2}}$ a \textbf{left
			excursion}.  
		
		Let $\nocolor{\cT^1(k)}=\sum_{0\le l<k} (\sigma_{2l+1}-\sigma_{2l})$ and
		$\nocolor{\cT^0(j)}=\sum_{0\le l<j} (\sigma_{2l+2}-\sigma_{2l+1})$.  Let $\nocolor{\cZ^1=(\cL^1,\cR^1)}$ be the concatenation of right
		excursions. More precisely, let $\cZ^1_0=(0,0)$. For $l\ge 0$ and $0\le i\le \sigma_{2l+1}-\sigma_{2l}$,  let 
		$\nocolor{\cZ^1_{i+\cT^1(l)}-\cZ^1_{\cT^1(l)}}=\cZ_{i+\sigma_{2l}}-\cZ_{\sigma_{2l}}$.  Let
		$\nocolor{\cZ^0=(\cL^0,\cR^0)}$ be the concatenation of left excursions defined
		in the same way.
		
		When
		$\sigma_{2l}\le n<\sigma_{2l+1}$ for some $l\ge 0$, let
		$(\cK_n,\cJ_n)=(l,l)$.  When $\sigma_{2l+1}\le n<\sigma_{2l+2}$ for
		some $l\ge 0$, let $(\cK_n,\cJ_n)=(l,l+1)$. We define the {\bf
			discrete local time} of $\cZ$ to be
		\[
		\mathfrak{l}_n= - \inf_{m\leq \cT^{1}(\cJ_n)} \cL^{1}_m - \inf_{m\leq {\cT^{0}(\cK_n)}} \cR^{0}_m. 
		\]
	\end{definition}
	By the strong Markov property, $\cZ^{1}$ and $\cZ^{0}$ are independent
	and equal in distribution to $\cZ$.  The process
	$(\cZ,\cZ^0,\cZ^1,\mathfrak{l})$ is a discrete analogue of
	$(\scZ^p,\scZ^0,\scZ^1,\ell^p)$ in
	Lemma~\ref{lem:constructW}. \nocolor{(In fact, we will show in Section~\ref{sec:triple} that the scaling limit of the former is the latter with $p=\frac{\sqrt{2}}{1+\sqrt2}$.)}
	Moreover, $\cJ_n-\cK_n$ is the indicator
	of the event that $\cZ$ reaches time $n$ during a left excursion,
	which is analogous to $\delta$ in Definition~\ref{defn:delta}.

	We have a similar story when considering the time reversal of $\fraketa^\b$ and  dual red flow lines in $\cM_{\b\r}$ (recall Section~\ref{subsec:infinite}).  Recall $Z^\b=(L^\b,R^\b)$.
	\begin{definition}\label{def:reverse}
		Let $T$ be a reverse stopping time of $w$, (that is, a stopping time
		with respect to the filtration $\cF_n=\sigma(\{w_i\}_{i\ge n})$)
		such that $w_T=\b$ a.s. Let $\overline{\ta}(0)=T-1$ and for all $k\ge 1$, let
		$\overline{\ta}(k)=\max\{i<\overline{\ta}(k-1): w_i\neq r \}$. Let
		$\ol\cL_k=R^\b_{\overline{\ta}(k)} - R^\b_{\overline{\ta}(0)}$ and
		$\ol\cR_k=L^\b_{\overline{\ta}(k)} - L^\b_{\overline{\ta}(0)}$ for all $k\ge 0$.  We call
		$\ol\cZ:=(\ol\cL,\ol\cR)$ the {\bf reverse grouped-step walk} 
		$Z^{\b}$  \bf{viewed from} $T$.
	\end{definition}
	Note the interchange of $L$ and $R$, owing to the fact that reversing
	$\fraketa^\b$ swaps the notion of left and right (see the precise
	convention below).  We drop the dependence of $\ol\cZ$ on $T$ because
	the following lemma implies that the distribution of $\ol\cZ$ is
	independent of $T$.
	\begin{lemma}\label{lem:reverse}
		Let $\ol\P^\infty$ be the law of a random walk with i.i.d.\ 
		increments distributed as
		$\tfrac{1}{2} \delta_{(1,-1)} + \sum_{i=0}^\infty 2^{-i-2}
		\delta_{(-1,i)}$ (see Figure~\ref{fig:steps}(c)). Then
		$\ol\cZ$ is distributed as $\ol\P^\infty$.
	\end{lemma}
	\begin{proof} 
		We have $w_{\overline{\ta}(0)}\neq \r$ a.s. because $w_T=\b$. If $w_{\overline{\ta}(0)}=\b$ we
		have $\ol \cZ_1=(1,-1)$.  If $w_{\overline{\ta}(0)}=\g$, then $\ol \cL_1=-1$
		and $\ol\cR_1=\overline{\ta}(0)-\overline{\ta}(1)-1$. It is straightforward to check that the first
		increment of $\ol\cZ$ has law indicated in Figure~\ref{fig:steps}(c). A similar argument
		can be used to check the joint distribution of
		$\{\ol\cZ_k-\ol\cZ_{k-1}\}_{k\ge 1}$.
	\end{proof}
	The increments $\{(\eta^x_k,\eta^y_k)\}_{k\ge 0}$ of $\ol\P^{\infty}$
	satisfy $\E [(\eta^x_k,\eta^y_k)] = (0,0)$, $\E(|\eta^x_k|^2)=1$,
	$\E(|\eta^y_k|^2)=2$, and $\mathrm{cov}(\eta^x_k,\eta^y_k)=-1$
	(compare to Remark~\ref{rmk:step}).
	
	Given a blue edge $\fraketa^\b(i)$, its tail $v$ and the face
	$f=\cF(\fraketa^\b(i))$ (recall Lemma~\ref{lem:face_structure}), let
	$P'_i$ be the dual red flow line starting from $f$. Since $P'_i$ lives
	on the dual map of $\cM_{\b\r}$, we have to be  careful when
	talking about the left and right of $P'_i$.  Although particular
	choices do not matter in the scaling limit, for the sake of
	concreteness we fix the following convention.  Given $P'_i$,
	define the edge subset $\ol P_i$ of $S$ as follows: an edge $e$ is in
	$\ol P_i$ iff there is a blue edge $e'$ such that $\cF(e')$ is on
	$P'_i$ and $e$ is the next edge after $e'$ when clockwise rotating
	around the tail of $e$.  By Lemma~\ref{lem:face_structure}, $\ol P_i$
	is a simple path on $S$. As a ray in a topological half plane bounded
	by the blue and red flow lines from $v$, we have a natural notion of
	left and right for $\ol P_i$. For $k\le 0$, we say $\fraketa^\b(k)$ is
	on the {\bf right} of $P'_i$ if $\fraketa^\b(k)$ lies on $\ol P_i$ or
	to its right.  Otherwise we say $\fraketa^\b(k)$ is on the {\bf left}
	of $P'_i$.  We have the following analogue of
	Lemma~\ref{lem:leftrightexcursions}.
	\begin{figure}
		\centering 
		\subfigure[]{\includegraphics[height=3.5cm]{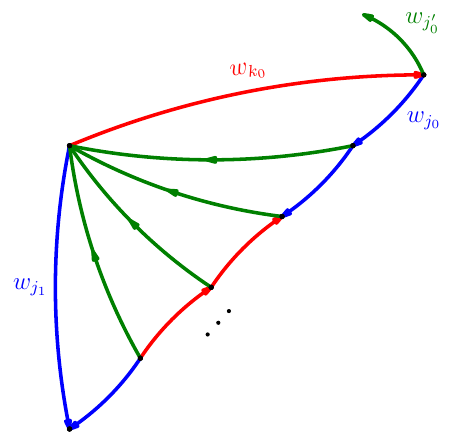}}
		\subfigure[]{\includegraphics[height=3.5cm]{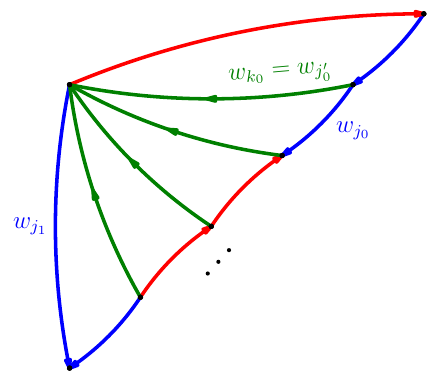}}
		\subfigure[]{\includegraphics[height=3.5cm]{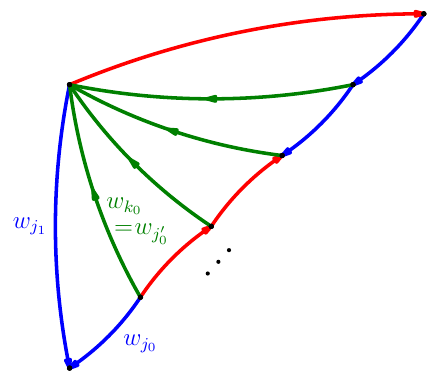}}
		\caption{An illustration of one step of the iterative procedure in the proof of 
			Lemma~\ref{lem:dualleftrightexcursions}. We have the case 
			$w_{j_0'-1}=\r$ in (a), the case $w_{j_0'-1}=\g$ in (b), and the
			case   $w_{j_0'-1}=\b$ in (c). In all three cases our proof of Lemma~\ref{lem:dualleftrightexcursions} works.
			\label{fig:twofacecases}}
	\end{figure}

	\begin{lemma} 
		\label{lem:dualleftrightexcursions} 
		In the above setting,
		let $T$ be a reverse stopping time such that $w_T = \b$ almost
		surely, and let $\ol \cZ$ be the reverse grouped-step walk viewed
		from $T$. Define a sequence $\dsigma$
		inductively by $\dsigma_0 = 0$ and, for $i \geq 1$, 
		\begin{align*}
		\dsigma_{i} &= \inf \{j > \dsigma_{i-1} \, : \, \overline{\cL}_j <
		\overline{\cL}_{\dsigma_{i-1}}\} \text{ if } i \text{ is
			odd,} \\
		\dsigma_{i} &= \inf \{j > \dsigma_{i-1} \, : \, \overline{\cR}_j <
		\overline{\cR}_{\dsigma_{i-1}}\} \text{ if } i \text{ is
			even}. 
		\end{align*}
		Then for each $k \in [\overline{\ta}(\sigma_{i+1})+1,\overline{\ta}(\sigma_{i})]$, the edge $\fraketa^{\b}(k)$ is right of the dual red flow line $P'_T$ if $i$ is even and  left of $P'_T$
		if $i$ is odd.
	\end{lemma}
	\begin{proof} 
		%Given a vertex $v$, the edges $e_1,e_2,\ldots$ crossed by the  dual red flow line from $v$  can constructed inductively via: 
		%\begin{enumerate} 
		%	\item[(i)] $e_1$ is the outgoing
		%	blue edge from $v$, and 
		%	\item[(ii)] for $i \geq 1$, denote by $f_i$ be the face on the right of $e_i$ (that is, as you
		%	move along $e_i$ in the direction of its orientation, select the face
		%	on your right). Then we define $e_{i+1}$ to be the unique blue edge
		%	with $f_i$ on its left (the existence of which is ensured by Lemma~\ref{lem:face_structure}). 
		%\end{enumerate}
		Beginning with $j_0 \colonequals T$ (see
		Figure~\ref{fig:dual_flow_line_excursions}(c)), we perform the following algorithm.
		\begin{enumerate} 
			\item Identify $w_{j_0}$'s gb match $(w_{j_0'},w_{j_0})$
			\item  Let $k_0=\max\{k<j'_0: w_k\neq r\}$+1.
			\item Let $j_1=\max\{j<k_0: w_j=\b,\;\textrm{and the
				br match $(w_j,w_\ell)$ satisfies}\; \ell \ge k_0 \}$. 
			\item Repeat, beginning  with $j_1$ in place of $j_0$ to obtain $\{(j_i, k_i)\}_{i\ge 0}$ iteratively.
		\end{enumerate} 
		By  Definition~\ref{def:graph_from_word} and Lemma~\ref{lem:face_structure}, edges of $\ol P_T$ are exactly $\{\fraketa^\b(k_i)\}_{i\ge 0}$ in order. Moreover, for $i\ge 0$, the edge $\fraketa^\b(k)$ is on the right of $P'_T$ (according to the convention above) if $k_i\le k<j_i$ and on the left of $P'_T$ if $j_{i+1} \le k<k_i $. 
		Note that $Z^\b_k-Z^\b_{k-1}=f(w_k)$. It can be checked by induction that $\bar \ta (\sigma_{2i+1}) =k_i-1$ and $\bar{\ta}(\sigma_{2i}) =j_i-1$ for all $i\ge 0$. 
		This concludes the proof.
		%The parenthesis matching rule implies by induction that
		%$\overline{\ta}(\sigma_{2l}) =
		%{j_l}$ for all $l \geq 0$ and that $\fraketa^\b(\sigma_{2l+1})$ is a
		%green edge whose tail is equal to that of $\fraketa^\b(\sigma_{2l})$,
		%for all $l \geq 0$. Moreover, we see that $\fraketa^{\b}(k)$ is right of the flow
		%line for $j_l < k \leq  j_l'$ and left of the flow line  for $k_{l}'
		%\leq k < j_{l+1}$, which concludes the proof. 
		% Since each edge in the dual red flow line may be reached from the
		% previous one by making two of the above excursions (first right, then
		% left), it follows that $\{\fraketa^{\b}(i) \,:\, i \in 2\Z\}$ is the
		% dual red flow line. 
	\end{proof}
	
	%We extend $\sigma$ to a doubly-infinite sequence by setting
	%$\sigma_{-k} = \overline{\sigma}_k$ for all $k > 0$. 
	
	% \begin{definition}\label{def:dual}
	% Let $\ol \cL_i=: \cR_{-i-1}-\cR_i$ and $\ol \cR_i=: \cL_{-i-1}-\cL_i$. Let  $\ol\cZ=:(\ol\cL,\ol\cR)$. For $i\ge 1$, let
	%   \begin{align*}
	%   \ol\sigma_{2i-1} = \inf \{t > \sigma_{2i-2} \, : \, \ol\cL_t <
	%   \ol\cL_{\sigma_{2i-2}}\} \quad \textrm{and}\quad
	%   \ol\sigma_{2i} = \inf \{t > \sigma_{2i-1} \, : \, \ol\cR_t <
	%   \ol\cR_{\sigma_{2i-1}}\}. 
	%   \end{align*}
	%  Now we define $(\ol\cZ,\ol\cZ^0,\ol\cZ^1,\ol{\mathfrak l})$ exactly in the same way as $(\cZ,\cZ^0,\cZ^1,{\mathfrak l})$ in Definition~\ref{defn:leftrightexcursions}. 
	% \end{definition}
	% The reason we swap $\cL,\cR$ in Definition~\ref{def:dual} is because
	% if in order to discover the dual flow line we reverse the exploration
	% of $\eta^\b$, which means  left and right are swapped.

	We can define $(\ol\cZ,\ol\cZ^0,\ol\cZ^1,\ol{\mathfrak l})$  in the same way as $(\cZ,\cZ^0, \cZ^1,{\mathfrak l})$.   The process $(\ol\cZ,\ol\cZ^0,\ol\cZ^1,\ol{\mathfrak l})$ is another
	discrete analogue of $(\scZ^p,\scZ^0,\scZ^1,\ell^p)$ in
	Lemma~\ref{lem:constructW}.
	Although we need to analyze both $\cZ$ and
	$\ol\cZ$, the arguments will be almost identical. And as explained in
	Remark~\ref{rmk:overshoot} below, $\cZ$ offers slightly more technical
	challenges. So in Section~\ref{sec:triple} and~\ref{sec:3tree} we
	treat $\cZ$ in detail and $\ol\cZ$ in brief.
	\begin{remark}\label{rmk:overshoot}
		The right  excursions of the flow line exhibit an \textit{overshoot}
		phenomenon. More precisely, let 
		$\Delta_i=\cL_{\sigma_{2i-2}}-\cL_{\sigma_{2i-1}}$ for $i\ge 1$. Then \nocolor{by the law $\P^\infty$ of $\cZ$ described} in Lemma~\ref{lem:forward}, $\{\Delta_i\}_{i\ge 1}$ are i.i.d. with distribution $\Geo$. Meanwhile, the left  excursions of the flow line and \textit{both} left
		and right  excursions of the dual flow line always hit $-1$ deterministically.
	\end{remark}
	
	Finally, we note that although the discussion above is in the
	setting where we view the UIWT $S$ from a blue edge
	$e=\eta^\b(T)$, where $T$ is a forward or reverse stopping time,
	the combinatorial results are the same if we instead view $S$
	from an arbitrary blue edge $e$ on the UIWT or on a finite
	wooded triangulation.
	
	\subsection{Joint convergence of a pair of trees and a flow
		line}\label{sec:triple}
	Suppose we are in the setting of Section~\ref{subsec:1D}. Throughout this subsection, let the forward stopping time $T$ in Lemma~\ref{lem:leftrightexcursions} be $\min\{i\ge 0:w_i=b\}$ 
	and the reverse stopping time $T$ in Lemma~\ref{lem:dualleftrightexcursions} be $\max\{i\le 0:w_i=b\}$.
	Recall  Definition~\ref{defn:leftrightexcursions}. Let $(\cL^n_t)_{t\ge 0}$ and
	$(\cR^n_t)_{t\ge 0}$ be the linear interpolation of
	$\left(\frac 1{\sqrt{4n}} \cL_{2nt}\right)_{t\in \frac{1}{2n}\N}$ and
	$\left(\frac 1{\sqrt{2n}} \cR_{2nt}\right)_{t\in \frac{1}{2n}\N}$,
	respectively, and set $\cZ^{n}_t=(\cL^n_t, \cR^n_t)$ for $t \ge 0$.
	We  define $\cZ^{0,n}_t = (\cL^{0,n}_t,\cR^{0,n}_t)$ and
	$\cZ^{1,n}_t = (\cL^{1,n}_t,\cR^{1,n}_t)$ in the same way with $\cZ^0$
	and $\cZ^1$ in place of $\cZ$.
	Let
	\begin{align*}
	&J^n_t=\cJ_{\lfloor2nt\rfloor}\quad \textrm{and}\quad K^n_t=\cK_{\lfloor2nt\rfloor},\\
	&T^{1,n}(j)=(2n)^{-1}\cT^1(j)\quad \textrm{and}\quad T^{0,n}(k)=(2n)^{-1}\cT^0(k).
	\end{align*}
	Let the rescaled local time  of $\cZ^{n}$ to be
	
	\begin{equation}\label{eq:local}
	\mathfrak{l}^{\,n}_t = - \inf_{s \leq T^{1,n}(J_t)} \cL^{1,n}_s - \inf_{s \leq {T^{0,n}(K_t)}} \cR^{0,n}_s \quad\quad \forall \; t\ge 0.
	\end{equation}
	
	We may also define the rescaled version of
	$(\ol\cZ,\ol \cZ^1,\ol \cZ^0,\ol{\mathfrak{l}})$, denoted by
	$(\ol\cZ^n,\ol \cZ^{1,n},\ol \cZ^{0,n},\ol{\mathfrak{l}}^n)$. Since we
	swap $L$ and $R$ in the definition of $\ol\cZ$, the scaling of the
	abscissa and ordinate is also swapped:
	$\ol\cZ^n_t=\left(\frac 1{\sqrt{2n}} \ol\cL_{2nt},\frac 1{\sqrt{4n}}
	\ol\cR_{2nt}\right)$. We also adapt the scaling for
	$\ol\cZ^{1,n},\ol\cZ^{0,n},\ol{\mathfrak{l}}^n$ accordingly.
	
	\begin{proposition}\label{prop:joint}
		In the above setting, the triple
		$ \left(\tfrac{1}{\sqrt{4n}}L^\b_{\lfloor 3nt\rfloor},
		\tfrac{1}{\sqrt{2n}}R^\b_{\lfloor 3nt\rfloor}\right)_{t\in \R},
		(\cZ^n_t)_{t\ge 0}, $ and $(\ol\cZ^n_t)_{t\ge 0}$ jointly converge
		in law. Write $\scZ^\b=(\scL^\b,\scR^\b)$ in
		Theorem~\ref{thm:main1}. Then the joint limit of the triple is
		$(\scZ^\b_t)_{t\in\R}$, $(\scZ^\b_t)_{t\ge 0}$ and
		$(\scR^\b_{-t},\scL^\b_{-t})_{t\ge 0}$.
	\end{proposition}
\nocolor{\begin{proof}
By Proposition~\ref{prop:reverse},  	$ \left(\tfrac{1}{\sqrt{4n}}L^\b_{\lfloor 3nt\rfloor},
\tfrac{1}{\sqrt{2n}}R^\b_{\lfloor 3nt\rfloor}\right)_{t\in \R}$ converges to $(\scZ^\b_t)_{t\in\R}$.
By the proof of Proposition~\ref{prop:reverse} based on Lemma~\ref{lem:mc_concentration}, $ \left(\tfrac{1}{\sqrt{4n}}L^\b_{\lfloor 3nt\rfloor},
\tfrac{1}{\sqrt{2n}}R^\b_{\lfloor 3nt\rfloor}\right)_{t\in \R}$ and $
(\cZ^n_t)_{t\ge 0}$ jointly converge to  	$(\scZ^\b_t)_{t\in\R}$ and  $(\scZ^\b_t)_{t\ge 0}$.
The same concentration in Lemma~\ref{lem:mc_concentration} also applies to the reverse
grouped-step walks, which shows that  $ \left(\tfrac{1}{\sqrt{4n}}L^\b_{\lfloor 3nt\rfloor},
\tfrac{1}{\sqrt{2n}}R^\b_{\lfloor 3nt\rfloor}\right)_{t\in \R}$ and $(\ol\cZ^n_t)_{t\ge 0}$ jointly converge to $(\scZ^\b_t)_{t\in\R}$ and $(\scR^\b_{-t},\scL^\b_{-t})_{t\ge 0}$.
\end{proof}}

	The key to the proof of Theorem~\ref{thm:main1} is the following:
	\begin{proposition}\label{prop:triple} 
		With respect to the topology of uniform convergence on compact sets,
		$(\cZ^{1,n},\cZ^{0,n},\cZ^n,\mathfrak l^n)$ jointly converges in law
		to $(\mathscr{Z}^1,\mathscr{Z}^0,\mathscr{Z}^p,\ell^p)$ defined in
		Section~\ref{subsec:excursion}, with $\alpha=-\tfrac{\sqrt{2}}{2}$
		and $p=\tfrac{\sqrt{2}}{1+\sqrt{2}}$.  Furthermore, in any coupling such that the convergence is \nocolor{almost sure}, for any fixed
		$t \in \R$, we have that $J^n_t - K^n_t$ converges  to
		$\1_{\{\delta(t)=1\}}$ a.s., where $\delta(t)$ is the random variable defined
		in Definition~\ref{defn:delta}.
		The same is true for
		$(\ol\cZ^n,\ol \cZ^{1,n},\ol \cZ^{0,n},\ol{\mathfrak{l}}^n)$.   
	\end{proposition}
	
	We begin with two lemmas. Recall the notation $\tau^-(\cdot)_a$ from
	\eqref{eq:e1}.  The first lemma is a standard fact about Brownian
	motion whose proof we will skip. 
	\begin{lemma} \label{lem:inf_conv} Suppose that $X^n$ is a sequence of
		stochastic process converging a.s.\ to a 1D standard Brownian motion
		$X^0$ in local uniform topology. Then for all $a>0$, we have
		$\tau^-(X^n)_a$ converges to $\tau^-(X^0)_a$ a.s.
		Furthermore, if we define
		$\sigma^n = \sup \{s \leq 1\,: \, X^n_s = \inf_{u \leq s} X^n_u\}$,
		then $\sigma^n \to \sigma^0$ a.s. 
	\end{lemma}
	
	The second lemma is a 2D extension of Lemma~\ref{lem:inf_conv}. 
	
	\begin{lemma}\label{lem:time} 
		Define $\widetilde{\tau}^n_1$ and $\tau_1^n$ to be the random times
		for which $(\widetilde{\tau}^n_1, \tau_1^n)$ is the largest open
		interval containing $t=1$ over which $t\mapsto \mathfrak{l}^{n}_t$ is
		constant. 
		Then in any coupling such that $(\cZ^{1,n},\cZ^{0,n})_{n=1}^\infty$
		converges to $(\scZ^1,\scZ^0)$ a.s., we have that 
		$(\tilde{\tau}_1^n, \tau_1^n, \mathfrak{l}_1^{\,n}, J^n_1-K^n_1)$
		converges to $ (\tilde{\tau}_1^p,\tau_1^p,\ell^p_1,\1_{\{\delta(1)=1\}})$
		a.s., where $p =\tfrac{\sqrt{2}}{1+\sqrt{2}}$ and $\ell^p$ is as in
		Lemma~\ref{lem:constructW}.
	\end{lemma}
	\begin{figure}
		\centering
		\includegraphics[width=0.4\textwidth]{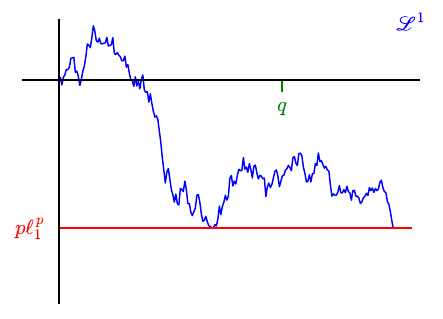} \hspace{1cm}
		\includegraphics[width=0.4\textwidth]{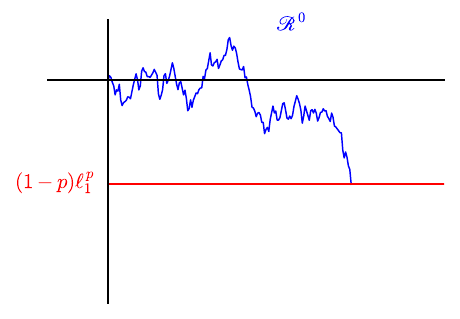}
		\caption{An illustration of the event $E=E^{i}_q$ when
			$i=1$. We lower the horizontal red lines in the two figures
			at rates $p$ and $1-p$, respectively, until the sum of the
			first hitting times exceeds $1$.  The event $E$ occurs if
			this jump happens during an excursion of $\scL^1$ including
			the time $q$ as shown.  }
		\label{fig:timesconverge}
	\end{figure}
	\begin{proof} 
		Recall that $\cZ^n$ may be constructed from $(\cZ^{1,n},\cZ^{0,n})$
		by alternately concatenating excursions, by the definition of
		$(\cZ^{1,n},\cZ^{0,n})$. After $2k$ such excursions have been
		concatenated ($k$ from $\cZ^{1,n}$ and $k$ from $\cZ^{0,n})$, the
		running infimum of the $\cR^{0,n}$ is equal to $-k/\sqrt{2n}$, and
		the running infimum of $\cL^{1,n}$ is
		$-(\Delta_1 + \cdots + \Delta_k)/\sqrt{4n}$, where $\{\Delta_i\}$
		are the i.i.d. overshoots defined in Remark~\ref{rmk:overshoot}.
		Therefore, by Lemma~\ref{lem:mc_concentration},
		\begin{equation}\label{eq:ratio}
		\lim_{n\to\infty} \frac{\cL_{T^{1,n}(\lfloor a \sqrt{n} \rfloor)}^{1,n}}{\cR_{T^{0,n}(\lfloor a \sqrt{n} \rfloor)}^{0,n}}
		=
		\frac{\E[\Delta_1]}{\sqrt{2}} =\nocolor{\sqrt{2}=\frac{p}{1-p}}\;\; \textrm{a.s. for all $a>0$ simultaneously.}
		\end{equation}
		
		When $\delta(1)=1$ (resp. $\delta(1)=-1$),  we have that $p\ell^p_1$  (resp. $(1-p)\ell^p_1$) is the global minimum of $\scL^1$  (resp. $\scR^0$)  in $(0,T)$ for some $T>0$. 
		For $q\in \Q_{>0}$, let $E^1_q$ the event that  $\delta(1)=1$ and $p\ell^p_1$ is the global minimum of $\scL^1|_{[0,q]}$.
		Let $E^0_q$ the event that  $\delta(1)=-1$ and $(1-p)\ell^p_1$ is the global minimum of $\scR^0|_{[0,q]}$. Then  
		$\P[\bigcup_{i=0,1;\; q\in \Q_{>0}}  E^i_q]=1$.
		
		Let $E=E^i_q$ for some $i=0$ or 1 and $q\in \Q_{>0}$ such that $\P[E]>0$. For concreteness we  may assume $i=1$ but the same argument works for $i=0$. (See Figure~\ref{fig:timesconverge} for an illustration.) From now on  we assume $E$ occurs.
		Let $a^n$ be such that $-pa^n\colonequals \min_{t\in[0,q]} \cL^{1,n}_t$. Then 
		since  $(\cZ^{1,n},\cZ^{0,n})_{n=1}^\infty$ converge to $(\scZ^1,\scZ^0)$ a.s., Lemma~\ref{lem:inf_conv} yields that  $a^n\to \ell^p_1$ and $\tau^-(\cL^{1,n})_{pa^n} \to \tau^{-}(\scL^1)_{p\ell^p_1}$ a.s.
		Let $j^n$ be the number of left excursions of $\cL^1$ over the
		interal $[0,2qn]$. Then by Lemma~\ref{lem:inf_conv},
		$$T^{1,n}(j^n)=\tau^{-}(\cL^{1,n})_{pa^n} \to \tau^{-}(\scL^1)_{p\ell^p_1}\quad \textrm{and}\quad T^{1,n}(j^n+1)\to \inf\{t:\scL^1_t<-p\ell^p_1\}  \;\;
		\text{a.s.}$$ Let $k^n=j^n$. By \eqref{eq:ratio} \nocolor{and the fact that $(1-p) \ell^p_1$ is a.s. a continuous point of $a\mapsto \tau^{-}(\scR^0)_{a}$},
		we have that	$\nocolor{T^{0,n}(k^n)} \to \tau^{-}(\scR^0)_{(1-p) \ell^p_1}$.  
		
		Therefore, on
		the event $E$, for sufficiently large $n$ we have
		\[
		T^{1,n}(j^n)+T^{0,n}(k^n)<1 \quad \textrm{and} \quad
		T^{1,n}(j^n+1)+T^{0,n}(k^n)>1\quad \text{ a.s.}
		\]
		Recall the
		definition of $\mathfrak l^n,J^n$ and $K^n$. We have $K^n_1=k_n$ and
		$J^n_1=j_n+1$. Therefore on $E$ we have that
		$(\tilde{\tau}_1^n, \tau_1^n, \mathfrak{l}_1^{\,n}, J^n_1-K^n_1)$
		converges a.s.\ to $ (\tilde{\tau}_1^p,\tau_1^p,\ell^p_1,\1_{\delta(1)=1})$
		as desired. Letting $E$ range over all the events $E^i_q$, we conclude the proof.
	\end{proof}
	
	\begin{proof}[Proof of Proposition~\ref{prop:triple}]
		We begin by showing joint convergence of
		$(\cZ^{1,n},\cZ^{0,n},\cZ^n)$. This triple is tight since its
		marginals are tight. By possibly extracting a subsequence if
		necessary, we consider a coupling of $(\cZ^{1,n},\cZ^{0,n},\cZ^n)$
		where $(\cZ^{1,n},\cZ^{0,n})$ converges to $(\scZ^1,\scZ^0)$ and
		$\cZ^n$ converges to some Brownian motion $\wt\scZ$ a.s. Let $\scZ^p$
		to be coupled with $(\scZ^0,\scZ^1)$ as in
		Section~\ref{subsec:excursion}. Our goal is to show that $\wt\scZ = \scZ^p$.
		
		To provide intuition for our proof that $\wt \scZ = \scZ^p$, let us
		articulate the reason that $\wt \scZ$ and $\scZ^p$ might, \textit{a
			priori}, be different. The process $\scZ^p$ is obtained by
		concatenating its \xin{constituent} It\^{o} excursions strictly in their
		order of their local-time appearance. By contrast $\wt \scZ$ is obtained
		as a limit of processes obtained by concatenating excursions in an
		{alternating} manner. So the concern is that this distinction
		between concatenation procedures introduces some discrepancy in the
		limit, even though Lemma~\ref{lem:time} ensures that macroscopic
		excursions coincide. We will use Lemma~\ref{lem:atomic} to resolve
		this concern. 
		
		By Lemma~\ref{lem:constructW} and~\ref{lem:time}, we have  
		$\left.(\wt \scZ -
		\wt \scZ_{\tilde{\tau}^p_1})\right|_{[\tilde{\tau}^p_1,{\tau}^p_1]} =
		\left.(\scZ^p -
		\scZ^p_{\tilde{\tau}^p_1})\right|_{[\tilde{\tau}^p_1,{\tau}^p_1]}$ a.s.  The same holds if we replace $(\tilde{\tau}^p_1,\tau^p_1)$ by $(\tilde{\tau}^p_t,\tau^p_t)$ for any $t>0$. This verifies the first condition in Lemma~\ref{lem:atomic}.
		
		To verify the second condition, define $(\mathcal{F}_t)_{t\ge 0}$ to be the filtration generated by $\wt\scZ$ and
		$\scZ^p$. We claim that $\wt\scZ$ and $\scZ^p$ are both
		$\mathcal{F}$-Brownian motions. They are adapted to $\mathcal{F}$ by
		definition, so it suffices to check that $\left.(\wt\scZ -
		\wt\scZ_t)\right|_{[t,\infty)}$ and
		$\left.(\scZ^p-\scZ^p_t)\right|_{[t,\infty)}$ are independent of
		$\mathcal{F}_t$ for all $t \geq 0$. Without loss of generality, we take
		$t=1$. 
		
		We claim that $\cF_1$ is a
		subset of the $\sigma$-algebra generated by the collection of all the
		random walk steps which occur prior to time $1$. That is, 
		\begin{equation} \label{eq:sigmacontain} 
		\cF_1 \subset \Sigma_1 \colonequals \sigma(\{ (\cZ^{n}_t)_{t \leq 1}:n\in \N\}).
		\end{equation}
		Since $\cZ^n$ converges to $\wt\scZ$ a.s., $(\wt\scZ_t)_{t
			\leq 1}$ is measurable with respect to $\Sigma_1$. Now we
		show that $(\scZ^p_t)_{t \leq 1}$ is measurable with respect
		to $\Sigma_1$ too. Define $\overline{\tau}^{1,n}_1$ and
		$\overline{\tau}_1^{0,n}$ to be the times with respect to
		$\cZ^{1,n}$'s and $\cZ^{0,n}$'s clocks (respectively) when
		$\cZ^n$ reaches time 1. Similarly, define
		$\overline{\tau}^{1}_1$ and $\overline{\tau}_1^{0}$ to be the
		times with respect to $\scZ^{1}$'s and $\scZ^{0}$'s clocks
		(respectively) when $\scZ^p$ reaches time 1. Note that
		$(\scZ^{1}_t)_{t \leq \overline{\tau}^1}$ and
		$(\scZ^{0}_t)_{t \leq \overline{\tau}^0}$ are the a.s. limit
		of $(\cZ^{1,n}_t)_{t \leq \overline{\tau}^{1,n}}$ and
		$\cZ^{0,n}_t)_{t \leq \overline{\tau}^{0,n}}$ respectively,
		thus measurable with respect to $\Sigma_1$.  By
		Lemma~\ref{lem:constructW} and Item 2 of
		Proposition~\ref{prop:W}, $(\scZ^p_t)_{t \leq 1}$ is
		determined by $(\scZ^{1}_t)_{t \leq \overline{\tau}^1}$ and
		$(\scZ^{0}_t)_{t \leq \overline{\tau}^0}$, thus measurable
		with respect to $ \Sigma_1$.
		
		Since $\left.(\scZ-\scZ_1)\right|_{[1,\infty)}$ is independent
		of $\Sigma_1$, \eqref{eq:sigmacontain} yields that $\wt \scZ$
		is an $\cF$-Brownian motion.
		So it remains to show that $\scZ^p$ is an $\cF$-Brownian
		motion. Similarly as \eqref{eq:sigmacontain}, we have
		\[
		\cF_{\tau_1} \subset \sigma(\{ (\cZ^{n}_t)_{t \leq \tau^n_1}:n\in \N\}).
		\]
		Therefore is $\cF_{\tau_1}$ independent of
		$(\wt\scZ-\wt \scZ_{\tau_1})|_{[\tau_1,\infty)}$. Moreover,
		the process $(\scZ^p-\scZ^p_{\tau_1})|_{[\tau_1,\infty)}$ is
		measurable with respect to to the random walk steps appearing
		after time $\tau^n_1$ (since these determine $\cZ^{n,1}$ and
		$\cZ^{n,0}$ after $\tau^n_1$, which in turn determine
		$(\scZ^p-\scZ^p_{\tau_1})|_{[\tau_1,\infty)}$). Thus
		$\cF_{\tau_1}$ is independent of
		$(\scZ^p-\scZ^p_{\tau_1})|_{[\tau_1,\infty)}$. Note that
		$\scZ|_{[1,\infty)}-\scZ_1$ and
		$\scZ^p|_{[1,\infty)}-\scZ^p_1$ agree up to time $\tau_1$.
		Furthermore, the conditional law of
		$\scZ^p|_{[\tau_1,\infty)}-\scZ^p_{\tau_1}$ given
		$\cF_{\tau_1}$ is the same as that of
		$\wt\scZ|_{[\tau_1,\infty)}-\wt\scZ_{\tau_1}$. Therefore, the
		conditional law of $\scZ^p|_{[1,\infty)}-\scZ^p_1$ given
		$\cF_1$ is the same as that of
		$\wt \scZ|_{[1,\infty)}-\wt\scZ_1$. Thus $\scZ^p$ is an
		$\cF$-Brownian motion, and Lemma~\ref{lem:atomic} tells us
		that $\wt\scZ = \scZ^p$.
		
		The joint convergence of
		$(\cZ^{1,n},\cZ^{n,0},\mathfrak{l}^n)$ follows from
		Lemma~\ref{lem:time} and the monotonicity of
		$\mathfrak{l}^n$. Since $\ell^p$ is determined by
		$(\scZ^0,\scZ^1)$ by Lemma~\ref{lem:constructW}, this
		concludes the proof for the forward case.
		
		The same proof applies to
		$(\ol\cZ,\ol\cZ^1,\ol\cZ^0,\ol{\mathfrak l})$, except here we
		don't have the geometric overshoot $\Delta$ in
		Remark~\ref{rmk:overshoot}. Therefore the ratio of
		$\ol\cL^{1,n}$ to $\ol\cR^{0,n}$ as in Lemma~\ref{lem:time} is
		determined by the scaling of the abscissa and ordinate, which
		is $\tfrac{1}{\sqrt{2n}}:\tfrac{1}{\sqrt{4n}}=\sqrt{2}$ as in
		the forward case.
	\end{proof}
	
	We now provide a geometric interpretation of
	Proposition~\ref{prop:triple}. Recall $\mu_\fh,\eta^\b$ and
	$\scZ^\b$ in Theorem~\ref{thm:main1}. Let
	$\theta:= \theta(\tfrac{\sqrt{2}}{1+\sqrt{2}})$ as in
	Lemma~\ref{lem:p}. We will eventually show that
	$\theta=\tfrac{\pi}{6}$, but at this moment $\theta$ is an
	unknown constant. Consider the flow \nocolor{lines}
	$\eta^{\theta},\eta^{\theta+\pi}$ and the time reversal
	$\bb\eta$ of $\eta^\b$.  Proposition~\ref{prop:joint}
	and~\ref{prop:triple} ensure the existence of the following
	coupling:
	\begin{definition}\label{def:coupling} 
		We call a coupling of \nocolor{$\mu_\fh,\eta^\b,\scZ^\b$} and a sequence of UIWTs a {\bf usual coupling} if \nocolor{the following holds. Set $\scZ^p_t:=\scZ^\b_t $ and $ \ol\scZ^p_t:= (\scR^\b_{-t},\scL^\b_{-t})$ for $t\ge 0$. Then} 
		\begin{align*}
		&(\cZ^n,\cZ^{1,n},\cZ^{0,n},\mathfrak l^n)   \;\;\textrm{converge to}\;\;  (\scZ^p,\scZ^1,\scZ^0,\ell^p) \;\; \textrm{a.s.};\\
		& (\ol\cZ^n,\ol\cZ^{1,n},\ol\cZ^{0,n},\ol{\mathfrak l}^n) \;\;\textrm{converge to}\;\;  (\ol\scZ^p,\ol\scZ^1,\ol\scZ^0,\ol\ell^p) \;\; \textrm{a.s.};
		%&\nocolor{where}\scZ^p_t=\scZ^\b_t \quad \textrm{and} \quad \ol\scZ^p_t= (\scR^\b_{-t},\scL^\b_{-t}) \quad %\forall \;t\ge 0.
		\end{align*}and  moreover,  the same convergence holds with $\cZ$ (resp. $\ol\cZ$)
		replaced by the forward (resp., reverse) grouped-step walk
		viewed at $\min\{i\ge 3nt: w_i=b\}$ (resp. $\max\{i\le
		3nt:w_i=b  \}$)  and $\scZ^\b$ recentered at $t$,
		simultaneously for all $t\in \Q$. 
	\end{definition}
 \nocolor{By definition,} in a usual coupling, $(\scZ^p,\scZ^1,\scZ^0)$ (resp. $(\ol\scZ^p,\ol\scZ^1,\ol\scZ^0)$)  is  the Brownian motion decomposition  corresponding to $(\eta^\b,\eta^\theta)$ (resp. $(\bb\eta,\eta^{\theta+\pi})$) in Section~\ref{subsec:singleflow}.

	Consider  the setting of Proposition~\ref{prop:W} with $\eta'=\eta^\b$. 
	\nocolor{For $t>0$, let $\ell_t$} be the quantum length of 
	\nocolor{$\eta^\theta\cap \eta^{\b}([0,t])$}. 
	The curve $\eta^\theta$  divides $\eta^{\b}([0,1])$ into a right region and a left region whose
	$\mu_{\fh}$-area is denoted by $A_1$ and $1-A_1$, respectively.  \nocolor{(Recall from Lemma~\ref{lem:p} that the right  side of $\eta^\theta$ is the region bounded by $\eta^\theta$ and the right  frontier of $\eta^\b(-\infty,0]$.)}
	
	We define discrete analogues of the same quantities. Consider a UIWT
	$S$, its encoding walk $Z^{\b}$ and its blue tree exploration path
	$\fraketa^{\b}$.  We define
	$\fraketa^{\b,n} = \fraketa^{\b}(\lfloor 3nt \rfloor)$.  We define the
	\textbf{discrete quantum length} of a portion of a flow line
	(resp. dual flow line) to be $\frac{1}{\sqrt{2n}}$
	(resp. $\frac{1}{\sqrt{4n}}$) times its number of edges. In light of
	Proposition~\ref{prop:dual}, this scaling allows $\cZ^n$ and
	$\ol\cZ^n$ to measure the net change of the quantum length of flow
	lines and dual flow lines.  For any edge subset of the $n$th UIWT in a
	usual coupling, we define its \textbf{discrete quantum area} to be
	$\tfrac{1}{3n}$ times its cardinality.
	
	\nocolor{For $t>0$, let $\ell^n_t$}  denote the discrete quantum length of the
	intersection of \nocolor{$\fraketa^{\b,n}([0,t])$}  and the red flow line   from the tail of $\fraketa^{\b,n}(0)$.  \nocolor{By Lemma~\ref{lem:leftrightexcursions} $\sqrt{2n}\ell^n_t$ equals $ J^n_t$ in Proposition~\ref{prop:triple}. }
	\nocolor{Now  without loss of generality we focus on $t=1$.} Let $A_1^n$ denote the set of edges in $\fraketa^{\b,n}([0,1])$ right of the red flow line from the tail of $\fraketa^{\b,n}(0)$. We also use $A_1^n$ to denote its discrete quantum area.
	We similarly define $\bar\ell^n$ and $\ol A_1^n$ with red dual flow line in place of red flow line and  $\fraketa^{\b,n}([-1,0])$  in place of \nocolor{$\fraketa^{\b,n}([0,1])$}.

	\begin{proposition} \label{prop:onepocket}
		In a usual coupling,  we have that $A_1^n$ converges to $ A_1$ in probability.
		The corresponding result holds for $\ol A^n_1$ as well.
	\end{proposition}
	\begin{proof} 
		Let 
		\begin{align*}
		&\cA_1^n=\nocolor{\frac{1}{2n}}|\{T\le i\le 3n: w_i\neq \r \;\textrm{and}\; \eta^\b(i)\in A_1^n \}|,\\
		&\wt \cA_1^n=\nocolor{\frac{1}{2n}}|\{0\le i\le 3n: w_i\neq \r \;\textrm{and}\; \eta^\b(i)\in A_1^n   \}|.
		\end{align*}where 	$T=\min\{i\ge 0: w_i=b \}$. 
		Then by Lemma~\ref{lem:time}, we have $\mathcal{A}_1^n\to A^1$ a.s.
		By Remark~\ref{rmk:wheretostart} and Lemma~\ref{lem:mc_concentration}, 	$\P[\wt\cA_1^n-\cA_1^n\ge n^{-0.1}]$ and  $\P[\wt\cA_1^n-\nocolor{A_1^n}\ge n^{-0.1}]$ decay superpolynomially. This concludes the proof for $A^n_1$. The reverse case follows from the same argument.
	\end{proof}
	
	\begin{remark} \label{rmk:wheretostart} As in the proof of
		Proposition~\ref{prop:onepocket}, when we transfer the convergence
		in Proposition~\ref{prop:triple} to the convergence of a natural
		geometric quantity, there is an error coming from the discrepancy
		introduced by starting $\cZ^{\b,n}$ at the first $\b$ following time
		0. This error can be handled by Remark~\ref{rmk:argument}. There is
		another error coming from skipping the red edges, which can be
		handled by Lemma~\ref{lem:mc_concentration}.  The same issues will
		appear often in the argument below and can always be handled this
		way. So for simplicity we will elide this issue henceforth.
	\end{remark}
	\begin{proposition} \label{lem:DQL} In a usual coupling, we have that
		$\ell^n$ converges to $ \tfrac{c}{1+\sqrt2}\ell$ in probability.
		Here $c$ is the constant in Proposition~\ref{prop:W} with
		$\theta=\theta\left(\tfrac{\sqrt2}{1+\sqrt2}\right)$ as above.		
		The corresponding result holds if we consider
		$(\ol\cZ^n,\ol\cZ^{1,n},\ol\cZ^{0,n},\ol{\mathfrak l}^n)$ and
		$\ol\ell^n$.
	\end{proposition}
	\begin{proof} 
		Pursuant to Remark~\ref{rmk:wheretostart}, we may consider the red
		flow line starting from the tail of $\eta^\b(T)$ where
		$T=\min\{i\ge 0: w_i=\b\}$. (In other words, the discrete quantum
		length of the red flow line from the tail of $\eta^\b(0)$ differs from this flow
		line by $o_n(n^{-0.1})$ with superpolynomial probability, and hence
		negligibly). By Lemma~\ref{lem:leftrightexcursions}, the discrete
		quantum length of this flow line is $\tfrac{1}{\sqrt{2n}}$ times the
		number of left and right excursion pairs encountered over the time
		interval $[0,t]$. Recall \eqref{eq:local} and
		Remark~\ref{rmk:overshoot}. Each left excursion increases the local
		time $\mathfrak l^n$ by $\tfrac{1}{\sqrt{2n}}$, and each right
		excursion increases $\mathfrak l^n$ by
		$\tfrac{1}{\sqrt{4n}} \Geo$. By
		Lemma~\ref{lem:mc_concentration} and the fact that
		$\mathfrak l^n\to \ell^p$, we have $\ell^n$ converges to
		$(1 + \sqrt{2})^{-1}\ell^p$ in probability.  Now the result follows
		from the definition of $c$.
		
		The reverse direction is similar but simpler since the length of the
		dual flow line is just $\tfrac{1}{\sqrt{4n}}$ times the number of
		right/left excursion pairs (as there is no overshoot). So each edge
		results in an $\ol{\mathfrak l}^n$ increment of
		$\tfrac{1}{\sqrt{2n}}+\tfrac{1}{\sqrt{4n}}=\tfrac{1}{\sqrt{4n}}\cdot(1+\sqrt2)$. Now the dual flow line  result follows from the same argument as in the flow line case before.
	\end{proof}

	\subsection{Scaling limit in the infinite-volume case}\label{sec:3tree}
	
	\begin{figure}
		\centering
		\subfigure[]{\includegraphics[width=0.23\textwidth]{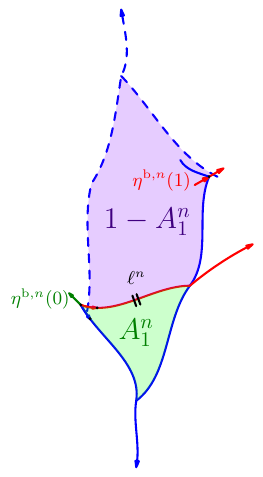}} \:
		\subfigure[]{\includegraphics[width=0.45\textwidth]{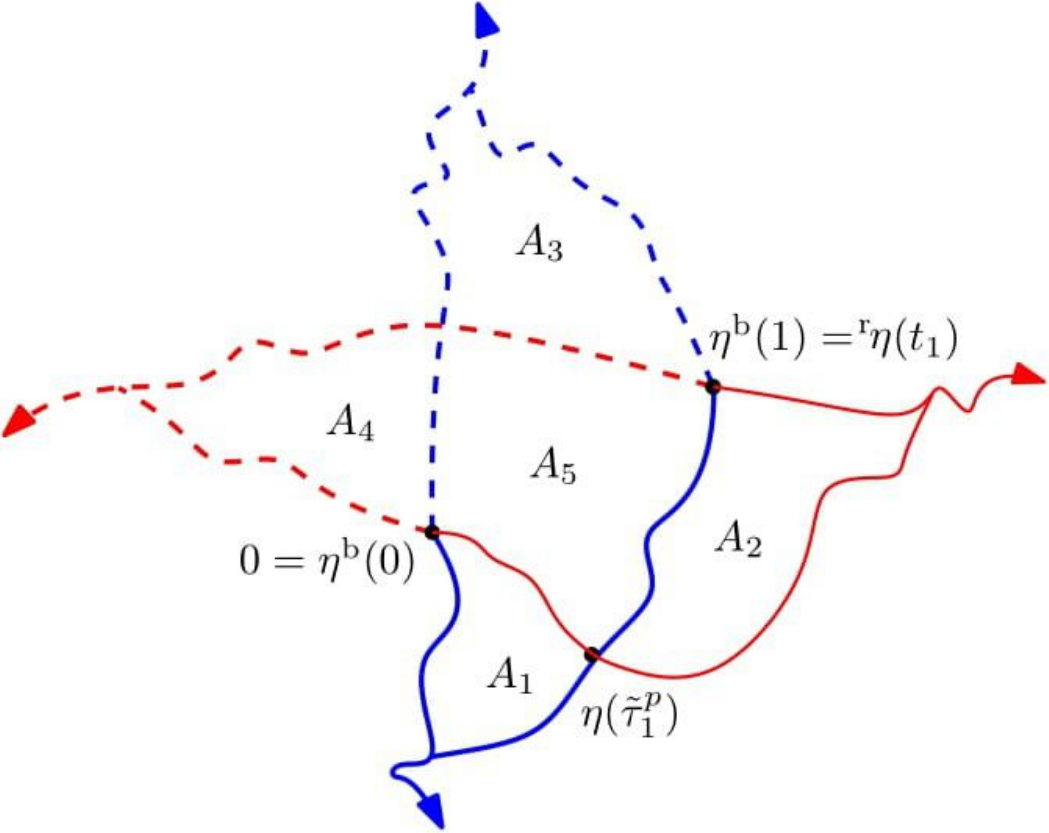}} \:
		\caption{(a) The blue flow lines and dual flow lines from $\eta^{\b,n}(0)$ (which happens to be a green edge here) and $\eta^{\b,n}(1)$ (which happens to be a red edge), together with the red flow line from the tail of $\eta^{\b}(0)$. (b) The five-region picture in Proposition~\ref{prop:tqconverges} in the case $t_1<0$.  \nocolor{[Yiting: change $\eta^{\mathrm r}$ to $^{\mathrm r}\eta$ ]}
			\label{fig:fourflowlines}  
		}
	\end{figure}
	
	Throughout this section we work in the setting of  a usual coupling, where $\eta^\b$ is the zero-angle Peano curve in an imaginary geometry.  Let $\eta^\r$ be the Peano curve of angle $\theta+\tfrac{\pi}{2}$ so that $\eta^\theta$ is a right boundary of  $\eta^\r$.  Let $\prescript{\r}{}{\eta}$ be the Peano curve of angle $\theta-\tfrac{\pi}{2}$, which is also the time reversal of $\eta^\r$. For $q\in \Q$, we define $t_q$ to be the  a.s.\ unique time  such that $\eta^\b(q) = \prescript{\r}{}{\eta}(t_q)$.   
	The time $t_1$ is negative if and only if  $\eta^\b(1)$ is left of $\eta^\theta$, which is further equivalent to $\{\delta=1\}$ in Proposition~\ref{prop:triple}. See Figure~\ref{fig:fourflowlines}(b) for an illustration of this case.  For concreteness we focus on this case in the rest of this section as the case $t_1>0$ can be treated in the same way.  The left and right boundary of $\eta^\b$ at $t=0$ and $t=1$ and the left and right boundary of $^r\eta$ at $t=1$ and $t=t_1$ divide $\eta^\b([0,1]) \cup \: \prescript{\r}{}{\eta}([t_1,0])$ into five regions.  We use $\{A_i\}_{1\le i\le 5}$ to denote the five  regions as well as their  $\mu_{\fh}$-area, as indicated in Figure~\ref{fig:fourflowlines}(b).  Note that $t_1 =-A_2-A_5-A_4$ in this case. 
	
	We have a similar discrete picture for the sequence of UIWTs in a usual coupling.  More precisely, let $\prescript{\r}{}{\fraketa}$ be the counterclockwise exploration path of a UIWT  analogous to $\fraketa^\b$, and let $\rZr=(\rr R,\rr L)$ be the walk associated with  $\rr\fraketa$ (Section~\ref{subsec:dual}).  For the $n$th UIWT in a usual coupling, let $\prescript{\r}{}{\fraketa}^n(t)=\prescript{\r}{}{\fraketa}(\lfloor3nt\rfloor)$ and $\prescript{\r}{}{Z^{n}}=(\tfrac{1}{4n}\rr R_{\lfloor3nt\rfloor}, \tfrac{1}{2n}\rr L_{\lfloor3nt\rfloor})$. (Note that  the notation $\prescript{\r}{}{Z^{n}}$ has a different meaning from Section~\ref{sec:bijection} and~\ref{sec:pair}.) In light of Proposition~\ref{prop:dual} and the scaling in the definition of discrete quantum length,  the abscissa and ordinate of $\prescript{\r}{}{Z^{n}}$ measure the net change of the discrete quantum length of the right and left ``boundary'' of $\rr\fraketa^n$ respectively, which are dual red flow lines and red flow lines.  
	
	Let us define discrete analogues to $t_q$ via  
	\[
	t_q^n \colonequals \inf \{t\in \R \, : \, \fraketa^{\b,n}(q) = \rr\fraketa^{n}(t)\} \quad \forall\; q\in \Q.
	\]
	In analogue with  $(A_i)_{1\le i\le 5}$, the region $\fraketa^{\b,n}([0,1]) \cup \prescript{\r}{}{\fraketa}^n([t^n_1,0])$ can be divided into five subregions $(A^n_i)_{1\le i\le 5}$ bounded by the  red and blue flow lines and dual flow lines from  $\fraketa^{\b,n}(0)$ and $\fraketa^{\b,n}(1)$. We also use the same notation $(A_i^n)_{1\le i\le 5}$ to represent the discrete quantum areas (\nocolor{Figure~\ref{fig:overlap} (a)}) of the corresponding subregion.
	\begin{remark}\label{rmk:boundary}
		When defining these five subregions and their discrete quantum areas, we are not careful about whether to include their boundaries or not. The reason is that these boundaries consist of portions of flow line and dual flow lines. Once scaled by $(3n)^{-1}$, they have $o(n^{-0.1})$ quantum area with superpolynomial probability. Therefore they are negligible in the scaling limit.  This remark applies to all edge subsets whose discrete quantum areas are relevant. So we ignore the boundary issue from now on.
		
		When we talk about the boundaries of these subregions, which are flow lines and dual flow lines, we simply say that they start from $\fraketa^{\b,n}(0)$ and $\fraketa^{\b,n}(1)$ without further specifying the starting vertex or face. The reason is that due to Remark~\ref{rmk:wheretostart}, the precise starting point is immaterial in the scaling limit.
	\end{remark} 
	Since we assume $\delta=1$ in the continuum, in a usual coupling, $J^n_1-K^n_1$ in Proposition~\ref{prop:triple} converges to $1$ a.s. Therefore $\eta^{\b,n}(1)$ is left of the red flow line from the tail of $\eta^{\b,n}(0)$ for large enough $n$. See \nocolor{Figure~\ref{fig:overlap} (a)} for an illustration of the discrete picture in this case. Note that in this case $A^n_1$ in Proposition~\ref{prop:onepocket} is the same as $A^n_1$ defined here. 
	%$\ol A^n_1$ in Proposition~\ref{prop:onepocket} equals 
    \nocolor{Moreover, 	$A^n_1+A^n_3+A^n_5=1$.}
	
	\begin{proposition} \label{prop:tqconverges} 
		In a usual coupling, $(A_i^n)_{i=1}^5$ converges  to $(A_i)_{i=1}^5$ in probability. 
		
		Moreover, for any fixed $q\in \Q$, $t^n_q$ converges to $t_q$ in probability.
	\end{proposition}
	\begin{figure} 
		\centering
		\subfigure[]{\includegraphics[width=0.58\textwidth]{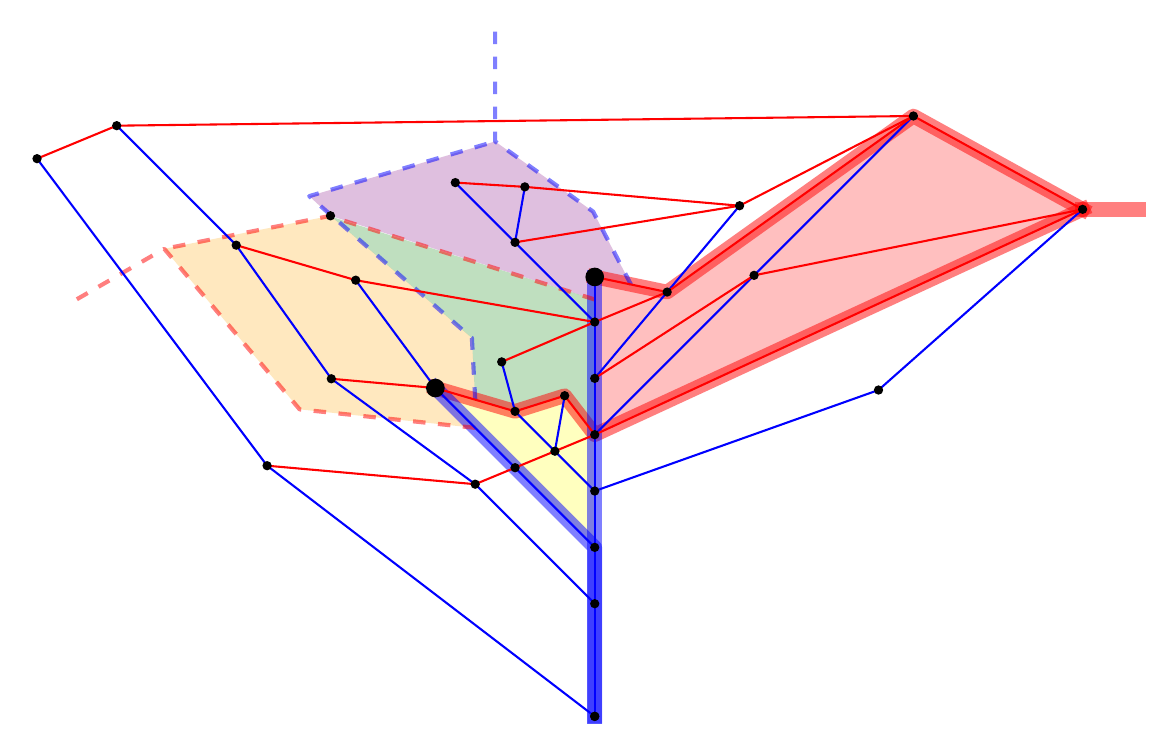}}
		\subfigure[]{\includegraphics[width=6cm]{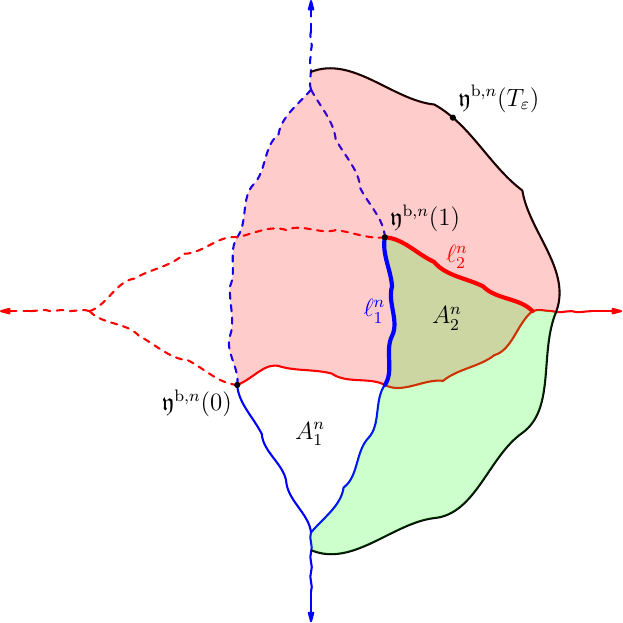}}
		\caption{(a) The five regions in Figure~\ref{fig:fourflowlines} shown for a small portion of an infinite wooded triangulation. The tails of $\fraketa^{\b,n}(0)$ and $\fraketa^{\b,n}(1)$ are indicated with large dots, and flow lines from these vertices are shown in bold red and blue. Dual flow lines from $\fraketa^{\b,n}(0)$  and $\fraketa^{\b,n}(1)$ are shown as dashed. (b) To show that $A_2^n \to A_2$, we fix  a large enough $T_\eps>0$ and consider the region $\fraketa^{\b,n}[0,T_\eps]$. We define the green shaded area $B_1^n$ and the pink shaded area $B_2^n$ (the pink and green regions both include the region marked $A_2^n$) and note that the area of $\fraketa^{\b,n}[0,T_\eps]$ is equal to $A_1^n + B_1^n +B_2^n-A_2^n$. Thus $A_2^n = B_1^n + B_2^n - A_1^n - T_\eps$.  
		}
		\label{fig:overlap} 
	\end{figure}
	\begin{proof}
		Suppose we are in the setting of a usual coupling; recall Proposition~\ref{prop:onepocket} and~\ref{lem:DQL}. We further assume that $t_1<0$ and $\delta=1$ as the other case is the same.
		Then Proposition~\ref{prop:onepocket} implies that $A_1^{n}\to A_1$ and $A^n_3+A^n_5\to A_3+A_5$ in probability. Regarding $\fraketa^{\b,n}(1)$ as the root of the UIWT and considering the time reversal of $\fraketa^{\b,n}$, we can treat $A^3_n$ as the right region of the dual red flow line. Since the bi-infinite word of  a UIWT is stationary,  the  UIWT recentered at $\fraketa^{\b,n}(1)$ is still a UIWT. Now \nocolor{the second  statement (on  $\overline A^n_1$) in} Proposition~\ref{prop:onepocket} gives  $A_3^n \to A_3$ in probability, hence $A_5^n\to A_5$ in probability.
		
		Most of the remainder of the proof is dedicated to the convergence of $A_2^n$. We refer to Figure~\ref{fig:overlap}(b) for various quantities involved in this argument.
		
		We first show that $(A_2^n)_{n \in \N}$ are tight. By
		Proposition~\ref{prop:triple}, the discrete quantum
		length $\ell^n_1$ of the blue flow line starting from
		$\fraketa^{\b,n}(1)$ inside $^r\fraketa^n([t^n_1,0])$
		converges a.s.\ to a finite number.   By the symmetry
		between blue and red in the UIWT, we can apply Proposition~\ref{prop:triple}, \ref{prop:onepocket} and~\ref{lem:DQL} to the exploration $^r\fraketa^n$ for the UIWT recentered at $\fraketa^{\b,n}(1)$. Then $A^n_2$ can be seen as the discrete quantum area of $^r\fraketa^n$ on the left of the blue flow line from $\fraketa^{\b,n}(1)$, when the local time of this blue flow line reaches $(1+\sqrt2+o_n(1))\ell^n_1$ (recall Proposition~\ref{lem:DQL}).  This gives the tightness of $A^n_2$.
		
		Now we claim that for all $\eps>0$, there exists $T_\eps<\infty$ such that with probability $1-\eps$, the regions $A^n_1,A^n_2,A^n_3,A^n_5$ are all contained in $\fraketa^{\b,n}[0,T_\eps]$. To show this, note that we can recenter the UIWT at $\fraketa^{\b,n}(1)$ and treat $\fraketa^{\b,n}(0)$ as $\fraketa^{\b,n}(-1)$. Therefore $A^n_4$ can be treated in the same way as $A^n_2$. In particular, $A^n_4$ is tight. Since we assume $\delta=1$,  under a usual coupling, $J^n_1-K^n_1$ in Proposition~\ref{prop:triple} converges to $1$ a.s. Therefore 
		\begin{equation}\label{eq:t1}
		t_1^n = -(A_2^n + A_5^n + A_4^n)+o_n(1),
		\end{equation}where $o_n(1)$ comes from Remark~\ref{rmk:boundary}. This means $(t_1^n)_{n \in \N}$ is tight. The same argument showing the tightness of $(A^n_2)_{n \in \N}$ gives the existence of $T_\eps$.
		
		Now we work on the event $T_\eps<\infty$. we define $B_1^n(T_\eps)$ to be the area of the intersection of $\fraketa^{\b,n}([1,T_\eps])$ and the region right of the red flow line from $\fraketa^{\b,n}(1)$. We define $B_2^n(T_\eps)$ to be the area of the intersection of $\fraketa^{\b,n}([0,T_\eps])$ and the region left of the red flow line from $\fraketa^{\b,n}(0)$. Then by the inclusion-exclusion principle applied to the discrete quantum area, we have 
		\begin{equation} \label{eq:PIE}
		A_2^n = B_1^n(T_\eps) + B_2^n(T_\eps) + A_1^n - T_\eps.
		\end{equation}
		
		The inclusion-exclusion principle also applies in the continuum, yielding a continuum analogue of \eqref{eq:PIE} with the LHS being  $A_2$ and the RHS being the continuum limit of the RHS of \eqref{eq:PIE}. This yields that $A_2^n \to A_2$ in probability. 
		
		As explained in the argument for the existence of $T_\eps$, the region $A^n_4$ can be treated similarly as $A^n_2$ by recentering the UIWT at $\fraketa^{\b,n}(1)$. Therefore $A_4^n \to A_4$ in probability.
		By \eqref{eq:t1}, we have $t^n_1\to t_1$ in probability. The argument for general $t^n_q$ is the same.
	\end{proof}
	
%\nocolor{Suppose  we are in a usual coupling. Since $\prescript{\r}{}{Z}^n$ converges in law, by possibly extracting a subsequence,  we may further assume that  in this coupling $\rr Z^n$ converges to a Brownian motion $\scZ=(\scR,\scL)$ a.s. in the locally uniform topology.
%	By Proposition~\ref{prop:tqconverges}  and possibly extracting a subsequence, we can assume that in our coupling, $t_q^n \to t_q$ a.s. for all $q\in \Q$.
%	\begin{lemma}\label{lem:key}
% We have $\scR_{t_q}=\tfrac{c}{1+\sqrt 2}$ $  \wt\scR_{t_q}$ for all $q\in \Q$, where $c$ is the constant in Proposition~\ref{lem:DQL}. 
%	\end{lemma}
%\begin{proof}
% By Proposition~\ref{prop:dual2},~\ref{prop:W} and~\ref{lem:DQL},...
%\end{proof}
%}

	\begin{proof}[Proof of Theorem~\ref{thm:main1}]
		Suppose  we are in a usual coupling. Since $\prescript{\r}{}{Z}^n$ converges in law, by possibly extracting a subsequence,  we may further assume that  in this coupling $\rr Z^n$ converges to a Brownian motion $\scZ=(\scR,\scL)$ a.s. in the locally uniform topology. We claim that 
		\begin{equation}\label{eq:Zr}
		\scL=\wt\scL \quad \textrm{and}\quad \scR=\wt\scR \quad \textrm{a.s.}
		\end{equation}
		where $\wt \scZ=\nocolor{(\wt\scR,\wt\scL)}$ is the Brownian motion associated with $(\mu_{\fh},$$^r\eta)$ in Theorem~\ref{thm:mating2}. We postpone the proof of this claim and proceed to prove Theorem~\ref{thm:main1}.
		
		Let  $\scZ^\r$ denote the mating of tree Brownian motion associated with $(\mu_{\fh},$$\eta^\r)$ and let $\scZ^\g$ be the Brownian motion such that   
		$(\scZ^\b, \scZ^\r)\overset{d}{=}(\scZ^\r, \scZ^\g)$.  Note that in Theorem~\ref{thm:main1}, the Brownian motions $\scZ^\b, \scZ^r,\scZ^\g$ correspond to the Peano curves of angle $0,\tfrac{2\pi}{3},\tfrac{4\pi}{3}$ respectively.  But at the moment we only know that $\eta^\b$ and $\eta^\r$ are of angle 0 and that $\theta+\tfrac{\pi}{2}$, respectively. We abuse notation with $\scZ^\r$ and $\scZ^\g$ here since we will see $\theta+\tfrac{\pi}{2}=\tfrac{2\pi}{3}$ momentarily. Note that $\scZ^\b,\scZ^\r,\scZ^\g$ determine each other a.s.
		
		\nocolor{Recall from Remark~\ref{rmk:notation} that }to simplify notation, we use  $Z^{\b,n}, Z^{\r,n},Z^{ \g,n}$ to denote the rescaled walks in Theorem~\ref{thm:main1}.  Proposition~\ref{prop:reverse} and~\eqref{eq:Zr} imply that $(Z^{\b,n}, Z^{\r,n})$ converges in law to $(\scZ^\b, \scZ^\r)$.   As a result of the symmetry between colors, $(Z^{\r,n},Z^{g,n})$  converges in law to  $(\scZ^\r, \scZ^\g)$.  Since $\scZ^\b,\scZ^\r,\scZ^\g$ determine each other a.s.,  $(Z^{\b,n}, Z^{\r,n},Z^{g,n})$ must jointly converge to $(\scZ^b,\scZ^r,\scZ^g)$.   By symmetry $(\scZ^\b, \scZ^\r)\overset{d}{=}(\scZ^\r, \scZ^\g)\overset{d}{=}(\scZ^\g, \scZ^\b)$. This yields that the angle difference between $\eta^\b$ and $\eta^\r$ is $\tfrac{2\pi}{3}$ and $\theta=\tfrac{\pi}{6}$.
		This will conclude the proof of Theorem~\ref{thm:main1}.
		
		We are left to prove~\eqref{eq:Zr}. We focus on proving $\scR= \wt \scR$ as $\scL= \wt \scL$ follows the same argument. By Proposition~\ref{prop:tqconverges}  and possibly extracting a subsequence, we can assume that in our coupling, $t_q^n \to t_q$ a.s. for all $q\in \Q$. By Proposition~\ref{prop:dual2},~\ref{prop:W} and~\ref{lem:DQL}, we have $\scR_{t_q}=\tfrac{c}{1+\sqrt 2}$ $  \wt\scR_{t_q}$ for all $q\in \Q$, 
		where $c$ is the constant in Proposition~\ref{lem:DQL}. \nocolor{More precisely, for a fixed $q$, $\scR_{t_q}$ and $\tfrac{c}{1+\sqrt 2}$ $  \wt\scR_{t_q}$ are two expressions of the scaling limit of the relative change of discrete quantum length of the dual red flow line from time $0$ to time $t^n_q$. The $\scR_{t_q}$ limit comes from the fact that $\rr Z^n$ converges to $\scZ=(\scR,\scL)$ a.s.  locally uniformly. The $\tfrac{c}{1+\sqrt 2}$ $  \wt\scR_{t_q}$ limit comes from the perspective where we view the lengths of the dual red flow lines as counting the local time of the left/right excursions  for the counterclockwise exploration of the red tree. In particular, we apply Proposition~\ref{lem:DQL} with $\cZ$ and $\overline Z$ being the forward and reversed grouped-step walks of $\rr Z^n$.  Here we need to make sure that the merging time (w.r.t. the counterclockwise exploration of the red tree) of the dual red flow lines from   $\fraketa^{\b,n}(0)$ and  $\fraketa^{\b,n}(1)$ is tight as $n\to \infty$. To see this, let $e^n$ be the first common  edge  of these two dual flow  red lines. Then the discrete quantum lengths from $e^n$  to $\fraketa^{\b,n}(0)$ is tight due to the local uniform convergence of $\rr Z^n$ and the tightness of $t_q^n$. This length can also be viewed as the local time accumulated for the left/right excursions  for the counterclockwise exploration of the red tree until the merging time. Since the local time process  converge to $(\ell^p_t)_{t\ge 0}$ which goes to $\infty$ as $t\to\infty$, the merging time must be tight as desired.}

		Note that the set $\{\eta^\b(q)\}_{q\in \Q}$ is dense in $\C$. Therefore $\forall \; a<b$, the open set $\rr\eta(a,b)$ contains a point in that set. This means that $\{t_q\,:\,q \in \Q\}$ is dense in $\R$.  Since $\scR_{t_q}=\tfrac{c}{1+\sqrt 2}$ $  \wt\scR_{t_q}$ for all $q\in \Q$, we have that $\scR_t=\tfrac{c}{1+\sqrt 2}\wt\scR_t$ a.s. for all $t\in \R$.   Since $\scR\overset{d}{=}\wt \scR$ hence $c=1+\sqrt2$. This concludes the proof.
	\end{proof} 
	
	\begin{remark} \label{rmk:specific_p_value}
		Our proof of Theorem~\ref{thm:main1}  implies that when $\theta=\tfrac{\pi}{6}$ and $\kappa=16$, the constants $p$ and $c$  defined in Lemma~\ref{lem:p} and Proposition~\ref{prop:W} satisfy $p= \frac{\sqrt{2}}{1+\sqrt{2}}$ and $c=1+\sqrt2$. In \cite{bipolarII}, it is shown that $c=2$  when $\theta=\tfrac{\pi}{2}$ and $\kappa=12$.  So far we don't have a derivation for these values independent of discrete models.  And we  do not know the general dependence of $p$ and $c$ on $\theta$ and $\kappa$ except that $p(\tfrac{\pi}{2})=\frac12$ holds for all $\kappa$ by symmetry.
	\end{remark}

	\subsection{From infinite- to finite-volume}\label{sec:finite}
	We will prove Theorem~\ref{thm:3pair} from Theorem~\ref{thm:main1} using a general approach for transferring convergence results of unconditioned random walks to the corresponding result for conditioned random walks. Roughly speaking, the idea is that a random walk conditioned on starting and ending at the origin and staying in the first quadrant is, away from the origin, sufficiently similar to the corresponding infinite-volume walk (see Proposition~\ref{prop:transfer}). A similar approach could be applied to the joint convergence result proved in \cite{bipolarII}.
	
	Consider a  Brownian motion $\scZ$ distributed as $\scZ^\b$ and denote by $\P$ the law of  $\scZ|_{[0,1]}$. For $\eps\ge 0$, let $E^\eps$ be the event that $\scZ|_{[0,1]}\in [-\eps,\infty)^2$ and $|\scZ_1|\le \eps $, and define $\P^\eps = \P[\cdot \,|\, E^\eps]$. It is well known that $\P^\eps$ converges weakly to  $\P^0$ which is supported on the space of continuous curves in the first quadrant which start and end at the origin. 
	
	Let $\cZ^n$ be the normalized lattice walk as defined at the beginning of Section~\ref{sec:triple}, \nocolor{which has i.i.d.\ increments.} 
	Let $\P^n$ be the law of $\cZ^n|_{[0,1]}$. For \nocolor{$\eps> 0$}, let $E^{\eps,n}$ be the event that $\cZ^n|_{[0,1]}\in [-\eps,\infty)^2$ and $|\cZ^n_1|\le \eps $, and let $\P^{n,\eps}=\P^n[\cdot \,|\, E^{\eps,n}]$.
	\nocolor{Let $\P^{n,0}$  be the conditional law of $\cZ^n|_{[0,1]}$ given that $\cZ^n|_{[0,1]}$ remains in  $\{(x,y)\in \R^2: x\ge 0; y\ge -\frac{1}{\sqrt{2n}}  \}$ and starts and ends at the origin. Then $\P^{n,0}$ is the law of $\left(\tfrac{1}{\sqrt{4n}}\cL_{\lfloor2nt\rfloor}^{\mathrm b,n},
		\tfrac{1}{\sqrt{2n}}\cR_{\lfloor2nt\rfloor}^{\mathrm b,n}\right)$ in the proof of Proposition~\ref{prop:dual_finite}. In particular,} $\P^{n,0}$ weakly converges to $\P^{0}$.  The following proposition allows us to transfer weak convergence under $\P^{n}$ to the same result under $\P^{n,0}$.
	\begin{proposition}\label{prop:transfer}
		Fix $0<\xi<u<1$. Let $\mathcal F_{\xi,u}$ be the sub-$\sigma$-algebra of the Borel $\sigma$-algebra of $\cC([0,1], \R^2)$ generated by the evaluation functionals at $x$, for all $x\in [\xi,u]$. Suppose $\{Y_n\}_{n\ge 1}$ and $Y$ are $\cF_{\xi,u}$-measurable random variables. 
		If the $\P^n$-law of $(\cZ^n,Y_n)$ weakly converges to the $\P$-law of $(\scZ,Y)$, then the $\P^{n,0}$-law of $(\cZ^n,Y_n)$  weakly converges to the $\P^0$-law of $(\scZ,Y)$.
	\end{proposition}
	
	We will prove Proposition~\ref{prop:transfer} in Appendix~\ref{sec:tran}. Now we proceed to the proof of Theorem~\ref{thm:3pair} and~\ref{thm:embedding}, which relies on Proposition~\ref{prop:transfer} and several arguments similar to the ones in Sections~\ref{subsec:1D}--\ref{sec:3tree}. So we will only highlight the new technical ingredients.
	
	We still use the same notations as in the infinite volume setting. Namely, let $S$ be a uniform wooded triangulation of size $n$ and $Z^\b,Z^\r,Z^\g$ be its encoding walks for counterclockwise  explorations. Let $\cZ=(\cL,\cR)$ be the grouped-step walk of $\cZ^\b$ \nocolor{as defined in Lemma~\ref{lem:fin-walk}.}  Let $Z^{\b,n},Z^{\r,n},Z^{\g,n}$ be the rescaled walks in Theorem~\ref{thm:3pair} and $\cZ^{n}=(\tfrac{1}{\sqrt{4n}}\cL^n_{2nt},  \tfrac{1}{\sqrt{2n}}\cR^n_{2nt})$.  
	\nocolor{
	\begin{remark}[Abuse of notations in the finite-volume setting] \label{rmk:notation-finite}
	In Section~\ref{sec:pair}  we used $S^n$ to denote a uniform wooded triangulation of size $n$, $Z^{\b,n},Z^{\r,n},Z^{\g,n}$ to denote the corresponding unscaled random walks, and $\cZ^{\mathrm b,n}$ as the grouped-step walk of $Z^{\b,n}$. In this section we denote them by $S$, $Z^\b$, $Z^\r$, $Z^\g$, and $\cZ$ for simplicity. On the other hand, we use  $Z^{\b,n},Z^{\r,n},Z^{\g,n}$ and $\cZ^{n}$ to denote the rescaled walks as in the infinite volume case; see Remark~\ref{rmk:notation}. Similarly, most of our notations below are in parallel to the infinite volume case. This is convenient because  our finite-volume result is derived from the infinite volume case. 
	\end{remark}}
	
	By Lemma~\ref{lem:fin-walk}, the  law of $\cZ^n$ is  $\P^{n,0}$ above.  
	\begin{figure}
		\centering
		\includegraphics[width=6cm]{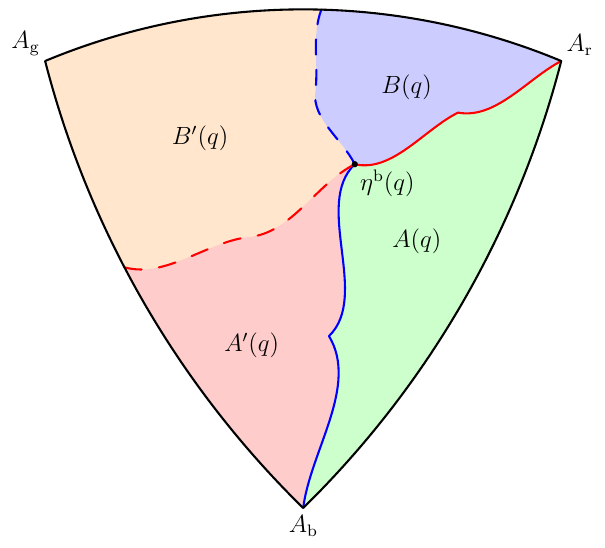}
		\caption{The red and blue flow lines and dual flow lines from $\eta^{\b}(q)$ divide the unit area $1$-LQG sphere into four regions whose areas are $A(q)$, $B(q)$, $A'(q)$, $B'(q)$ as indicated. In the continuum, the outer face $\Delta \Ab\Ar\Ag$ should be viewed as $\infty$.}
		\label{fig:finite_lines}
	\end{figure}
	For $q\in (0,1)\cap \Q$, we can define left and right excursion for $\cZ^n|_{[q,1]}-\cZ^n_q$ as in Section~\ref{subsec:1D} and~\ref{sec:triple}. Let $\cA^n(q)$ and $\cB^n(q)$ be the total time in the right and left excursions during $[q,1]$ respectively.  In the continuum, let $\mu_{\fh}, \eta^\b,\eta^\r,\eta^\g$ and $\scZ^\b,\scZ^\r,\scZ^\g$ be defined as in Theorem~\ref{thm:3pair}.  Note that $\mu_{\fh}$  has unit mass. Let $\eta^{\theta}_q$ be the flow line from $\eta^\b(q)$ of angle $\theta\colonequals \tfrac{\pi}{6}$. Let $A(q)$ and $B(q)$ be the $\mu_{\fh}$-area in $\eta^\b[q,1]$ on the right and left  of $\eta^{\theta}_q$ (see Figure~\ref{fig:finite_lines}). We claim that
	\begin{equation}\label{eq:AB}
	(\cZ^n, \cA^n(q), \cB^n(q) ) \to (\scZ^\b, A(q), B(q) ) \quad \textrm{in law}.
	\end{equation}
	For $\eps>0$ small, let $\cA^{n,\eps}(q)$ and $\cB^{n,\eps}(q)$ be the total time in the right and left excursions during $[q,1-\eps]$ respectively. And let $A^\eps(q)$ and $B^\eps(q)$ be the $\mu_{\fh}$-area in $\eta^\b[q,1-\eps]$ on the right and left  of $\eta^{\theta}_q$.  The by Proposition~\ref{prop:triple} and~\ref{prop:transfer}, $(\cZ^n, \cA^\eps(q), \cB^\eps(q) ) \to (\scZ^\b, A^\eps(q), B^\eps(q) )$  in law. By the tightness of $(\cZ^n, \cA^{n}(q), \cB^{n}(q))$, by possibly extracting a subsequence, we may find a coupling so that a.s. $(\cA^{n}(q), \cB^{n}(q))$ converges to some $(A,B)$ and $(\cA^{n,\eps}(q), \cB^{n,\eps}(q)) \to (A^\eps(q),B^\eps(q))$ for all  rational $\eps$.
	Since $\cA^{n,\eps}(q)\le \cA^n(q)$ and $\cB^{n,\eps}(q)\le \cB^n(q)$, we have $A^\eps(q)\le A$ and $B^\eps\le B$ a.s. Letting $\eps\to 0$, we have $A(q)\le A$ and $B(q)\le B$ a.s. On the other hand, $A+B=A(q)+B(q)=1-q$. Therefore $A(q)=A$ and $B(q)=B$ a.s. This proves \eqref{eq:AB}.
	
	Let $\fraketa^{\b,n}$ be the rescaled clockwise blue tree exploration path of the Schnyder wood as in the infinite volume setting.  Let $A^n(q)$ and $B^n(q)$ be the discrete quantum area of $\eta^\b[q,1]$ on the right and left  of the red flow line from $\fraketa^{\b,n}(q)$. We claim that due to \eqref{eq:AB} we have $(Z^{\b,n}, A^n(q), B^n(q) ) \to (\scZ^\b, A(q), B(q))$ in law. There are two caveats for this claim: (1) the time $q$ for $\eta^{\b,n}$ (i.e., for $Z^{\b,n}$) is not exactly the same as the time $q$ for $\cZ$; (2) the boundary contributions of the two regions $A^n(q)$ and $B^q(n)$ are nonzero. These two concerns are handled by Remark~\ref{rmk:wheretostart} and~\ref{rmk:boundary}, respectively. Note that the $o_n(1)$ errors there have rapidly decaying tails, in the contrast to the polynomial decay of the conditioning event in Proposition~\ref{prop:sampler}. Therefore the two remarks also apply to the finite volume setting, and the claim holds. 
	
	For $q\in (0,1)$, we can reverse $\fraketa^{\b,n}$ to define the left and right excursions of the dual red flow line from time $q$ to time $0$. Let $A'^n(q), B'^n(q)$ be the discrete quantum area of on the left and right of that dual flow line and $A'(q), B'(q)$ be their continuum analog (see Figure~\ref{fig:finite_lines}). Then similar to the forward case, we have $(Z^{\b,n}, A'^n(q), B'^n(q) ) \to (\scZ^\b, A'(q), B'(q))$ in law.  Let $\fraketa^{\r,n}$ be the clockwise red tree exploration and $t^n_q=:\inf\{t:\fraketa^{\r,n}(t)=\fraketa^{\b,n}(q) \}$. Then $t^n_q=A^n(q)+A'^n(q)+o_n(1)$ where $o_n(1)$ comes from  Remark~\ref{rmk:boundary}. Let $t_q$ be the a.s.\ unique time such that $\eta^\b(q)=\eta^\r(t_q)$. Then  
	\begin{equation}\label{eq:tq}
	(Z^{\b,n}, t^n_q)\to (\scZ^\b,t_q)\quad \textrm{in law}\quad \forall \; q\in (0,1) \cap \Q.
	\end{equation}

	Consider a coupling such that $Z^{\b,n} \to \scZ^\b$ and $ t^n_q\to t_q$ a.s. for all $q\in (0,1)\cap \Q$.  Let $\mathfrak l^n(q)$ be the local time accumulated by $\cZ^n|_{[q,1]}-\cZ^n_q$ during $[q,1]$ defined as in \eqref{eq:local}.   By possibly extracting a subsequence, we may further assume $Z^{\r,n}$ converges a.s.\ to some $\scZ=(\scL,\scR)$ and $\lim_{n\to \infty}\mathfrak l^n(q)$ exists a.s.  Write $Z^{\r,n}=(L^{\r,n},R^{\r,n})$. By Proposition~\ref{prop:paren_matching}, $R^{\r,n}_{t^n_q}$ is  the total discrete quantum length of the red flow from $\fraketa^{\b,n}(q)$.  By a truncation argument similar to the one used above to show $A(q)\le A$, combined with Proposition~\ref{prop:W}, we have 
	\begin{equation}\label{eq:localbound}
	\lim_{n\to\infty}\mathfrak l^n(q)\le (1+\sqrt 2) \scR_{t_q}.
	\end{equation} 
	Write $\scZ^\r$ in Theorem~\ref{thm:3pair} as $(\scL^\r,\scR^\r)$. Then \xin{the} quantum length of $\eta^\theta\cap \eta^\b[q,1]$ equals $\scR^{r}_{t_q}$.  Remark~\ref{rmk:concen} and \eqref{eq:concen} in Lemma~\ref{lem:mc_concentration} ensures  that Lemma~\ref{lem:DQL} can be applied to the finite volume setting. Combined with~\eqref{eq:localbound}, we have  $\scR^\r_{t_q}\le \scR_{t_q}$.
	
	Since $\{t_q\}_{q\in (0,1)\cap\Q}$ is dense in $[0,1]$ as in the proof of Theorem~\ref{thm:main1}, $\scR^\r_{t}\le \scR_{t}$ a.s. for all $t\in [0,1]$. Since $\scR\overset{d}{=}\scR^r$, we must have $\scR=\scR^\r$ a.s. The same argument gives $\scL=\scL^\r$.  Thus  $(Z^{\b,n},Z^{\r,n})$ converge to $(\scZ^\b,\scZ^r)$ in law.  By symmetry in the role of the colors, we conclude the proof of Theorem~\ref{thm:3pair}. 
	
	To prove Theorem~\ref{thm:embedding}, note that $(Z^{\b,n},A^n(t))\to (\scZ^\b, A(t))$ in law for all $t\in[0,1]$. Therefore the discrete quantum area of the region bounded by red and blue flow lines from a uniformly chosen vertex converges to the corresponding continuum quantum area.   The ratio between the number of faces and edges in this region  is $2:3$,  modulo the negligible boundary contribution (Remark~\ref{rmk:boundary}). This gives the  convergence of the $z$-coordinate (Figure~\ref{fig:schnyderembedding}.) in Theorem~\ref{thm:embedding}. Since    $(Z^{\b,n}, Z^{\r,n}, Z^{\g,n})$ jointly converge as in Theorem~\ref{thm:3pair}, the three coordinates jointly converge as stated in Theorem~\ref{thm:embedding}.

	\appendix 
	\section{Proof of Proposition~\ref{prop:transfer}}
	\label{sec:tran}
	Suppose we are in the setting of Proposition~\ref{prop:transfer}. For a probability measure $\Q$ on $\cC([0,1], \R^2)$ under the uniform topology, let $\Q_{\xi,u}$ be $\Q$ restricted to $\cF_{\xi,u}$. Since $\cZ^n_1$ under $\P^n$ converges to $\scZ_1$ under $\P$, we have that  $Y_n$ under $\P^{n,\eps}$ weakly converges to $Y$ under $\P^{\eps}$. Therefore,  Proposition~\ref{prop:transfer} follows from Lemma~\ref{lem:dtv1} and~\ref{lem:dtv2} below.
	\begin{lemma}\label{lem:dtv1}
		\begin{equation}\label{eq:total1}
		d_{\mathrm{tv}} (\P^{\eps}_{\xi,u}, \P^{0}_{\xi,u} )=o_\eps(1), 
		\end{equation}where $d_{\mathrm{tv}}$ denotes the total variational distance.
	\end{lemma}
	\begin{lemma}\label{lem:dtv2}
		\begin{equation}\label{eq:total2}
		d_{\mathrm{tv}} (\P^{n,\eps}_{\xi,u}, \P^{n,0}_{\xi,u} )=o_\eps(1), 
		\end{equation}
		where $o_\eps(1)$ is uniform in $n$.
	\end{lemma}
	\begin{proof}[Proof of Lemma~\ref{lem:dtv1}]
		Denote by $\scZ^\eps$ a sample from $\P^\eps$, for all $\eps\ge 0$.   We want to find a coupling $\Q^{\eps}$ of $\scZ^\eps$ and $\scZ^0$  such that
		\begin{equation}\label{eq:coupling1}
		\Q^\eps[\scZ^\eps_t=\scZ^0_t\;\;\textrm{for all}\; \xi\le t\le u ]=1-o_\eps(1).
		\end{equation}
		Let $\xi_0=\xi/2$ and $u_0=(u+1)/2$.  Let $G^\eps= \{\scZ^\eps|_{[\xi_0,u_0]}\in [-\eps,\infty)^2\}$ for all $\eps\ge 0$.
		Let $\wt \scZ^\eps$ and $\wt \scZ^0$ be $\scZ$ conditioned on $E^\eps\setminus G^\eps$ and $E^0\setminus G^0$ respectively.  
		We start by showing  that there exists a coupling $\wt \Q^{\eps}$ of $\wt \scZ^\eps$ and $\wt \scZ^0$  where
		\begin{equation}\label{eq:coupling2}
		\wt \Q^\eps[\wt \scZ^\eps_t=\wt \scZ^0_t\;\;\textrm{for all}\; \xi\le t\le u ]=1-o_\eps(1).
		\end{equation}
		We first notice that 
		\begin{equation}\label{eq:end}
		\lim_{\eps\to 0}(\wt \scZ^{\eps}_{\xi_0}, \wt \scZ^{\eps}_{u_0})=(\wt \scZ^{0}_{\xi_0}, \wt \scZ^{0}_{u_0}) \quad\textrm{in law}.
		\end{equation}
		By Skorohod's embedding theorem, we can find a coupling $\wt \Q$ of $\{\wt \scZ^{\eps}\}_{\eps\ge0}$ 
		such that for each fixed $\delta_0>0$,
		\[
		\wt \Q[|\wt \scZ^{\eps}_{\xi_0}-\wt \scZ^{0}_{\xi_0}|>\delta_0]=o_\eps(1)\quad \textrm{and}\quad \wt \Q[|\wt \scZ^{\eps}_{u_0}-\wt \scZ^{0}_{u_0}|>\delta_0]=o_\eps(1)
		\] where $o_\eps(1)$ depends only on $\delta_0,\xi,$ and $u$. For all $\eps>0$, given a sample of $(\wt \scZ^\eps,\wt \scZ^0)$ under $\wt \Q$, we claim that by resampling  $\wt \scZ^{\eps}|_{[\xi_0,u_0]}$ and $\wt \scZ^0|_{[\xi_0,u_0]}$ conditioned on $\wt \scZ^{\eps}_{[0,\xi_0]},$ $\wt \scZ^{\eps}_{[u_0,1]},$ $\wt \scZ^{0}_{[0,\xi_0]},$ $\wt \scZ^0_{[u_0,1]}$, we can obtain a coupling satisfying \eqref{eq:coupling2}.
		
		In fact, the $[\xi_0,u_0]$ segment of $\wt \scZ^\eps$ and $\wt \scZ^0$ are both planar Brownian bridges conditioning on its end points. Therefore, by the explicit Brownian heat kernel,  we can couple  $(\wt \scZ^\eps, \wt \scZ^0)$  so that 
		\begin{equation}\label{eq:coupling3}
		\wt \Q^\eps\left[\wt \scZ^\eps_{\xi_0}=\wt \scZ^0_{\xi_0}\;\;\textrm{and}\; \wt \scZ^\eps_{u_0}=\wt \scZ^0_{u_0}\right]=1-o_\eps(1).
		\end{equation}
		Now we can resample the $[\xi_0,u_0]$ segment of $\wt \scZ^\eps$ and $\wt \scZ^0$ using the same Brownian bridge to achieve \eqref{eq:coupling2}.
		
		For all $\eps\ge 0$, the process $\scZ^\eps$ is $\wt \scZ^\eps$ conditioning on $G^\eps$. Since $\wt \Q^\eps[G^\eps\Delta G^0]=o_\eps(1)$ and $\wt \Q^\eps[G^0] \asymp 1$, by a further rejection sampling, we can find a coupling $\Q^\eps$ satisfying \eqref{eq:coupling1}.
	\end{proof}

	\begin{proof}[Proof of Lemma~\ref{lem:dtv2}]
		The proof goes exactly in the same way as Lemma~\ref{lem:dtv1}. So we only point out the necessary modification. Recall the setting of Lemma~\ref{lem:dtv1}.  We use the notations \nocolor{$\cZ^{n,\eps}$, $\cZ^{n,0}$, $G^{n,\eps}$, $G^{n,0}$, $\Q^{n,\eps}$, $\wt \Q^{n,\eps}$, $\wt \Q^{n}$, $\wt \cZ^{n,\eps}$, $\wt  \cZ^{n,0}$} to denote the discrete analogues of  $\scZ^\eps,\scZ^0,G^{\eps}, G^{0},\Q^{\eps}, \wt \Q^{\eps}, \wt \Q, \wt \scZ^{\eps}, \wt \scZ^{0}$.
		
		We can first find a coupling $\wt \Q^n$ of $\wt \cZ^{n,\eps}$ and  $\wt \cZ^{n,0}$ such that for each fixed $\delta_0>0$
		\[
		\Q^n\left[|\cZ^{n,\eps}_{\xi_0}-\cZ^{n,0}_{\xi_0}|>\delta_0\right]=o_\eps(1)\quad \textrm{and}\quad \Q^n\left[|\cZ^{n,\eps}_{u_0}-\cZ^{n,0}_{u_0}|>\delta_0\right]=o_\eps(1)
		\] where the quantity denoted $o_\eps(1)$ may depend on $\delta_0,\xi,u$, but not on $n$.
		In fact, by the invariance principle of random walk in cones  \cite{duraj2015invariance},
		\begin{align}\label{eq:end2}
		\lim_{n\to\infty}(\wt \cZ^{n,\eps}_{\xi_0},\wt \cZ^{n,\eps}_{u_0})=(\wt Z^{\eps}_{\xi_0},\wt Z^{\eps}_{u_0})\;\;\textrm{and}\;\;
		\lim_{n\to\infty}(\wt \cZ^{n,0}_{\xi_0},\wt \cZ^{n,0}_{u_0})=(\wt Z^{0}_{\xi_0}, \wt Z^{0}_{u_0}) \;\; \textrm{in law}.
		\end{align}
		Now the existence of $\wt \Q^n$ is equivalent to the definition of weak convergence in terms of L\'evy-Prokhorov metric.
		
		We can perform a resampling procedure analogous to the one in Lemma~\ref{lem:dtv1} to get $\wt \Q^{n,\eps}$ such that 
		\begin{equation}\label{eq:coupling5}
		\wt \Q^{n,\eps}\left[\wt \cZ^{n,\eps}_t=\wt \cZ^{n,0}_t\;\;\textrm{for all}\; \xi\le t\le u \right]=1-o_\eps(1).
		\end{equation}
		The only difference is that the $[\xi_0,u_0]$ segment of $\wt \cZ^\eps$ and $\wt \cZ^0$ are planar random walk bridges rather than Brownian bridges. Therefore we need to use a random walk heat kernel estimate (in place of the Brownian motion heat kernel) to couple  $(\wt \cZ^{n,\eps}, \wt \cZ^{n,0})$  so that 
		\begin{equation}\label{eq:coupling4}
		\wt \Q^{n,\eps}[\wt \cZ^{n,\eps}_{\xi_0}=\wt \cZ^{n,0}_{\xi_0}\;\;\textrm{and}\; \wt \cZ^{n,\eps}_{u_0}=\wt \cZ^{n,0}_{u_0}]=1-o_\eps(1).
		\end{equation}
		Since the increment $\xi$ of our walk in Lemma~\ref{lem:fin-walk}  satisfies $E[e^{b|\xi|}]<\infty$ for some $b>0$ and has zero mean, the following local central limit theorem (as an immediate     corollary of \cite[Theorem 2.3.11]{Lawler-Walk}) can be  applied to find the coupling in \eqref{eq:coupling4}:
		
		Let $p_n$ be the heat kernel of the random walk $\P^\infty$ in Lemma~\ref{lem:fin-walk} and $\bar p_n$ be the heat kernel of the Brownian motion. When we estimate $p_n$ using CLT, then 
		\begin{equation}\label{eq:LCLT}
		(p_n(x)+p_{n+1}(x)) /2= \bar p_n(x)\exp\left\{O\left(\frac{1}{\sqrt n}\right)+ O\left(\frac{|x|}{n}\right )\right\},\;\; \quad \forall \;|x| \le \sqrt n\log n. 
		\end{equation}
		The appearance of $(p_n(x)+p_{n+1}(x)) /2$ is due to the parity issue, namely that $p_n(0,0)=0$ for odd $n$. However, \eqref{eq:LCLT} implies the desired result in our case, since 
		$\wt \cZ^{n,\eps}_{\xi_0}$ and $\wt \cZ^{n,0}_{\xi_0}$ (as well as $\wt \cZ^{n,\eps}_{u_0}$ and $\wt \cZ^{n,0}_{u_0}$) are always in the same parity class, as they are two possible locations of a random walk at the same time. 
		
		The rest of the argument is the same  as in the proof of Lemma~\ref{lem:dtv1}.
	\end{proof}
	
	\begin{remark}\label{rmk:general}
		Our proof of Proposition~\ref{prop:transfer} holds for all lattice walks with zero mean and some finite exponential moment. This is useful when one studies the scaling limit of the mating of trees encoding of other decorated random planar maps, such as the spanning-tree-decorated maps and bipolar oriented maps \cite{burger,bipolar}.
		Moreover, in light of \cite[Lemma 2.4.3]{Lawler-Walk}, we expect that the finite exponential moment assumption can be significantly weakened. However, we won't pursue this here.
	\end{remark}

	\textbf{Typesetting note}. The figures in this article were
	produced using \href{http://asymptote.sourceforge.net}{Asymptote},
	\href{http://sagemath.org}{SageMath}, \href{https://sourceforge.net/projects/pgf/}{Ti\textit{k}Z}, and \href{http://ipe.otfried.org}{Ipe}. 
	
	\medbreak {\noindent\bf Acknowledgment:} We thank Olivier Bernardi, Ewain Gwynne, Nina Holden, Richard Kenyon and Scott Sheffield	for  helpful discussions.  We are also grateful \xin{to} the anonymous referees for their thorough and helpful feedback on an early version of the paper.


\bibliographystyle{amsplain}

\end{document}